\documentclass[leqno,11pt]{article}

\usepackage[utf8]{inputenc}
\usepackage[T1]{fontenc}
\usepackage{microtype}

\usepackage[dvipsnames]{xcolor}
\usepackage{xcolor-solarized}

\usepackage{calc}
\usepackage[a4paper]{geometry}

\ifdefined\screenLayout
  \pagecolor{solarized-base3}
  \color{solarized-base03}
  \usepackage{fancyhdr}
\fi

\usepackage{amsmath}
\usepackage{amsthm}

\usepackage{tikz-cd}
\usetikzlibrary{arrows} 
\tikzset{
  commutative diagrams/.cd, 
  arrow style=tikz, 
  diagrams={>=stealth}
}
\usetikzlibrary{matrix,decorations.pathreplacing,calc}
\usetikzlibrary{graphs,graphs.standard}
\usepackage{pgfplots}
\usepgfplotslibrary{external}

\usepackage{afterpage}
\usepackage{pdflscape}

\usepackage{colortbl}
\usepackage{adforn}
\usepackage{nameref}

\usepackage{textcomp}
\usepackage[sb]{libertine}
\usepackage[varqu,varl]{zi4}%
\usepackage[libertine,bigdelims,vvarbb]{newtxmath}

\IfPackageAtLeastTF{superiors}{2.0}{%
  \usepackage[supsfam=LibertinusSerif-Sup,supscaled=1.2,raised=-.13em]{superiors}
}{%
  \usepackage[supstfm=libertinesups,supscaled=1.2,raised=-.13em]{superiors}
}

\useosf
\usepackage[scr=boondox,cal=euler]{mathalfa}
\usepackage{marvosym}

\usepackage{slashed}
\usepackage{esint} 

\usepackage[ngerman,english]{babel}
\usepackage{imakeidx}
\makeindex[intoc]
\usepackage{csquotes}
\usepackage[
  backend=biber,
  hyperref=true,
  backref=true,
  isbn=false,
  doi=true,
  natbib=true,
  eprint=true,
  useprefix=true,
  maxcitenames=99,
  maxbibnames=99,  
  maxalphanames=99, 
  minalphanames=99,
  safeinputenc,
  style=alphabetic,
  citestyle=alphabetic,
  block=space,
  datamodel=preamble/ext-eprint,
  sorting=nyt
]{biblatex}
\usepackage[
  bookmarksnumbered = true,
  hypertexnames = false,
  colorlinks    = true,
  citecolor     = solarized-blue,
  linkcolor     = solarized-blue,
  urlcolor      = solarized-blue,
  breaklinks
]{hyperref}

\DeclareFieldFormat{url}{%
  \href{#1}{\ComputerMouse}
}
\DeclareFieldFormat{doi}{%
  \mkbibacro{DOI}\addcolon\space\href{https://doi.org/#1}{#1}
}
\makeatletter
\DeclareFieldFormat{arxiv}{%
  arXiv\addcolon\space\href{http://arxiv.org/\abx@arxivpath/#1}{#1}
}
\makeatother
\DeclareFieldFormat{mr}{%
  MR\addcolon\space\href{http://www.ams.org/mathscinet-getitem?mr=MR#1}{#1}
}
\DeclareFieldFormat{zbl}{%
  Zbl\addcolon\space\href{http://zbmath.org/?q=an:#1}{#1}
}
\renewbibmacro*{eprint}{%
  \printfield{arxiv}%
  \newunit\newblock
  \printfield{mr}%
  \newunit\newblock
  \printfield{zbl}%
  \newunit\newblock
  \iffieldundef{eprinttype}
  {\printfield{eprint}}
  {\printfield[eprint:\strfield{eprinttype}]{eprint}}
}

\AtEveryBibitem{%
  \clearlist{address}%
}
\DeclareFieldFormat[article,inproceedings,inbook,incollection,thesis]{title}{\textit{#1}}
\renewbibmacro{in:}{}
\newcommand{\printreferences}{\raggedright\printbibliography}

\usepackage[inline,shortlabels]{enumitem}

\usepackage{subcaption}

\usepackage[yyyymmdd]{datetime}

\usepackage{etoolbox}
\ifundef{\abstract}{}{\patchcmd{\abstract}%
    {\quotation}{\quotation\noindent\ignorespaces}{}{}}

\usepackage[super]{nth}

\usepackage{thmtools}
\usepackage[framemethod=TikZ]{mdframed}

\numberwithin{equation}{section}

\renewcommand{\qedsymbol}{$\blacksquare$}

\newcommand{\CorollaryQED}{\qedsymbol}
\newcommand{\ConjectureQED}{$\square$}
\newcommand{\SituationQED}{$\times$}
\newcommand{\DefinitionQED}{$\bullet$}
\newcommand{\NotationQED}{$\circ$}
\newcommand{\ExampleQED}{$\spadesuit$}
\newcommand{\RemarkQED}{$\clubsuit$}
\newcommand{\ExerciseQED}{?!}

\ifdefined\screenLayout
  \declaretheoremstyle[
  bodyfont=\itshape,
  mdframed={    
    backgroundcolor=solarized-base3!90!solarized-blue,
    linewidth=0,
    innerleftmargin=.5em,
    innerrightmargin=.5em,
    innertopmargin=.5em,
    innerbottommargin=.5em,
    leftmargin=-.5em,
    rightmargin=-.5em,
  }
]{theorem}

\declaretheoremstyle[
mdframed={
  backgroundcolor=solarized-base3!90!solarized-green,
  linewidth=0,
  innerleftmargin=.5em,
  innerrightmargin=.5em,
  innertopmargin=.5em,
  innerbottommargin=.5em,
  leftmargin=-.5em,
  rightmargin=-.5em,
  }
]{definition}

\declaretheoremstyle[
  mdframed={
    backgroundcolor=solarized-base3!90!solarized-yellow,
    linewidth=0,
    innerleftmargin=.5em,
    innerrightmargin=.5em,
    innertopmargin=.5em,
    innerbottommargin=.5em,
    leftmargin=-.5em,
    rightmargin=-.5em,
  }
]{example}

\declaretheoremstyle[
  mdframed={
    backgroundcolor=solarized-base3!90!solarized-orange,
    linewidth=0,
    innerleftmargin=.5em,
    innerrightmargin=.5em,
    innertopmargin=.5em,
    innerbottommargin=.5em,
    leftmargin=-.5em,
    rightmargin=-.5em,
  }
]{remark}
\else
  \declaretheoremstyle[
  bodyfont=\itshape
  ]{theorem}
  \declaretheoremstyle[]{definition}
  \declaretheoremstyle[]{example}
  \declaretheoremstyle[]{remark}
\fi

\declaretheorem[numberlike=equation,style=theorem]{theorem}
\declaretheorem[numbered=no,name=Theorem,style=theorem]{theorem*}
\declaretheorem[numberlike=equation,name=Lemma,style=theorem]{lemma}
\declaretheorem[numberlike=equation,name=Proposition,style=theorem]{prop}
\declaretheorem[numberlike=equation,name=Corollary,qed=\CorollaryQED,style=theorem]{cor}
\declaretheorem[numberlike=equation,name=Conjecture,qed=\ConjectureQED,style=theorem]{conjecture}

\declaretheorem[numberlike=equation,name=Hypothesis]{hypothesis}

\declaretheorem[numberlike=equation,name=Definition,style=definition,qed=\DefinitionQED]{definition}
\declaretheorem[numbered=no,name=Definition,style=definition,qed=\DefinitionQED]{definition*}

\declaretheorem[numberlike=equation,style=definition,qed=\ExampleQED]{example}

\declaretheorem[numberlike=equation,style=remark,qed=\RemarkQED]{remark}
\declaretheorem[numbered=no,style=remark,name=Remark,qed=\RemarkQED]{remark*}

\declaretheorem[numberlike=equation,style=remark]{convention}
\declaretheorem[numberlike=equation,style=definition]{question}

\def\makeautorefname#1#2{\AtBeginDocument{\expandafter\def\csname#1autorefname\endcsname{#2}}}
\makeautorefname{table}{Table}        
\makeautorefname{chapter}{Chapter}
\makeautorefname{section}{Section}
\makeautorefname{subsection}{Section}
\makeautorefname{subsubsection}{Section}
\makeautorefname{footnote}{Footnote}
\AtBeginDocument{\def\itemautorefname~#1\null{(#1)\null}}
\AtBeginDocument{\def\equationautorefname~#1\null{(#1)\null}}

\newtheorem{step}{Step}

\numberwithin{substep}{step}
\makeautorefname{step}{Step}
\makeautorefname{substep}{Step}

\makeautorefname{case}{Case}
\makeautorefname{substep}{Step}

\setlist[description]{leftmargin=!,labelindent=1em}
\setlist[enumerate]{label={\rm (\arabic*)},ref=\arabic*}
\setlist[enumerate,2]{label={\rm (\alph*)},ref=\theenumi.\alph*}
\setlist[enumerate,3]{label={\rm (\roman*)},ref=\theenumii.\roman*}

\let\C\undefined

\usepackage{bm}
\usepackage{mathtools} 
\usepackage{stmaryrd} 

\DeclareFontFamily{U}{mathx}{\hyphenchar\font45}
\DeclareFontShape{U}{mathx}{m}{n}{
      <5> <6> <7> <8> <9> <10>
      <10.95> <12> <14.4> <17.28> <20.74> <24.88>
      mathx10
      }{}
\DeclareSymbolFont{mathx}{U}{mathx}{m}{n}
\DeclareFontSubstitution{U}{mathx}{m}{n}
\DeclareMathAccent{\widecheck}{0}{mathx}{"71}
\DeclareMathAccent{\wideparen}{0}{mathx}{"75}

\DeclareMathOperator{\Aut}{Aut}
\DeclareMathOperator{\aut}{\mathfrak{aut}}

\DeclareMathOperator{\Diff}{Diff}

\DeclareMathOperator{\Ext}{Ext}

\DeclareMathOperator{\HF}{\HF}

\DeclareMathOperator{\Map}{Map}

\DeclareMathOperator{\PD}{PD}

\DeclareMathOperator{\Res}{Res}
\DeclareMathOperator{\Spec}{Spec}

\DeclareMathOperator{\codim}{codim}

\DeclareMathOperator{\coker}{coker}

\DeclareMathOperator{\im}{im}
\DeclareMathOperator{\ind}{index}
\DeclareMathOperator{\inj}{inj}

\DeclareMathOperator{\sign}{sign}

\DeclarePairedDelimiter\paren{\lparen}{\rparen}
\DeclarePairedDelimiter\sqparen{[}{]}
\DeclarePairedDelimiter{\Abs}{\|}{\|}

\DeclarePairedDelimiter{\Inner}{\langle}{\rangle}

\DeclarePairedDelimiter{\abs}{\lvert}{\rvert}
\DeclarePairedDelimiter{\bracket}{\langle}{\rangle}

\DeclarePairedDelimiter{\set}{\lbrace}{\rbrace}
\def\({\left(}
\def\){\right)}
\def\<{\left\langle}
\def\>{\right\rangle}

\newcommand{\CP}{{\C P}}

\newcommand{\C}{{\mathbf{C}}}

\newcommand{\N}{{\mathbf{N}}}

\newcommand{\R}{\mathbf{R}}

\newcommand{\Vect}{\mathrm{Vect}}

\newcommand{\Z}{\mathbf{Z}}

\newcommand{\co}{\mskip0.5mu\colon\thinspace}

\newcommand{\cyl}{{\rm{cyl}}}

\newcommand{\defined}[2][\key]{\def\key{#2}\textbf{#2}\index{#1}}
\newcommand{\delbar}{\bar{\del}}

\newcommand{\del}{\partial}
\newcommand{\ev}{\mathrm{ev}}

\newcommand{\id}{\mathrm{id}}

\newcommand{\emb}{\hookrightarrow}

\newcommand{\inner}[2]{\bracket{#1, #2}}
\newcommand{\into}{\hookrightarrow}
\newcommand{\iso}{\cong}

\newcommand{\loc}{\mathrm{loc}}

\newcommand{\ob}{\mathrm{ob}}

\newcommand{\one}{\mathbf{1}}

\newcommand{\pr}{\mathrm{pr}}
\newcommand{\qandq}{\quad\text{and}\quad}

\newcommand{\qifq}{\quad\text{if}\quad}

\newcommand{\qand}{\quad\text{and}}

\newcommand{\qforeveryq}{\quad\text{for every}\quad}

\newcommand{\qwithq}{\quad\text{with}\quad}

\newcommand{\sExt}{\mathrm{\sE\!xt}}
\newcommand{\sHom}{\mathrm{\sH\!om}}

\newcommand{\vol}{\mathrm{vol}}

\renewcommand{\P}{\mathbf{P}}

\renewcommand{\emptyset}{\varnothing}
\renewcommand{\epsilon}{\varepsilon}
\renewcommand{\setminus}{{\backslash}}

\renewcommand{\leq}{\leqslant}
\renewcommand{\geq}{\geqslant}

\makeatletter
\renewcommand*\env@matrix[1][*\c@MaxMatrixCols c]{%
  \hskip -\arraycolsep
  \let\@ifnextchar\new@ifnextchar
  \array{#1}}

\renewcommand\xleftrightarrow[2][]{%
  \ext@arrow 9999{\longleftrightarrowfill@}{#1}{#2}}
\newcommand\longleftrightarrowfill@{%
  \arrowfill@\leftarrow\relbar\rightarrow}
\makeatother

\newcommand{\even}{\mathrm{even}}

\newcommand{\rd}{{\rm d}}

\newcommand{\rI}{{\rm I}}
\newcommand{\rII}{{\rm II}}
\newcommand{\rIII}{{\rm III}}
\newcommand{\rIV}{{\rm IV}}

\newcommand{\bx}{{\mathbf{x}}}

\newcommand{\bF}{{\mathbf{F}}}

\newcommand{\sE}{\mathscr{E}}
\newcommand{\sF}{\mathscr{F}}

\newcommand{\sH}{\mathscr{H}}
\newcommand{\sI}{\mathscr{I}}
\newcommand{\sJ}{\mathscr{J}}

\newcommand{\sM}{\mathscr{M}}

\newcommand{\sO}{\mathscr{O}}

\newcommand{\sS}{\mathscr{S}}
\newcommand{\sT}{\mathscr{T}}
\newcommand{\sU}{\mathscr{U}}

\newcommand{\sX}{\mathscr{X}}
\newcommand{\sY}{\mathscr{Y}}

\newcommand{\fd}{{\mathfrak d}}
\newcommand{\fe}{{\mathfrak e}}

\newcommand{\fn}{{\mathfrak n}}
\newcommand{\fo}{{\mathfrak o}}

\newcommand{\fq}{{\mathfrak q}}
\newcommand{\fr}{{\mathfrak r}}

\newcommand{\fF}{{\mathfrak F}}

\newcommand{\fO}{{\mathfrak O}}

\newcommand{\bxi}{{\bm\xi}}

\usepackage{halloweenmath}

\renewcommand{\emb}{\mathrm{emb}}
\newcommand{\push}{\mathrm{push}}
\newcommand{\pull}{\mathrm{pull}}
\newcommand{\fuse}{\mathrm{fuse}}
\newcommand{\diff}{\mathrm{diff}}
\newcommand{\ghost}{\mathrm{ghost}}
\newcommand{\basefree}{\bullet}
\newcommand{\bubble}{\mathrm{bubble}}

\DeclareMathOperator{\vdim}{vdim}
\DeclareMathOperator{\bdry}{bd}
\newcommand{\rext}{\mathrm{ext}}
\newcommand{\rint}{\mathrm{int}}
\newcommand{\GW}{\mathrm{GW}}
\newcommand{\BPS}{\mathrm{BPS}}
\newcommand{\fibprod}[2]{\,{}_{#1}\mkern-4mu\times_{#2}}

\author{
  Aleksander Doan
  \and
  Thomas Walpuski
}
\title{
  Counting embedded curves in symplectic $6$--manifolds
}
\date{2022-02-14}

\begin{document}

\maketitle

\begin{abstract}
  Based on computations of \citet{Pandharipande1999},  
  \citet{Zinger2011} proved that the Gopakumar--Vafa BPS invariants $\BPS_{A,g}(X,\omega)$ for primitive Calabi--Yau classes and arbitrary Fano classes $A$ on a symplectic $6$--manifold $(X,\omega)$ agree with the signed count $n_{A,g}(X,\omega)$ of embedded $J$--holomorphic curves representing $A$ and of genus $g$ for a generic almost complex structure $J$ compatible with $\omega$.
  \citeauthor{Zinger2011}'s proof of the invariance of $n_{A,g}(X,\omega)$ is indirect,
  as it relies on Gromov--Witten theory.
  In this article we give a direct proof of the invariance of $n_{A,g}(X,\omega)$.
  Furthermore, we prove that $n_{A,g}(X,\omega) = 0$ for $g \gg 1$,
  thus proving the Gopakumar--Vafa finiteness conjecture for primitive Calabi--Yau classes and arbitrary Fano classes.
\end{abstract}

\tableofcontents


\section{Introduction}

\emph{Are there invariants of symplectic manifolds which count \emph{embedded} pseudo-holomorphic curves?}
Such counts can fail to be invariants for two reasons:
\begin{enumerate*}[(a)]
\item
  \label{Eq_MultipleCovers}
  pseudo-holomorphic embeddings can degenerate to multiple covers, and
\item
  \label{Eq_NodalMaps}
  they can undergo bubbling and their domains can degenerate.
\end{enumerate*}
In the following we consider two situations in which both of these can be ruled out.

Let $(X,\omega)$ be a closed symplectic $6$--manifold equipped with an almost complex structure $J$ compatible with $\omega$.
Denote by $\sM_{A,g}^\star(X,J)$ the moduli space of simple $J$--holomorphic maps representing a homology class $A \in H_2(X,\Z)$ and of genus $g$.
For a generic choice of $J$ the moduli space $\sM_{A,g}^\star(X,J)$ is an oriented smooth manifold of dimension 
\begin{equation*}
  \dim \sM_{A,g}^\star(X,J)
  =
  2\inner{c_1(X,\omega)}{A}.
\end{equation*}
If $A$ is a \defined{Calabi--Yau class}, that is: $\Inner{c_1(X,\omega),A} = 0$,
then $\sM_{A,g}^\star(X,J)$ is a finite set of signed points and can be counted.
If $A$ is primitive in $H_2(X,\Z)$,
then multiple cover phenomena can be ruled out, and it will be proved that this count defines an invariant $n_{A,g}(X,\omega)$.
If $A$ is a \defined{Fano class},
that is: $\Inner{c_1(X,\omega),A} > 0$,
then $\sM_{A,g}^\star(X,J)$ can be cut-down to a finite set of signed points by imposing incidence conditions governed by suitable cohomology classes $\gamma,\ldots,\gamma_\Lambda \in H^\even(X,\Z)$.
In this case, multiple cover phenomena can be ruled out regardless of whether $A$ is primitive or not,
and it will be proved that counting the cut-down moduli space defines an invariant $n_{A,g}(X,\omega;\gamma_1,\ldots,\gamma_\Lambda)$.

These invariants are not new.
They were considered by 
\citet[Theorem 1.5 and footnote 11]{Zinger2011} who proved that they agree with Gopakumar and Vafa's BPS invariants.
The proof of the invariance of $n_{A,g}(X,\omega)$ and $n_{A,g}(X,\omega;\gamma_1,\ldots,\gamma_\Lambda)$ in \cite{Zinger2011} is indirect:
it relies on these numbers satisfying the Gopakumar--Vafa formula and the invariance of Gromov--Witten invariants.
The novelty in the present work is that we give a much simpler direct proof of invariance.
Furthermore, we prove that the invariants vanish for $g$ sufficiently large; thus establishing the Gopakumar--Vafa finiteness conjecture for primitive Calabi--Yau classes and arbitrary Fano classes.

\subsection{Ghost components}

The main technical result of this paper allows us to rule out,
in certain situations,
degenerations in which the limiting nodal pseudo-holomorphic map has a \defined{ghost component},
that is: a component on which it is constant.
The precise definitions used in the following statement are given in \autoref{Sec_ModuliSpaces} and \autoref{Sec_GromovCompactness}.

\begin{theorem}
  \label{Thm_LimitConstraints}
  Let $(X,g_\infty,J_\infty)$ be an almost Hermitian manifold and 
  let $(J_k)_{k\in\N}$ be a sequence of almost complex structure on $X$ converging to $J_\infty$ in the $C^1$ topology.
  If $\paren*{u_k\co (\Sigma_k,j_k) \to (X,J_k)}_{k\in\N}$ is a sequence of pseudo-holomorphic maps from smooth, closed Riemann surfaces which Gromov converges to the nodal $J_\infty$--holomorphic map $u_\infty\co (\Sigma_\infty,j_\infty,\nu_\infty) \to (X,J_\infty)$,
  then one of the following holds:
  \begin{enumerate}
  \item
    $(\Sigma_\infty,j_\infty,\nu_\infty)$ has no ghost components.
   \item
    $(\Sigma_\infty,j_\infty,\nu_\infty)$ has a ghost component $C$ with at least two non-ghost components attached to $C$.
   \item
     $(\Sigma_\infty,j_\infty,\nu_\infty)$ has a ghost component $C$ with a non-ghost component attached to $C$ at at least two nodes.
   \item
     $(\Sigma_\infty,j_\infty,\nu_\infty)$ has a ghost component $C$ with precisely one non-ghost component attached to $C$ at a single node $n \in C$; in that case,  $\rd_{\nu_\infty(n)}u_\infty = 0$, 
     that is: the corresponding node $\nu_\infty(n)$ in the non-ghost component is a critical point of $u_\infty$. 
  \end{enumerate}
\end{theorem}

\begin{remark}
  \citet[Theorem 1.2]{Zinger2009} has analyzed in detail when a nodal pseudo-holomorphic map whose domain has arithmetic genus one appears as a Gromov limit of pseudo-holomorphic maps with smooth domain.
  Jingchen Niu's PhD thesis \cite{Niu2016} extends \citeauthor{Zinger2009}'s analysis to genus two.
  Their results are based on analyzing the obstruction map of a Kuranishi model of a neighborhood of the limiting pseudo-holomorphic map.
  The proof of \autoref{Thm_LimitConstraints} in \autoref{Sec_SmoothingNodalMaps} uses similar methods.
  This idea goes back to \citet[Proposition 1.20]{Ionel1998} and \citet[Lemma 1]{Pandharipande1995}.
  Recently, a different proof of  a result similar to \autoref{Thm_LimitConstraints} has appeared in the work of \citet[Lemma 4.9]{Ekholm2019}.
\end{remark}

Given a symplectic manifold $(X,\omega)$ of dimension at least $6$,
denote by $\sJ(X,\omega)$ the set of almost complex structures $J$ compatible with $\omega$ and denote by $\sJ_\emb(X,\omega)$ the subset of those $J$ for which the following hold:
\begin{enumerate}[(a)]
\item
  there are no simple $J$--holomorphic maps of negative index,
\item
  every simple $J$--holomorphic map is an embedding, and
\item
  every two simple $J$–holomorphic maps of index zero either have disjoint images or are related by a reparametrization;
\end{enumerate}
see \autoref{Def_JEmb}.
The complement of $\sJ_\emb(X,\omega)$ in $\sJ(X,\omega)$ has codimension two;
in particular:
$\sJ_\emb(X,\omega)$ is open and dense, and
every path $(J_t)_{t\in[0,1]}$ in $\sJ(X,\omega)$ with end points in $\sJ_\emb(X,\omega)$ is homotopic relative to the end points to a path in $\sJ_\emb(X,\omega)$.

\begin{theorem}
  \label{Thm_CY3PrimitiveGenericLimitHasSmoothDomain}
  Let $(X,\omega)$ be a compact symplectic $6$--manifold,
  let $(J_k)_{k\in\N}$ be a sequence of almost complex structures compatible with $\omega$ converging to $J_\infty$, and
  let $\paren*{u_k\co (\Sigma_k,j_k) \to (X,J_k)}_{k\in\N}$ be a sequence of pseudo-holomorphic maps which Gromov converges to the nodal $J_\infty$--holomorphic map $u_\infty\co (\Sigma_\infty,j_\infty,\nu_\infty) \to (X,J_\infty)$.
  Set $A \coloneq (u_\infty)_*[\Sigma_\infty] \in H_2(X,\Z)$.
  If $A$ is primitive,
  satisfies $\Inner{c_1(X,\omega),A} = 0$,
  and $J_\infty \in \sJ_\emb(X,\omega)$,  
  then $(\Sigma_\infty,j_\infty,\nu_\infty)$ is smooth and $u_\infty$ is an embedding.
\end{theorem}

There is a variant of the definition of $\sJ(X,\omega)$ adapted to pseudo-holomorphic maps with $\Lambda$ marked points constrained by pseudo-cycles $f_1,\ldots,f_\Lambda$.
(See \autoref{Sec_Pseudocycles} for a review of the theory of pseudo-cycles.)
The precise definition of this subspace $\sJ(X,\omega;f_1,\ldots,f_\Lambda)$ is rather lengthy and deferred to \autoref{Def_JEmbPseudoCycles}.

\begin{theorem}
  \label{Thm_Fano3GenericLimitHasSmoothDomain}
  Let $(X,\omega)$ be a compact symplectic $6$--manifold,
  let $(J_k)_{k\in\N}$ be a sequence of almost complex structures compatible with $\omega$ converging to $J_\infty$, and
  let $\paren*{u_k\co (\Sigma_k,j_k) \to (X,J_k)}_{k\in\N}$ be a sequence of pseudo-holomorphic maps which Gromov converges to the nodal $J_\infty$--holomorphic map $u_\infty\co (\Sigma_\infty,j_\infty,\nu_\infty) \to (X,J_\infty)$.
  Set $A \coloneq (u_\infty)_*[\Sigma_\infty] \in H_2(X,\Z)$.  
  Let $f_1,\ldots,f_\Lambda$ be even-dimensional pseudo-cycles  of positive codimension in general position.
  If 
  \begin{enumerate}
  \item
    $\im u_k \cap \im f_\lambda \neq \emptyset$ for every $\lambda=1,\ldots,\Lambda$,
  \item
    $2\Inner{c_1(X,\omega),A} = \sum_{\lambda=1}^\Lambda \paren*{\codim f_\lambda - 2} >0$, and
  \item
    $J_\infty \in \sJ_\emb\paren[\big]{X,\omega;f_1,\ldots,f_\Lambda}$,
  \end{enumerate}
  then $(\Sigma_\infty,j_\infty,\nu_\infty)$ is smooth and $u_\infty$ is an embedding with $\im u_\infty \cap \im f_\lambda \neq \emptyset$ for every $\lambda=1,\ldots,\Lambda$. 
\end{theorem}

\subsection{Embedded curve counts}

Denote by $\sJ_\emb^\star(X,\omega)$ the subset of those $J \in \sJ_\emb(X,\omega)$ for which every simple $J$--holomorphic map is unobstructed;
see \autoref{Def_JEmbStar}.

\begin{theorem}
  \label{Thm_CY3PrimitiveGV}
  Let $(X,\omega)$ be a symplectic $6$--manifold.
  Let $A \in H_2(X,\Z)$ be a primitive class such that $\Inner{c_1(X,\omega),A} = 0$.
  \begin{enumerate}
  \item
    \label{Thm_CY3PrimitiveGV_Defined}
    For every $g \in \N_0$ and $J \in \sJ_\emb^\star(X,\omega)$ the moduli space $\sM_{A,g}^\star(X,J)$ of simple $J$--holomorphic maps representing the class $A$ and of genus $g$ is a compact oriented zero-dimensional manifold,
    and the signed count 
    \begin{equation}
      \label{Eq_CY3PrimitiveGV}
      n_{A,g}(X,\omega)
      \coloneq
      \# \sM_{A,g}^\star(X,J)
    \end{equation}
    is independent on the choice of $J$.
  \item
    \label{Thm_CY3PrimitiveGV_Finiteness}
    There is a $g_0 \in \N_0$, depending on $(X,\omega)$ and $A$, such that 
    \begin{equation*}
      n_{A,g}(X,\omega) = 0 \qforeveryq g\geq g_0.
    \end{equation*}
  \end{enumerate}  
\end{theorem}

\begin{remark}
  \label{Rmk_CY3PrimitiveGV}
  In fact, $n_{A,g}(X,\omega)$ depends on $\omega$ only up to deformation.
\end{remark}

Again, there is a variant $\sJ_\emb^\star(X,\omega;f_1,\ldots,f_\Lambda)$ of $\sJ_\emb^\star(X,\omega)$ adapted to pseudo-holomorphic maps with $\Lambda$ marked points constrained by pseudo-cycles $f_1,\ldots,f_\Lambda$;
see \autoref{Def_JEmbStarPseudoCycles}.

\begin{theorem}
  \label{Thm_Fano3GV}
  Let $(X,\omega)$ be a symplectic $6$--manifold,
  let $A \in H_2(X,\Z)$,
  let $\gamma_1,\ldots,\gamma_\Lambda \in H^\even(X,\Z)$ be such that $\deg(\gamma_\lambda) > 0$ and 
  \begin{equation*}
    2\Inner{c_1(X,\omega),A} = \sum_{\lambda=1}^\Lambda \paren*{\deg(\gamma_\lambda)-2} > 0.
  \end{equation*}
  \begin{enumerate}
  \item
    \label{Thm_Fano3GV_Defined}
    Let $f_1,\ldots,f_\Lambda$ be pseudo-cycles in $X$ which are Poincaré dual to $\gamma_1,\ldots,\gamma_\Lambda$ and in  general position.
    For every $g \in \N_0$ and $J \in \sJ_\emb^\star\paren[\big]{X,\omega;f_1,\ldots,f_\Lambda}$ the moduli space $\sM_{A,g}^\star\paren[\big]{X,J;f_1,\ldots,f_\Lambda}$ of simple $J$--holomorphic maps representing the class $A$, of genus $g$, and intersecting $f_1,\ldots,f_\Lambda$ is a compact oriented zero-dimensional manifold,
    and the signed count 
    \begin{equation}
      \label{Eq_Fano3GV}
      n_{A,g}(X,\omega;\gamma_1,\ldots,\gamma_\Lambda)
      \coloneq
      \#\sM_{A,g}^\star(X,J;f_1,\ldots,f_\Lambda)
    \end{equation}
    is independent on the choice of $f_1,\ldots,f_\Lambda$ and $J$.
  \item
    \label{Thm_Fano3GV_Finiteness}
    There exists a $g_0 \in \N_0$, depending on $(X,\omega)$, $A$, and $\gamma_1,\ldots,\gamma_\Lambda$,  such that
    \begin{equation*}
      n_{A,g}(X,\omega;\gamma_1,\ldots,\gamma_\Lambda) = 0 \quad\text{for all } g\geq g_0.
    \end{equation*}
  \end{enumerate}  
\end{theorem}

\begin{remark}
  \autoref{Rmk_CY3PrimitiveGV} applies mutatis mutandis.
\end{remark}

\subsection{Gopakumar and Vafa's BPS invariants}

Using ideas from $M$--theory,
\citet{Gopakumar1998,Gopakumar1998a} predicted that there are integer invariants $\BPS_{A,g}(X,\omega)$ associated with every closed symplectic $6$--manifold $(X,\omega)$, a class $A \in H_2(X,\Z)$ with $\Inner{c_1(X,\omega),A} = 0$, and $g \in \N_0$,
which count BPS states supported on embedded $J$--holomorphic curves representing $A$ and of genus $g$.
\citeauthor{Gopakumar1998} did not give a direct mathematical definition of $\BPS_{A,g}(X,\omega)$;
however, they conjectured that their invariants are related to the Gromov--Witten invariants $\GW_{A,g}(X,\omega)$ by the marvelous formula
\begin{equation}
  \label{Eq_CY3GopakumarVafa}
  \sum_{A}\sum_{g=0}^\infty
  \GW_{A,g}(X,\omega)
  \cdot
  t^{2g-2} q^A
  =
  \sum_{A}\sum_{g=0}^\infty
  \BPS_{A,g}(X,\omega)
  \cdot
  \sum_{k=1}^\infty
  \frac1k
  \paren*{2\sin\paren*{kt/2}}^{2g-2}
  q^{kA}
\end{equation} 
with the sum taken over all non-zero Calabi--Yau classes $A$ and, moreover,
that $\BPS_{A,g}(X,\omega) = 0$ for $g \gg 1$.
 
In algebraic geometry, there are approaches to defining the BPS invariants for projective Calabi--Yau three-folds \cite{Hosono2001, Pandharipande2009, Pandharipande2010, Kiem2012, Maulik2018}.
These satisfy the Gopakumar--Vafa formula \autoref{Eq_CY3GopakumarVafa} in some cases,
but it is not currently known whether the formula holds in general.

An alternative approach is to take \autoref{Eq_CY3GopakumarVafa} as the \emph{definition} of $\BPS_{A,g}(X,\omega)$;
see \cite[Section 2]{Bryan2001}.
This approach leads to the following conjecture.

\begin{conjecture}[{\citet{Gopakumar1998,Gopakumar1998a}; see also \cite[Conjecture 1.2]{Bryan2001}}]
  \label{Conj_GV}
  The numbers $\BPS_{A,g}(X,\omega)$ defined by \autoref{Eq_CY3GopakumarVafa} satisfy
  \begin{description}[labelwidth=\widthof{\bfseries (integrality)}]
  \item[(integrality)]
    \label{Conj_GV_Integrality}
    $\BPS_{A,g}(X,\omega) \in \Z$, and
  \item[(finiteness)]
    \label{Conj_GV_Finiteness}
    $\BPS_{A,g}(X,\omega) = 0$ for $g \gg 1$.
    \qedhere
  \end{description}
\end{conjecture}

The Gopakumar--Vafa integrality conjecture has been proved by \citet{Ionel2018}.
\citet[footnote 11]{Zinger2011} has proved that for primitive Calabi--Yau classes
\begin{equation*}
  \BPS_{A,g}(X,\omega) = n_{A,g}(X,\omega);
\end{equation*}
see also \autoref{Sec_N=BPS}.
Therefore,
\autoref{Thm_CY3PrimitiveGV} implies the following.

\begin{cor}
  \label{Cor_CY3GVFiniteness}
  The Gopakumar--Vafa finiteness conjecture holds for primitive Calabi--Yau classes;
  that is: for every closed symplectic $6$--manifold $(X,\omega)$ and every primitive Calabi--Yau class $A \in H_2(X,\Z)$ there is a $g_0(\omega,A)$ such that for every $g \geq g_0(\omega,A)$
  \begin{equation*}
    \BPS_{A,g}(X,\omega) = 0.
    \qedhere
  \end{equation*}
\end{cor}

\begin{remark}
  The finiteness conjecture for general Calabi--Yau classes has been resolved recently \cite{Doan2021}.
\end{remark}

The genus bound in \autoref{Cor_CY3GVFiniteness} is not effective;
therefore, it is natural to ask the following.

\begin{definition}
  Let $(X,\omega)$ be a closed symplectic $6$--manifold and $A \in H_2(X,\Z)$ a Calabi--Yau class.
  Define the \defined{BPS Castelnuovo number} $\gamma_A(X,\omega)$ by
  \begin{equation*}
    \gamma_A(X,\omega)
    \coloneq
    \inf\set*{ g \in \N : \BPS_{A,g}(X,\omega) = 0 }
    \in \N_0
    \qedhere
  \end{equation*}
\end{definition}

\begin{question}
  \label{Q_CY3Castelnuovo}
  Is there an bound on $\gamma_A(X,\omega)$ analogous to Castelnuovo's bound for the genus of an irreducible degree $d$ curve in $\P^n$ \cites{Castelnuovo1889}[Chapter III Section 2]{Arbarello1985};
  that is: a bound of $\gamma_A(X,\omega)$ by a formula involving $A$ and the geometry of $X$?
  (See \citet{Huang2015,Knapp2021} for some work in this direction.)
\end{question}

There is an analogue of the Gopakumar--Vafa formula for Fano classes.
Given $A \in H_2(X,\Z)$, $g \in \N_0$, and $\gamma_1, \ldots, \gamma_\Lambda \in H^\even(X,\Z)$
satisfying $\deg(\gamma_\lambda) > 0$ and 
\begin{equation}
  \label{Eq_Fano3DimensionConstraint}
  2\Inner{c_1(X,\omega),A} = \sum_{\lambda=1}^\Lambda (\deg(\gamma_\lambda) - 2) > 0,
\end{equation}
denote by $\GW_{A,g}(X,\omega;\gamma_1,\ldots,\gamma_\Lambda)$ be the corresponding Gromov--Witten invariant.
The analogue of \autoref{Eq_CY3GopakumarVafa} is
\begin{equation}
  \label{Eq_Fano3GopakumarVafa}
  \begin{split}
    &
    \sum_A\sum_{g=0}^\infty
    \GW_{A,g}(X,\omega;\gamma_1,\ldots,\gamma_\Lambda)
    \cdot
    t^{2g-2} q^A \\
    &=
    \sum_A\sum_{g=0}^\infty
    \BPS_{A,g}(X,\omega;\gamma_1,\ldots,\gamma_\Lambda)
    \cdot
    \paren*{2\sin(t/2)}^{2g-2+\inner{c_1(X,\omega)}{A}}
    q^A
  \end{split}
\end{equation}
with the sum taken over all $A \in H_2(X,\Z)$ satisfying \autoref{Eq_Fano3DimensionConstraint}.
\citet[Theorem 1.5]{Zinger2011} has proved that
\begin{equation*}
  \BPS_{A,g}(X,\omega;\gamma_1,\ldots,\gamma_\Lambda)
  =
  n_{A,g}(X,\omega;\gamma_1,\ldots,\gamma_\Lambda);
\end{equation*}
thus establishing the analogue of the Gopakumar--Vafa integrality conjecture.
Furthermore,
\autoref{Thm_Fano3GV} implies the following.

\begin{cor}
  The analogue of the Gopakumar--Vafa finiteness conjecture holds for all Fano classes.
\end{cor}

Of course,
there is an analogue of \autoref{Q_CY3Castelnuovo} in the Fano case.

\paragraph*{Acknowledgements}
We thank Aleksei Zinger for several discussions about \cite{Zinger2009} and Eleny Ionel for answering our questions about \cite{Ionel2018}.
We thank the referees for their meticulous work and, especially, for alerting us to an issue in an earlier version of \autoref{Sub_GhostComponents} and its implications for \autoref{Sec_LeadingOrderTermOfObstructionGhostComponents}.
This material is based upon work supported by \href{https://www.nsf.gov/awardsearch/showAward?AWD_ID=1754967&HistoricalAwards=false}{the National Science Foundation under Grant No.~1754967}, 
\href{https://sloan.org/grant-detail/8651}{an Alfred P. Sloan Research Fellowship},
\href{https://sites.duke.edu/scshgap/}{the Simons Collaboration on Special Holonomy in Geometry, Analysis, and Physics}, 
and \href{https://www.simonsfoundation.org/simons-society-of-fellows/}{the Simons Society of Fellows}.



\section{Nodal pseudo-holomorphic maps}
\label{Sec_ModuliSpaces}

This section reviews a few definitions and results regarding nodal pseudo-holomorphic maps.

\subsection{Nodal manifolds}

\begin{definition}
  \label{Def_NodalManifold}
  Let $X$ be a manifold, possibly disconnected. 
  A \defined{nodal structure} on $X$ is an involution $\nu \co X \to X$ whose fixed-point set has a discrete complement. 
  (This involution is discontinuous unless $\nu = \id$.)
  The set of points not fixed by $\nu$ is called the \defined{nodal set}.
  A \defined{nodal manifold} is a manifold together with a nodal structure.
\end{definition}

The quotient $X/\nu$ should be considered as the topological space underlying the nodal manifold $(X,\nu)$.
The atlas of $X$ induces a ``nodal atlas'' for $X/\nu$ consisting of ``charts'' mapping either to $\R^n$ or $\R^n\times\set{0} \cup \set{0}\times \R^n \subset \R^{2n}$.
The nodes of $X/\nu$ are precisely the points mapping to $(0,0) \in \R^{2n}$ in some chart
or, equivalently,
the images of the points in the nodal set.

\begin{definition}
  Let $(X_1,\nu_1)$ and $(X_2,\nu_2)$ be nodal manifolds.
  A \defined{nodal map} $f \co (X_1,\nu_1) \to (X_2,\nu_2)$ is a smooth map $f\co X_1 \to X_2$ such that
  \begin{equation*}
    f\circ\nu_1 = \nu_2 \circ f.
    \qedhere
  \end{equation*}
\end{definition}

\begin{definition}
  Let $(X,\nu)$ be a nodal manifold.
  A \defined{diffeomorphism} of $(X,\nu)$ is an element of
  \begin{equation*}
    \Diff(X,\nu)
    \coloneq
    \set*{
      \phi \in \Diff(X)
      :
      \phi\circ\nu = \nu\circ\phi
    }.
    \qedhere
  \end{equation*}
\end{definition}

Every manifold $X$ canonically is a nodal manifold with $\nu = \id_X$ and
a smooth map between manifolds, trivially, is a nodal map.
In other words, the category of manifolds is a full subcategory of the category of nodal manifolds.

\begin{definition}
  Let $(X,\nu)$ be a nodal manifold,
  let $Y$ be a manifold, and
  let $f\co (X,\nu) \to Y$ be a nodal map.
  For a vector bundle $E \to Y$, set 
  \begin{equation*}
    \Gamma(X,\nu;f^*E)
    \coloneq
    \set*{
      \xi \in \Gamma(X,f^*E) : \xi \circ \nu = \xi
    }.
    \qedhere
  \end{equation*}
\end{definition}

\begin{remark}
  In the situation of the preceding definition,
  set $n \coloneq \dim X$ and let $p > n$.
  Given a Riemannian metric on $X$, a Euclidean metric on $E$, and a metric covariant derivative on $E$,
  denote by $W^{1,p}\Gamma(X, f^*E)$ the completion of $\Gamma(X,f^*E)$ with respect to the corresponding $W^{1,p}$ norm.
  By Morrey's embedding theorem,
  $W^{1,p} \into C^{0,1-n/p}$.
  Therefore,
  the evaluations maps $\ev_x\co \Gamma(X,f^*E) \to E_{f(x)}$ extend to $W^{1,p}\Gamma(X,f^*E)$ and
  \begin{equation*}
    W^{1,p}\Gamma(X,\nu;f^*E)
    =
    \set*{
      \xi \in W^{1,p}\Gamma(X;f^*E)
      :
      \xi(\nu(x)) = \xi(x) \text{ for every } x \in X
    }.
  \end{equation*}
  For $p < n$ it can be shown that the $W^{1,p}$ completion of $\Gamma(X,\nu;f^*E)$ agrees with the $W^{1,p}$ completion of $\Gamma(X;f^*E)$.
\end{remark}

\subsection{Nodal Riemann surfaces}

\begin{definition}
  A \defined{nodal Riemann surface} is a Riemann surface $(\Sigma,j)$ together with a nodal structure $\nu$.
\end{definition}

\begin{definition}
  Let $C$ be a complex analytic curve. 
  A point of $C$ is a \defined{node} if it has a neighborhood which is isomorphic to a neighborhood of the point $(0,0)$ in the curve
  \begin{equation*}
    \set*{ (z,w) \in \C^2 : zw = 0 }. 
  \end{equation*}  
  A \defined{nodal curve} is a complex analytic curve all of whose points are either smooth or a node.
\end{definition}

Let $C$ be a nodal curve and denote by $\pi \co \tilde C \to C$ its normalization.
The complex analytic curve $\tilde C$ is smooth and, hence, equivalent to a closed Riemann surface $(\Sigma,j)$.
Since $\tilde C$ is obtained from $C$ by replacing every node with a pair of points,
$\Sigma$ inherits a canonical nodal structure $\nu$.
This sets up an equivalence between complete, nodal curves $C$ and closed, nodal Riemann surfaces $(\Sigma,j,\nu)$.

\begin{definition}
  The \defined{automorphism group} of a nodal Riemann surface $(\Sigma,j,\nu)$ is
  \begin{equation*}
    \Aut(\Sigma,j,\nu) \coloneq \set*{ \phi \in \Diff(\Sigma,\nu) : \phi_*j = j }.
  \end{equation*}
  A nodal Riemann surface $(\Sigma,j,\nu)$ is \defined{stable} if $\Aut(\Sigma,j,\nu)$ is finite.
\end{definition}

\begin{definition}
  Let $(\Sigma,\nu)$ be a nodal surface with nodal set $S$.
  The \defined{arithmetic genus} of $(\Sigma,\nu)$ is
  \begin{equation}
    \label{Eq_ArithmeticGenus}
    p_a(\Sigma,\nu) \coloneq 1 - \frac12\paren*{\chi(\Sigma) - \#S}.
    \qedhere
  \end{equation}
\end{definition}

\begin{remark}
  \label{Rmk_ArithmeticGenus_Smoothing}
  If $(\tilde \Sigma,\tilde \nu)$ denotes a nodal surface obtained from $(\Sigma,\nu)$ by attaching a $1$--handle at some pairs of nodes $\set{n,\nu(n)}$,
  then
  \begin{equation*}		
    p_a(\Sigma,\nu) = p_a(\tilde\Sigma,\tilde\nu).
    \qedhere
  \end{equation*}
\end{remark}

\subsection{Nodal \texorpdfstring{$J$}{J}--holomorphic maps}
\label{Subsec_NodalJHolomorphicMaps}

Throughout the next four subsections,
let $(X,J)$ be an almost complex manifold of dimension $2n$.

\begin{definition}
  A \defined{nodal $J$--holomorphic map} $u\co (\Sigma,j,\nu) \to (X,J)$ is a nodal Riemann surface $(\Sigma,j,\nu)$ together with a nodal map $u\co (\Sigma,\nu) \to X$ which is $J$--holomorphic;
  that is:
  \begin{equation}
    \label{Eq_JHolomorphic}
    \delbar_J(u,j)
    \coloneq
    \frac12\paren*{
      \rd u + J(u)\circ\rd u\circ j
    }
    =
    0.
    \qedhere
  \end{equation}
\end{definition}

\begin{definition}
  If $u\co (\Sigma,j,\nu) \to (X,J)$ is a nodal $J$--holomorphic map and $\phi \in \Diff(\Sigma,\nu)$,
  then the \defined{reparametrization} $\phi_*u \coloneq u\circ\phi^{-1}\co (\Sigma,\phi_*j,\nu) \to (X,J)$ is a nodal $J$--holomorphic map as well.
  The \defined{automorphism group} of a nodal $J$--holomorphic map $u\co(\Sigma,j,\nu) \to (X,J)$ is 
  \begin{equation*}
    \Aut(u) \coloneq \set*{ \phi\in \Aut(\Sigma,j,\nu) : u\circ\phi = u }.
  \end{equation*}
  The map $u$ is said to be \defined{stable} if $\Aut(u)$ is finite. 
\end{definition}

\begin{definition}
  Let $(\Sigma,j)$ and $(\tilde\Sigma, \tilde j)$ be smooth Riemann surfaces.
  Let $u\co (\Sigma,j) \to (X,J)$ be a $J$--holomorphic map and let $\pi\co (\tilde \Sigma,\tilde j) \to (\Sigma,j)$ be a holomorphic map of degree $\deg(\pi) \geq 2$.
  The composition $u\circ \pi \co (\tilde\Sigma,\tilde j) \to (X,J)$ is said to be a \defined{multiple cover of $u$}.
  A $J$--holomorphic map is \defined{simple} if it is not constant and not a multiple cover.
\end{definition}

\subsection{Ghost components}
\label{Sub_GhostComponents} 

Let $u\co (\Sigma,j,\nu) \to (X,J)$ be a nodal $J$--holomorphic map.
Let $S$ be the nodal set of $(\Sigma,\nu)$. 

\begin{definition}
  Suppose $C \subset \Sigma$ is a union of connected components of $\Sigma$.
  Set
  \begin{equation*}
    S_C^\rint \coloneq \set { n \in S : n \in C \text{ and } \nu(n) \in C} \qandq
    S_C^\rext \coloneq \set { n \in S : n \in C \text{ and } \nu(n) \notin C}
  \end{equation*}
  and denote by $\nu_C$ the nodal structure on $C$ which agrees with $\nu_0$ on $S_C^\rint$ and the identity on the complement of $S_C^\rint$.
  Denote by $\check{C}$ the nodal curve associated with $(C,j,\nu)$.
\end{definition}

\begin{definition}
  A \defined{ghost component} of $u$ is a union $C$ of connected components of $\Sigma$ such that $u|_C$ is constant, $\check{C}$ is connected, and which is a maximal subset satisfying these properties.
\end{definition}

\autoref{Prop_BaseLocusOfDualizingSheaf} below, which will be used in the proof of  \autoref{Thm_LimitConstraints} (specifically  in \autoref{Sec_LeadingOrderTermOfObstructionGhostComponents}), concerns the dualizing sheaf of a nodal curve $C$. 
The dualizing sheaf is a generalization of the canonical sheaf of a smooth curve; for the reader's convenience, we describe its construction in the proof of \autoref{Prop_BaseLocusOfDualizingSheaf}. 

\begin{definition}
  Let $C$ be a nodal curve.
  The \defined{dual graph of $C$} is the weighted graph whose
  set of vertices is the set of irreducible components of $C$ with the genus as the weight function,
  and edges between two vertices if and only if the corresponding irreducible components of $C$ intersect.
\end{definition}

\begin{prop}
  \label{Prop_BaseLocusOfDualizingSheaf}
  Let $C$ be a nodal curve.
  Denote the dual graph of $C$ by $\Gamma$.
  Denote by $\omega_C$ the dualizing sheaf of $C$ and by $B$ its base-locus:
  \begin{equation*}
    B \coloneq \set{ x \in C : \zeta(x) = 0 ~\text{for every}~ \zeta \in H^0(C,\omega_C)}.
  \end{equation*}
  The base-locus has the following description:
  \begin{enumerate}
  \item
    \label{Prop_BaseLocusOfDualizingSheaf_Components}
    $B$ is a union of irreducible rational components of $C$.
  \item
    \label{Prop_BaseLocusOfDualizingSheaf_Graph}
    The dual graph of $B$ is the subgraph $\Delta \subset \Gamma$ obtained by
    \begin{enumerate}
    \item
      removing every vertex of non-zero weight, and
    \item
      removing every simple cycle in $\Gamma$.
    \end{enumerate}
    In particular, $\Delta$ is a forest with weight zero.
    Moreover, if $e_1,e_2$ are distinct vertices of a tree $T \subset \Delta$,
    then they cannot be connected by a path in $(\Gamma\setminus T) \cup \set{e_1,e_2}$.
  \end{enumerate}
\end{prop}

The proof relies on the following.

\begin{prop}
  \label{Prop_MeromorphicOneFormsWithPoles}
  Let $\Sigma$ be a connected smooth curve.
  For every three $p,q,r$ distinct points on $\Sigma$
  there is a $\zeta \in H^0(K_\Sigma(p+q))$ with
  \begin{equation*}
    \Res_p\zeta = - \Res_q\zeta \neq 0 \qandq
    \zeta(r) \neq 0.
  \end{equation*}
  Here $\Res_p\zeta$ denotes the residue at $p$ of the meromorphic $1$--form $\zeta$.
\end{prop}

\begin{proof}
  For $\Sigma = \P^1$ without loss of generality $p=0$ and $q=\infty$;
  hence, the meromorphic $1$--form can be taken to be $z^{-1}\rd z$.

  Suppose $\Sigma \neq \P^1$.
  Consider the exact sequence
  \begin{equation*}
    \begin{tikzcd}[column sep=scriptsize]
      H^0(K_\Sigma) \ar[hook]{r} & H^0(K_\Sigma(p+q)) \ar{r}{\rho} & H^0(\sO_p\oplus \sO_q) \iso \C\oplus\C \ar[two heads]{r}{\delta} & H^1(K_\Sigma) \iso \C
    \end{tikzcd}
  \end{equation*}
  with $\rho(\zeta) \coloneq (\Res_p\zeta,\Res_q\zeta)$ and $\delta(a,b) \coloneq a-b$.
  This implies that there is a $\zeta \in H^0(K_\Sigma(p+q))$ with non-vanishing residues at $p$ and $q$.
  Since $K_\Sigma$ is base-point-free, $\zeta$ can be arranged not to vanish at $r$. 
\end{proof}

\begin{proof}[Proof of \autoref{Prop_BaseLocusOfDualizingSheaf}]
  The dualizing sheaf of $C$ is constructed as follows;
  see \cite[p. 91]{Arbarello2011}.
  Denote by $\pi\co \Sigma \to C$ the normalization map.
  Denote by $S$ the set of nodal points of $C$.
  Denote by $\tilde \omega_C$ the subsheaf of $K_\Sigma(S)$ whose sections $\zeta$ satisfy
  \begin{equation}
    \label{Eq_DualizingSheafResidueCondition}
    \Res_n \zeta + \Res_{\nu(n)}\zeta = 0
  \end{equation}
  for every $n \in S$.
  Here $\nu$ denotes the obvious involution on $\pi^{-1}(S)$.
  The dualizing sheaf $\omega_C$ then is
  \begin{equation*}
    \omega_C = \pi_*\tilde \omega_C.
  \end{equation*}

  The base-locus of $K_\Sigma$ are precisely the rational connected components of $\Sigma$.
  This implies \autoref{Prop_BaseLocusOfDualizingSheaf_Components}.
  It follows from the above \autoref{Prop_MeromorphicOneFormsWithPoles} that the dual graph of $B$ is contained in $\Delta$.
  By the Residue Theorem any meromorphic $1$--form with simple poles must have at least two poles.
  This implies that the dual graph of $B$ agrees with $\Delta$.  
\end{proof}

\subsection{Moduli spaces of nodal pseudo-holomorphic maps}

\begin{definition}
  \label{Def_ModuliOfStableMaps}
  Given $A \in H_2(X,\Z)$ and $g \in \N_0$,
  the \defined{moduli space of stable nodal $J$--holomorphic maps} representing $A$ and of genus $g$ is the set
  \begin{equation*}
    \overline\sM_{A,g}(X,J)
  \end{equation*}
  of equivalence classes of stable nodal $J$--holomorphic maps $u\co (\Sigma,j,\nu) \to (X,J)$ up to reparametrization with
  \begin{equation*}
    u_*[\Sigma] = A
    \qandq
    p_a(\Sigma,\nu) = g.
  \end{equation*}
  The subset of $\overline\sM_{A,g}(X,J)$ parametrizing simple $J$--holomorphic maps is denoted by
  \begin{equation*}
    \sM_{A,g}^\star(X,J).
    \qedhere
  \end{equation*}
\end{definition}

At this stage,
$\overline\sM_{A,g}(X,J)$ is just a set.
In \autoref{Sub_GromovTopology},
it will be equipped with the \defined{Gromov topology}.
This topology induces the $C^\infty$ topology on $\sM_{A,g}^\star(X,J)$.

\begin{definition}
  \label{Def_UniversalModuliSpace}
  Let $(X,\omega)$ be a symplectic manifold.
  Denote by $\sJ(X,\omega)$ the space of almost complex structures on $X$ which are \defined{compatible with $\omega$};
  that is:
  \begin{equation*}
    g(\cdot,\cdot) \coloneq \omega(\cdot, J\cdot)
  \end{equation*}
  defines a Riemannian metric on $X$.
  Equip $\sJ(X,\omega)$ with the $C^\infty$ topology.   
\end{definition}

\begin{definition}
  Given $A \in H_2(X,\Z)$ and $g \in \N$,
  set
  \begin{equation*}
    \overline\sM_{A,g}(X,\omega)
    \coloneq
    \coprod_{J \in \sJ(X,\omega)} \overline\sM_{A,g}(X,J)
    \qandq
    \sM_{A,g}^\star(X,\omega)
    \coloneq
    \coprod_{J \in \sJ(X,\omega)} \sM_{A,g}^\star(X,J).
  \end{equation*}
  Denote by $\pi\co \overline\sM_{A,g}(X,\omega) \to \sJ(X,\omega)$ the canonical projection.
\end{definition}

\subsection{Linearization of the \texorpdfstring{$J$}{J}--holomorphic map equation}

Let $u\co (\Sigma,j,\nu) \to (X,J)$ be a nodal $J$--holomorphic map.
Let $h$ be a Hermitian metric on $(X,J)$ and let $\nabla$ be a torsion-free connection on $TX$.
Throughout the remainder of this article,
let $p>2$. 

\begin{definition}
  \label{Def_FU}
  Given $\xi \in W^{1,p}\Gamma(\Sigma,\nu;u^*TX)$,
  set
  \begin{equation*}
    u_\xi \coloneq \exp_u(\xi)
  \end{equation*}
  and denote by $\Psi_\xi\co L^p\Omega^{0,1}(\Sigma,u^*TX) \to L^p\Omega^{0,1}(\Sigma,u_\xi^*TX)$ the map induced by parallel transport along the geodesics $t \mapsto \exp_u(t\xi)$.
  Define $\fF_{u,j,\nu;J} \co W^{1,p}\Gamma(\Sigma,\nu;u^*TX) \to L^p\Omega^{0,1}(\Sigma,u^*TX)$ by
  \begin{equation*}
    \fF_{u,j,\nu;J}(\xi) \coloneq \Psi_\xi^{-1}\delbar_J(u_\xi,j).
    \qedhere
  \end{equation*}
\end{definition}

\begin{definition}
  \label{Def_JHolomorphicLinearization}
  Define the linear operator $\fd_{u,j,\nu;J} \co W^{1,p}\Gamma(\Sigma,\nu;u^*TX) \to L^p\Omega^{0,1}(\Sigma,u^*TX)$ by
  \begin{align*}
    \fd_{u,j,\nu;J}\xi
    &\coloneq
      \rd_0\fF_{u,j,\nu;J}\xi
      =
      \frac12\paren*{
      \nabla\xi
      +
      J(u)\circ(\nabla\xi)\circ j
      +
      (\nabla_\xi J)\circ\rd u\circ j
      }.
      \qedhere
  \end{align*}
\end{definition}

\begin{remark}
  If $u$ is $J$--holomorphic,
  then $\fd_{u,j,\nu;J}$ does not depend on the choice of torsion-free connection $\nabla$ on $TX$;
  see \cite[Proposition 3.1.1]{McDuff2012}.
\end{remark}

The operator $\fd_{u,j,\nu;J}$ is the restriction to $W^{1,p}\Gamma(\Sigma,\nu;u^*TX)$ of the operator
\begin{equation*}
  \fd_{u,j;J}\co W^{1,p}\Gamma(\Sigma,u^*TX) \to L^p\Omega^{0,1}(\Sigma,u^*TX)
\end{equation*} 
given by the same formula. 
The former controls the deformation theory of $u$ as a nodal $J$--holomorphic map from the nodal Riemann surface $(\Sigma,j,\nu)$ whereas the latter controls the deformation theory of $u$ as a smooth $J$--holomorphic map from the smooth Riemann surface $(\Sigma,j)$, ignoring the nodal structure.

\begin{prop}
  The index of $\fd_{u,j,\nu;J}$ is given by
  \begin{equation}
    \label{Eq_IndexFormula}
    \ind \fd_{u,j,\nu;J}
    =
    2\inner{[\Sigma]}{u^*c_1(X,J)} + 2n\paren*{1-p_a(\Sigma,\nu)}.
  \end{equation}
\end{prop}

\begin{proof}
  The inclusion
  \begin{equation*}
    W^{1,p}\Gamma(\Sigma,\nu;u^*TX) \to W^{1,p}\Gamma(\Sigma,u^*TX).
  \end{equation*}
  has index $-n\#S$.
  By the Riemann--Roch Theorem,
  \begin{equation*}
    \ind \fd_{u,j;J}
    =
    2\inner{[\Sigma]}{u^*c_1(X,J)} + n\chi(\Sigma).
  \end{equation*}
  These together with \autoref{Eq_ArithmeticGenus} imply the index formula.
\end{proof}

\begin{remark}
  \label{Rmk_NodalKernelAndCokernel}
  For our discussion in \autoref{Sec_LeadingOrderTermOfObstructionGhostComponents},
  which establishes the key technical result of this article,
  the following detailed description of the kernel and cokernel of $\fd_{u,j,\nu;J}$ will be important.
  Denote by
  \begin{equation*}
    V_- \subset \bigoplus_{n \in S} T_{u(n)}X
  \end{equation*}
  the subspace of those $(v_n)_{n \in S}$ satisfying
  \begin{equation*}
    v_{\nu(n)} = -v_n.
  \end{equation*}
  Define $\diff \co \ker\fd_{u,j;J} \to V_-$ by
  \begin{equation*}
    \diff \kappa \coloneq \paren*{\kappa(n) - \kappa(\nu(n))}_{n\in S}.
  \end{equation*}
  Evidently,
  \begin{equation*}
    \ker\fd_{u,j,\nu;J} = \ker \diff.
  \end{equation*}
  The map $\diff$ is induced by the analogously defined map $W^{1,p}(\Sigma,u^*TX) \to V_-$ which fits in to the following commutative diagram with exact rows
  \begin{equation*}
    \begin{tikzcd}
      0 \ar[r] & W^{1,p}(\Sigma,\nu;u^*TX) \ar{r}\ar{d}{\fd_{u,j,\nu;J}} & W^{1,p}(\Sigma,u^*TX) \ar{r}\ar{d}{\fd_{u,j;J}} & V_- \ar{r}\ar{d} & 0 \\
      0 \ar{r} & L^p\Omega^{0,1}(\Sigma,u^*TX) \ar{r}{=}& L^p\Omega^{0,1}(\Sigma,u^*TX) \ar{r} & 0 \ar{r} & 0.
    \end{tikzcd}  
  \end{equation*}
  Therefore,
  the Snake Lemma yields the short exact sequence
  \begin{equation*}
    \begin{tikzcd}[column sep=small]
      0 \ar[r] & \coker\diff \ar[r] & \coker\fd_{u,j,\nu;J} \ar[r] &  \coker\fd_{u,j;J} \ar[r] & 0.
    \end{tikzcd}    
  \end{equation*}
  The dual sequence
  \begin{equation*}
    \begin{tikzcd}[column sep=small]
      0 \ar[r] & (\coker\fd_{u,j;J})^* \ar[r] & (\coker\fd_{u,j,\nu;J})^* \ar[r] & (\coker\diff )^* \ar[r] & 0
    \end{tikzcd}
  \end{equation*}
  can be understood as follows.  
  Let $q \in (1,2)$ be such that $1/p+1/q = 1$.
  The dual space $\paren{\coker\fd_{u,j\nu;J}}^*$ can be identified via the pairing between $L^p$ and $L^q$ with the space $\sH$ consisting of those $\zeta \in L^q\Omega^{0,1}(\Sigma,u^*TX)$ which satisfy a distributional equation of the form
  \begin{equation*}
    \fd_{u,j,J}^*\zeta = \sum_{n\in S} v_n \delta_n.
  \end{equation*}
  with $v = (v_n)_{n\in S} \in (\im \diff)^\perp \iso (\coker\diff)^*$ and $\delta_n$ denoting the Dirac $\delta$ distribution at $n$.  
  The map $(\coker\fd_{u,j,\nu;J})^* \to (\coker\diff)^*$ maps $\zeta$ to $v$.
\end{remark}

\begin{definition}
  Define the map $\fn_{u,j,\nu;J} \co W^{1,p}\Gamma(\Sigma,\nu;u^*TX) \to L^p\Omega^{0,1}(\Sigma,u^*TX)$ by
  \begin{equation*}
    \fn_{u,j,\nu;J}(\xi)
    \coloneq
    \fF_{u,j,\nu;J}(\xi) - \delbar_J(u,j) - \fd_{u,j,\nu;J}\xi.
    \qedhere
  \end{equation*}  
\end{definition}

\begin{prop}[{\cite[Proposition 3.5.3 and Remark 3.5.5]{McDuff2012}}]
  \label{Prop_QuadraticEstimate}
  Denote by $c_S > 0$ an upper bound for the norm of the embedding $W^{1,p}(\Sigma) \into C^{0,1-2/p}(\Sigma)$ and let $c_\xi > 0$.
  For every $\xi_1,\xi_2$ with $\Abs{\xi_1}_{W^{1,p}} \leq c_\xi$ and $\Abs{\xi_2}_{W^{1,p}} \leq c_\xi$
  \begin{equation*}
    \Abs{\fn_{u,j,\nu;J}(\xi_1)-\fn_{u,j,\nu;J}(\xi_2)}_{L^p}
    \leq
    c(c_S,c_\xi,\Abs{\rd u}_{L^p})
    \cdot
    \paren*{
      \Abs{\xi_1}_{W^{1,p}}
      +
      \Abs{\xi_2}_{W^{1,p}}
    }
    \cdot
    \Abs{\xi_1-\xi_2}_{W^{1,p}}.
  \end{equation*}
\end{prop}

So far, the complex structure $j$ has been held fixed.
Denote by $\sJ(\Sigma)$ the space of complex structures on $\Sigma$ and by $\Diff_0(\Sigma,\nu)$ the group of diffeomorphism of $\Sigma$ which are isotopic to the identity and commute with $\nu$.
Denote by
\begin{equation*}
  \sT \coloneq \sJ(\Sigma)/\Diff_0(\Sigma,\nu)
\end{equation*}
the corresponding Teichmüller space.
This is a complex manifold whose real dimension satisfies
\begin{equation*}
  \dim \sT  - \dim\aut(\Sigma,j,\nu) +\#S= 6(p_a(\Sigma,\nu)-1).
\end{equation*}
For every $j \in \sJ(\Sigma)$ there is a \defined{Teichmüller slice through $j$};
that is:
an open neighborhood $\Delta$ of $0 \in \C^{\dim_\C \sT}$ together with a $\Aut(\Sigma,j,\nu)$--equivariant map $\jmath\co \Delta \to \sJ(\Sigma)$ such that $\jmath(0) = j$.

\begin{definition}
  Consider the bundle over $\Delta$ whose fiber over $\sigma\in\Delta$ is the Banach space
  \begin{equation*}
  L^p\Omega^{0,1}(\Sigma,u^*TX).
  \end{equation*}
  Here the space of $(0,1)$ forms on $\Sigma$ is defined with respect to complex structure $\jmath(\sigma)$.
  A choice of a trivialization of this bundle  gives rise to the map
  \begin{equation}
    \label{Eq_BackTransportedParametrizedCauchyRiemann}
    \begin{split}
      W^{1,p}\Gamma(\Sigma,\nu;u^*TX)\times\Delta &\to L^p\Omega^{0,1}(\Sigma,u^*TX) \\
      (\xi,\sigma) &\mapsto \sF_{u,\jmath(\sigma),\nu;J}(\xi).
    \end{split}
  \end{equation}
  Define $\rd_{u,j}\delbar_{\nu;J} \co  W^{1,p}\Gamma(\Sigma,\nu;u^*TX) \oplus T_0 \Delta \to L^p\Omega^{0,1}(\Sigma,u^*TX)$ to be the derivative of the map \eqref{Eq_BackTransportedParametrizedCauchyRiemann} at $(0,0)$.
\end{definition}

\begin{definition}
  The \defined{index} of $u$ is 
  \begin{equation*}    
    \begin{split}
      \ind(u)
      &\coloneq
      \ind(\rd_{u,j}\delbar_{\nu;J}) - \dim\aut(\Sigma,j,\nu) + \#S \\
      &=
      2\Inner{[\Sigma],u^*c_1(X,J)} + 2(n-3)\paren*{1-p_a(\Sigma,\nu)}.
    \end{split}
  \end{equation*}
  The map $u\co (\Sigma,j,\nu) \to (X,J)$ is said to be \defined{unobstructed} if $\rd_{u,j}\delbar_{\nu;J}$ is surjective. 
\end{definition}

Henceforth,
to simplify notation,
we will often drop some or all of the subscripts $j$, $\nu$, $J$ from the maps defined above.

\subsection{Transversality for simple maps}

Throughout the remainder of this section,
$(X,\omega)$ is a compact symplectic manifold of dimension $2n \geq 6$ and we only consider pseudo-holomorphic maps from smooth Riemann surfaces. 

\begin{definition}
  \label{Def_JEmb}
  \label{Def_JEmbStar}
  Denote by $\sJ_\emb(X, \omega) \subset \sJ(X, \omega)$ the subspace of those almost complex structures compatible with $\omega$ for which the following hold:
  \begin{enumerate}
  \item
    there are no simple $J$--holomorphic maps with negative index,
  \item
    every simple $J$--holomorphic map with $\ind(u) < 2n-4$ is an embedding, and
  \item 
    every pair of simple $J$--holomorphic maps $u_1$, $u_2$ satisfying
    \begin{equation*}
      \ind(u_1) + \ind(u_2) < 2n-4
    \end{equation*}
    either have disjoint images or are related by a reparametrization.
  \end{enumerate}
  Denote by $\sJ_\emb^\star(X,\omega) \subset \sJ_\emb(X,\omega)$ the subset of those $J$ for which,
  moreover,
  \begin{enumerate}[resume]
  \item
    every simple $J$--holomorphic map is unobstructed.
    \qedhere
  \end{enumerate}
\end{definition}

\begin{definition}
  \label{Def_OneParameterFamilies}
  Given $J_0, J_1 \in \sJ(X,\omega)$,
  denote by $\sJ(X,\omega;J_0,J_1)$ the space of smooth paths $(J_t)_{t\in[0,1]}$ in $\sJ(X,\omega)$ from $J_0$ and $J_1$.
  Given $J_0, J_1\in\sJ_\emb^\star(X,\omega)$,
  denote by $\sJ_\emb^\star(X,\omega,J_0,J_1)$ the subset of those $(J_t)_{t\in[0,1]} \in \sJ(X,\omega;J_0,J_1)$ such that for every $t \in [0,1]$:
  \begin{enumerate}
  \item
    $J_t \in \sJ_\emb(X,\omega)$ and
  \item
    if $u \co (\Sigma,j) \to (X,J_t)$ is a simple $J_t$--holomorphic map,
    then either:
    \begin{enumerate}
    \item 
      $\coker \rd_{u,j}\delbar_{J_t}=\set{0}$ or
    \item
      $\dim\coker \rd_{u,j}\delbar_{J_t} = 1$
      and the map $\ker\rd_{u,j}\delbar_{J_t} \to \coker\rd_{u,j}\delbar_{J_t}$ defined by
      \begin{equation*}
        \xi
        \mapsto
        \pr\paren*{
          \left.\frac{\rd}{\rd s}\right|_{s=t}
          \rd_{u,j}\delbar_{J_s} \xi
        },
      \end{equation*}
      with $\pr \co \Omega^{0,1}(\Sigma,u^*TX) \to\coker\rd_{u,j}\delbar_{J_t}$ denoting the canonical projection,
      is surjective.
      \qedhere
    \end{enumerate}
  \end{enumerate}
\end{definition}

\begin{prop}
  \label{Prop_UnobstructedModuliSpaces}
  Let $A \in H_2(X,\Z)$ and $g \in \N_0$.
  \begin{enumerate}
  \item  
    For every $J\in\sJ_\emb^\star(X,\omega)$ the moduli space $\sM_{A,g}^\star(X,J)$ is an oriented smooth manifold of dimension
    \begin{equation*}
      2\Inner{c_1(X,\omega),A} + 2(n-3)(1-g).
    \end{equation*}
  \item
    For every pair $J_0, J_1 \in\sJ_\emb^\star(X,\omega)$ and $(J_t)_{t\in[0,1]}\in\sJ_\emb^\star(X,\omega;J_0,J_1)$ the moduli space
    \begin{align*}
      \sM_{A,g}^\star\paren*{X,(J_t)_{t\in[0,1]}}
      \coloneq
      \coprod_{t\in[0,1]}\sM_{A,g}^\star(X,J_t), 
    \end{align*}    
    is an oriented smooth manifold with boundary 
    \begin{equation*}
      \sM_{A,g}^\star(X,J_1) \amalg - \sM_{A,g}^\star(X,J_0).
    \end{equation*}
  \end{enumerate}
\end{prop}

This is a consequence of the Implicit Function Theorem;
see \cite[Theorem 3.1.6 and Theorem 3.1.7]{McDuff2012}.
The orientation on the moduli spaces is obtained by trivalizing the determinant line bundle of the family of operators $\rd_{u,j}\delbar_J$;
see \cite[Proof of Theorem 3.1.6, Remark 3.2.5, Appendix A.2]{McDuff2012}.
If the moduli space is zero-dimensional, that is: a discrete set,
then every $[u] \in \sM_{A,g}^\star(X,J)$ is assigned a sign 
\begin{equation*}
  \sign[u] \in \set{+1,-1}.
\end{equation*}
The signed count of $\sM_{A,g}^\star(X,J)$ is then
\begin{equation*}
  \# \sM_{A,g}^\star(X,J)
  \coloneq
  \sum_{[u] \in \sM_{A,g}^\star(X,J)} \sign[u].
\end{equation*}

\begin{prop}
  \label{Prop_Transversality}
  ~
  \begin{enumerate}
  \item
    $\sJ_\emb^\star(X, \omega) \subset \sJ(X, \omega)$ is residual.
  \item
    For every pair $J_0, J_1 \in \sJ_\emb^\star(X,\omega)$,
    $\sJ_\emb^\star(X,\omega;J_0,J_1) \subset \sJ(X,\omega;J_0,J_1)$ is residual.
  \end{enumerate}
\end{prop}

The proof is a standard application of the Sard--Smale theorem;
cf. \cites[Theorem 1.2]{Oh2009}[Proposition A.4]{Ionel2018}[Sections 3.2 and 6.3]{McDuff2012}. 
Some details of the proof will be reviewed in the proof of \autoref{Prop_TransversalityPseudoCycles}.

\subsection{\texorpdfstring{$J$}{J}--holomorphic maps  with constraints}
\label{Sub_JHolomorphicMapsWithConstraints}

\begin{definition}
  Let $\Lambda \in \N$.
  A \defined{$J$--holomorphic map with $\Lambda$ marked points} is a $J$--holomorphic map $u\co (\Sigma,j) \to (X,J)$ together with $\Lambda$ distinct labeled points $z_1,\ldots,z_\Lambda \in \Sigma$.
  
  The \defined{reparametrization} of $(u;z_1,\ldots,z_\Lambda)$ by $\phi \in \Diff(\Sigma)$ is the $J$--holomorphic map with $\Lambda$ marked points $\phi_*(u;z_1,\ldots,z_\Lambda) \coloneq (u\circ\phi^{-1};\phi(z_1),\ldots,\phi(z_\Lambda))$.
  
  A $J$--holomorphic map $(u;z_1,\ldots,z_\Lambda)$ with $\Lambda$ marked points is said to be \defined{simple} if $u$ is simple.
\end{definition}

\begin{definition}
  \label{Def_ModuliOfPointedMaps}
  Given $A\in H_2(X,\Z)$, $g \in \N_0$, $\Lambda \in \N$, and $J\in\sJ(X,\omega)$,
  the \defined{moduli space of simple $J$--holomorphic maps with $\Lambda$ marked points} representing $A$ and of genus $g$ is the set
  \begin{equation*}
    \sM_{A,g,\Lambda}^\star(X,J)
  \end{equation*}
  of equivalence classes $J$--holomorphic maps $u\co (\Sigma,j) \to (X,J)$ with $\Lambda$ marked points $z_1,\ldots,z_\Lambda$ up to reparametrization with
  \begin{equation*}
    u_*[\Sigma] = A \qandq
    g(\Sigma) = g.
  \end{equation*}
  Define the \defined{evaluation map} $\ev \co \sM_{A,g,\Lambda}^\star(X,J) \to X^\Lambda$
  by
  \begin{equation*}
    \ev([u;z_1,\ldots,z_\Lambda])
    \coloneq
    (u(z_1),\ldots,u(z_\Lambda)).
    \qedhere
  \end{equation*} 
\end{definition}

\begin{remark}
  Given two maps $f \co X \to Z$ and $g \co Y \to Z$, the \defined{fiber product} is
  \begin{equation*}
    X \fibprod{f}{g} Y \coloneq (f\times g)^{-1}(\Delta)
  \end{equation*}
  with $\Delta \subset Z \times Z$ denoting the diagonal. 
  If $X$, $Y$, $Z$ are smooth manifolds and $f$ and $g$ are transverse smooth maps, then $X \fibprod{f}{g} Y$ is a submanifold of $X\times Y$ of dimension $\dim(X) + \dim(Y) - \dim(Z)$.
\end{remark}

Let $\paren*{f_\lambda\co V_\Lambda \to X}_{\lambda=1}^\Lambda$ be a $\Lambda$--tuple of pseudo-cycles in general position such that
\begin{equation*}
  \codim(f_\lambda) \coloneq \dim X - \dim V_\lambda
\end{equation*}
is even and positive for every $\lambda$. 
The following discussion assumes some familiarity with the notions of pseudo-cycle, pseudo-cycle cobordism, and pseudo-cycle transversality.
In particular,
we make use of the following facts,
which are discussed in \autoref{Sec_Pseudocycles}:
\begin{enumerate}[(a)]
\item
  For every $\lambda \in \set{1,\ldots,\Lambda}$,
  there is a manifold $V_\lambda^\del$ of dimension $\dim(V_\lambda)-2$ and a smooth map $f_\lambda^\del \co V_\lambda^\del \to X$ whose image contains the pseudo-cycle boundary $\bdry(f_\lambda)$.
\item
  A smooth map $g \co M \to X$ is said to be transverse to the pseudo-cycle $f_\lambda$ if it is transverse to both $f_\lambda$  and $f_\lambda^\del$ in the usual sense.
\item
  For every $I \subset \set{1,\ldots,\Lambda}$
  the product $\prod_{\lambda \in I} f_\lambda$ is a pseudo-cycle and $f_\lambda^\del$ induces in a natural way a map from a smooth manifold whose image contains $\bdry\paren*{\prod_{\lambda \in I} f_\lambda}$.
\end{enumerate}

In the following,
$f_\lambda^\bullet\co V_\lambda^\bullet \to X$ stands for either $f_\lambda\co V_\lambda \to X$ or $f_\lambda^\del\co V_\lambda^\del \to X$.

\begin{definition}
  Given $A \in H_2(X,\Z)$, $g \in \N_0$, and $J \in \sJ(X,\omega)$,
  set
  \begin{equation*}
    \sM_{A,g}^\star(X,J;f_1^\bullet,\ldots,f_\Lambda^\bullet)  
    \coloneq
    \sM_{A,g,\Lambda}^\star(X,J) \fibprod{\ev}{f_1^\bullet\times\cdots\times f_\Lambda^\bullet} V_1^\bullet\times\cdots\times V_\Lambda^\bullet.
  \end{equation*}
  The \defined{expected dimension} of $\sM_{A,g}^\star(X,J;f_1,\ldots,f_\Lambda)$ is defined to be
  \begin{equation*}
    \vdim \sM_{A,g}^\star(X,J;f_1,\ldots,f_\Lambda)
    \coloneq
    2\inner{c_1(X,\omega)}{A}
    + (n-3)(2-2g)
    + \sum_{\lambda=1}^\Lambda \paren*{2-\codim(f_\lambda)}.
    \qedhere
  \end{equation*}
\end{definition}

The following are analogues of \autoref{Def_JEmb} and \autoref{Def_OneParameterFamilies} in the setting of $J$--holomorphic maps with constraints. 

\begin{definition}
  \label{Def_JEmbPseudoCycles}
  Denote by $\sJ_\emb(X,\omega;f_1,\ldots,f_\Lambda) \subset \sJ(X,\omega)$ the subset of those almost complex structures $J$ compatible with $\omega$ for which the following conditions hold for every $A,A_1,A_2 \in H_2(X,\Z)$, $g,g_1,g_2 \in \N_0$, and $I,I_1,I_2 \subset \set{1,\ldots,\Lambda}$ with $I_1\cap I_2 = \emptyset$:
  \begin{enumerate}
  \item 
    \label{It_JEmbPseudoCyclesNonnegativeIndex}
    if $\vdim\sM_{A,g}^\star\paren[\big]{X,J;(f_\lambda^\bullet)_{\lambda\in I}} < 0$,
    then $\sM_{A,g}^\star\paren[\big]{X,J;(f_\lambda^\bullet)_{\lambda\in I}} = \emptyset$;
  \item 
      \label{It_JEmbPseudoCyclesEmbedded}
    if $\vdim\sM_{A,g}^\star\paren[\big]{X,J;(f_\lambda^\bullet)_{\lambda\in I}}  < 2n-4$,
    then every $J$--holomorphic map underlying an element of $\sM_{A,g}^\star\paren[\big]{X,J;(f_\lambda^\bullet)_{\lambda\in I}}$ is an embedding;
    and
  \item
     \label{It_JEmbPseudoCyclesDisjointImages}
    if $\vdim\sM_{A_1,g_1}^\star\paren[\big]{X,J;(f_\lambda^\bullet)_{\lambda\in I_1}} + \vdim\sM_{A_2,g_2}^\star\paren[\big]{X,J;(f_\lambda^\bullet)_{\lambda\in I_2}} < 2n-4$,
    then every pair of every $J$--holomorphic maps underlying elements of $\sM_{A_1,g_1}^\star\paren[\big]{X,J;(f_\lambda^\bullet)_{\lambda\in I_1}}$ and $\sM_{A_2,g_2}^\star\paren[\big]{X,J;(f_\lambda^\bullet)_{\lambda\in I_2}}$ either have disjoint images or are related by a reparametrization.
  \end{enumerate}
  Denote by $\sJ_\emb^\star(X,\omega;f_1,\ldots,f_\Lambda)$ the subset of those elements of $\sJ_\emb(X,\omega;f_1,\ldots,f_\Lambda)$ for which, moreover:
  \begin{enumerate}[resume]
  \item \label{It_JEmbPseudoCyclesUnobstructed}
    every simple $J$--holomorphic map is unobstructed, and
  \item
     \label{It_JEmbPseudoCyclesEvaluationMaps}
     for every $A \in H_2(X,\Z)$, $g \in \N$, and $I \subset \set{1,\ldots,\Lambda}$, the pseudo-cycle $\prod_{\lambda \in I} f_\lambda$ is transverse to $\ev \co \sM_{A,g,\abs{I}}^\star(X,J) \to X^{\abs{I}}$ in the sense of \autoref{Def_PseudocycleTransversality}.
     \qedhere
  \end{enumerate}
\end{definition}

\begin{definition}
  \label{Def_JEmbStarPseudoCycles}
  Given $J_0, J_1 \in \sJ_\emb^\star(X,\omega;f_1,\ldots,f_\Lambda)$, denote by $\sJ_\emb^\star(X,\omega;f_1,\ldots,f_\Lambda; J_0,J_1)$ the space of smooth paths $(J_t)_{t\in[0,1]}$ in $\sJ(X,\omega)$ from $J_0$ and $J_1$ such that for every $t\in[0,1]$:
  \begin{enumerate} 
  \item
    $J_t \in \sJ_\emb(X,\omega;f_1,\ldots,f_\Lambda)$, 
  \item if $u \co (\Sigma,j) \to (X,J_t)$ is a simple $J_t$--holomorphic map, then either:
    \begin{enumerate}
    \item
      $\coker\rd_{u,j}\delbar_{J_t} = \set{0}$ or
    \item
      $\dim\coker \rd_{u,j}\delbar_{J_t} = 1$ and the map $\ker\rd_{u,j}\delbar_{J_t} \to \coker\rd_{u,j}\delbar_{J_t}$ defined by
      \begin{equation*}
        \xi
        \mapsto
        \pr\paren*{
          \left.\frac{\rd}{\rd s}\right|_{s=t}
          \rd_{u,j}\delbar_{J_s} \xi
        },
      \end{equation*}
      with $\pr \co \Omega^{0,1}(\Sigma,u^*TX) \to\coker\rd_{u,j}\delbar_{J_t}$ denoting the canonical projection,
      is surjective;
      in particular, for every $A \in H_2(X,\Z)$, $g \in \N$, and $k\in \N$ the moduli space
      \begin{equation*}
        \sM_{A,g,k}^\star\paren*{X,(J_t)_{t\in[0,1]}}
        \coloneq
        \bigsqcup_{t\in[0,1]}\sM_{A,g,k}^\star(X,J_t) 
      \end{equation*}
      is an oriented smooth manifold with boundary $\sM_{A,g,k}^\star(X,J_1) \amalg -\sM_{A,g,k}^\star(X,J_0)$,
    \end{enumerate}
    and
  \item
    for every $A \in H_2(X,\Z)$, $g \in \N$, and $I \subset \set{1,\ldots,\Lambda}$
    the pseudo-cycle $\prod_{\lambda \in I} f_\lambda$ is transverse to the evaluation map $\ev \co \sM_{A,g,\abs{I}}^\star\paren*{X,A;(J_t)_{t\in[0,1]}} \to X^{\abs{I}}$ in the sense of \autoref{Def_PseudocycleTransversality}.
    \qedhere
  \end{enumerate} 
\end{definition}

The next two results are analogues of \autoref{Prop_UnobstructedModuliSpaces} and \autoref{Prop_Transversality}.

\begin{prop}
  Let $A \in H_2(X,\Z)$ and $g \in \N_0$.
  \label{Prop_UnobstructedModuliSpacesPseudoCycles}
  \begin{enumerate}
  \item  
    For every $J\in\sJ_\emb^\star(X,\omega;f_1,\ldots,f_\Lambda)$
    the moduli space $\sM_{A,g}^\star\paren[\big]{X,J;f_1^\bullet,\ldots,f_\Lambda^\bullet}$ is an oriented smooth manifold of dimension
    \begin{equation*}
      \vdim \sM_{A,g}^\star\paren[\big]{X,J;f_1^\bullet,\ldots,f_\Lambda^\bullet}.
    \end{equation*}
  \item
    For every pair $J_0, J_1 \in\sJ_\emb^\star(X,\omega;f_1,\ldots,f_\Lambda)$ and $(J_t)_{t\in[0,1]}\in\sJ_\emb^\star(X,\omega;f_1,\ldots,f_\Lambda;J_0,J_1)$ the moduli space
    \begin{equation*}     
      \sM_{A,g}^\star\paren*{X,(J_t)_{t\in[0,1]});f_1^\bullet,\ldots,f_\Lambda^\bullet}
      \coloneq
      \coprod_{t\in[0,1]}\sM_{A, g}^\star(X, J_t; f_1^\bullet,\ldots,f_\Lambda^\bullet)
    \end{equation*}
    is an oriented smooth manifold with boundary 
    \begin{equation*}
      \sM_{A,g}^\star\paren[\big]{X,J;f_1^\bullet,\ldots,f_\Lambda^\bullet}
      \amalg
      -\sM_{A,g}^\star\paren[\big]{X,J;f_1^\bullet,\ldots,f_\Lambda^\bullet}.
    \end{equation*} 
  \end{enumerate}
\end{prop}

\begin{prop}
  \label{Prop_TransversalityPseudoCycles}
  ~
  \begin{enumerate}
  \item
    $\sJ_\emb^\star(X,\omega;f_1,\ldots,f_\Lambda) \subset \sJ(X, \omega)$ is residual.
  \item
    For every pair $J_0, J_1 \in \sJ_\emb^\star(X,\omega;f_1,\ldots,f_\Lambda)$,  $\sJ_\emb^\star(X,\omega;f_1,\ldots,f_\Lambda;J_0,J_1) \subset \sJ(X,\omega;J_0,J_1)$ is residual.
  \end{enumerate}
\end{prop}

\begin{proof}
  We will prove the first part; the proof of the second part is similar.
  It follows from \autoref{Prop_TransversalityEvaluationMaps}, proved in \autoref{Sec_TransversalityEvaluationMaps}, that the set of $J\in\sJ(X,\omega)$ satisfying conditions \autoref{It_JEmbPseudoCyclesUnobstructed} and  \autoref{It_JEmbPseudoCyclesEvaluationMaps} from \autoref{Def_JEmbPseudoCycles} is residual. 
  Note that condition \autoref{It_JEmbPseudoCyclesUnobstructed} implies condition \autoref{It_JEmbPseudoCyclesNonnegativeIndex}.
  To prove that condition \autoref{It_JEmbPseudoCyclesEmbedded} is also satisfied by a generic $J$, consider the evaluation map
  \begin{equation*}
      \ev \co \sM_{A,g,2}^\star\paren[\big]{X,J;(f_\lambda^\bullet)_{\lambda\in I}} \to X^2.
  \end{equation*}
  If $\ev$ is transverse to the diagonal $X =\Delta\subset X^2$, then $\ev^{-1}(\Delta)$ is a submanifold of codimension 
  \begin{equation*}
   \dim\sM_{A,g,2}^\star\paren[\big]{X,J;(f_\lambda^\bullet)_{\lambda\in I}} - 2n = \dim\sM_{A,g}^\star\paren[\big]{X,J;(f_\lambda^\bullet)_{\lambda\in I}} - (2n-4).
  \end{equation*} 
  Therefore, if $\dim\sM_{A,g}^\star\paren[\big]{X,J;(f_\lambda^\bullet)_{\lambda\in I}}  < 2n-4$, then $\ev^{-1}(\Delta)$ is empty and two distinct maps in $\sM_{A,g}^\star\paren[\big]{X,J;(f_\lambda^\bullet)_{\lambda\in I}}$ have disjoint images. 
  By \autoref{Prop_TransversalityEvaluationMaps}, the set of $J$ for which the map $\ev$ is transverse to the diagonal $X \hookrightarrow X^2$ is residual.  
  This shows that the set of $J$ satisfying condition \autoref{It_JEmbPseudoCyclesEmbedded} from \autoref{Def_JEmbPseudoCycles}  is residual. 
  In the same way we conclude that the set of $J$ satisfying condition \autoref{It_JEmbPseudoCyclesDisjointImages} is residual. 
\end{proof}

The following will be important for relating moduli spaces defined using cobordant pseudocycles. 
Let $F \co W \to X$ be a cobordism between two pseudo-cycles $f_1^0$ and $f_1^1$ in $X$, and let $F^\del \co W^\del \to X$ be such that $\bdry(F)$ is contained in the image of $F^\del$;  see \autoref{Def_Pseudocycle} for the notation and the definition of a pseudo-cycle cobordism. 
In what follows, $F^\bullet$ denotes either $F$ or $F^\del$.
Let $f_2, \ldots, f_\Lambda$ be pseudo-cycles in $X$ such that $F, f_2, \ldots, f_\Lambda$ are in general position, as in \autoref{Def_PseudocyclesInGeneralPosition}.

Given $J \in \sJ(X,\omega)$ and a subset $I \subset \set{2,\ldots,\Lambda}$, set
\begin{equation*}
  \sM_{A,g}^\star(X,J; F^\bullet, (f_\lambda^\bullet)_{\lambda\in  I}) \coloneq \sM_{A,g,\abs{I}+1}^\star(X,J) \fibprod{\ev}{F^\bullet \times \prod_{\lambda \in I} f_\lambda^\bullet} W^\bullet \times \prod_{\lambda \in I} V_\lambda^\bullet.
\end{equation*}

\begin{definition}
  \label{Def_JEmbPseudoCycleCobordism}
  Let
  \begin{equation*}
    \sJ_\emb^\star(X,\omega;F,f_2,\ldots,f_\Lambda) \subset \sJ_\emb^\star(X,\omega;f_1^0,f_2,\ldots,f_\Lambda) \cap \sJ_\emb^\star(X,\omega;f_1^1,f_2,\ldots,f_\Lambda)
  \end{equation*}
  be the subset of those $J$ for which the following conditions hold for every $A,A_1,A_2 \in H_2(X,\Z)$, $g,g_1,g_2 \in \N_0$, and $I,I_1,I_2 \subset \set{2,\ldots,\Lambda}$ with $I_1\cap I_2 = \emptyset$:
  \begin{enumerate}
  \item
    if $\vdim\sM_{A,g}^\star\paren[\big]{X,J;F^\bullet,(f_\lambda^\bullet)_{\lambda\in I}}  < 2n-4$,
    then every $J$--holomorphic map underlying an element of $\sM_{A,g}^\star\paren[\big]{X,J;F^\bullet,(f_\lambda^\bullet)_{\lambda\in I}}$ is an embedding;
  \item
    if $\vdim\sM_{A_1,g_1}^\star\paren[\big]{X,J;F^\bullet,(f_\lambda^\bullet)_{\lambda\in I_1}} + \vdim\sM_{A_2,g_2}^\star\paren[\big]{X,J;(f_\lambda^\bullet)_{\lambda\in I_2}} < 2n-4$,
    then every pair of $J$--holomorphic maps underlying elements of 
    \begin{equation*}
      \sM_{A_1,g_1}^\star\paren[\big]{X,J;F^\bullet,(f_\lambda^\bullet)_{\lambda\in I_1}}
      \qandq
      \sM_{A_2,g_2}^\star\paren[\big]{X,J;F^\bullet, (f_\lambda^\bullet)_{\lambda\in I_2}}
    \end{equation*}
    either have disjoint images or are related by a reparametrization;
    and
  \item
    for every $A \in H_2(X,\Z)$, $g \in \N$, and $I \subset \set{2,\ldots,\Lambda}$, the pseudo-cycle $F \times \prod_{\lambda \in I} f_\lambda$ is transverse as pseudo-cycle with boundary to $\ev \co \sM_{A,g,\abs{I}+1}^\star(X,J) \to X^{\abs{I}+1}$ in the sense of \autoref{Def_PseudocycleTransversality}.
    \qedhere
  \end{enumerate}
\end{definition}

It follows from this definition that for every $J \in \sJ_\emb^\star(X,\omega;F,f_2,\ldots,f_\Lambda)$,  $F \times f_2 \times \ldots \times f_\Lambda$ is transverse as pseudo-cycle cobordism to $\ev \co \sM_{A,g,\Lambda}^\star(X,J) \to X^\Lambda$.   
In this case,  $\sM_{A,g}^\star(X,J; F,f_2,\ldots,f_\Lambda)$ is an oriented cobordism from $\sM_{A,g}^\star(X,J; f_1^0,f_2,\ldots,f_\Lambda)$ to $\sM_{A,g}^\star(X,J; f_1^1, f_2,\ldots,f_\Lambda) $.

\begin{prop}
  \label{Prop_TransversalityPseudoCycleCobordism}
  $\sJ_\emb^\star(X,\omega;F,f_2,\ldots,f_\Lambda)$ is residual in $\sJ(X,\omega)$.   
\end{prop}

The proof is almost identical to that of \autoref{Prop_TransversalityPseudoCycles}.


\section{Gromov compactness}
\label{Sec_GromovCompactness}

\subsection{Deformations of nodal Riemann surfaces}

\begin{definition}
  Let $\sX$ and $A$ be complex manifolds and let $\pi \co \sX \to A$ be a holomorphic map.
  Set $n \coloneq \dim_\C A$ and suppose that $\dim_\C \sX = n + 1$.
  A critical point $x \in \sX$ of $\pi$ is called \defined{nodal} if there are holomorphic coordinates at $x$ and holomorphic coordinates at $\pi(x)$ with respect to which
  \begin{equation*}
    \pi(z,w,t_2,\ldots,t_n) = (zw,t_2,\ldots,t_n).
  \end{equation*}
  A \defined{nodal family} is a surjective, proper, holomorphic map $\pi\co \sX \to A$ between complex manifolds of dimension $\dim_\C \sX = \dim_\C A + 1$ such that every critical point of $\pi$ is nodal.
  The \defined{fiber} over $a \in A$ is the nodal Riemann surface $(\Sigma,j,\nu)$ associated with the nodal curve $\pi^{-1}(a)$.
  \emph{Henceforth, we engage in the abuse of notation of identifying $\pi^{-1}(a)$ with $(\Sigma,j,\nu)$.}
\end{definition}

\begin{definition}
  Let $(\Sigma,j,\nu)$ be a nodal Riemann surface.
  A \defined{deformation} of $(\Sigma,j,\nu)$ is a nodal family $\pi\co \sX \to A$, together with a base-point $\star \in A$, and a nodal, biholomorphic map $\iota\co (\Sigma,j,\nu) \to \pi^{-1}(\star)$.  
\end{definition}

\begin{definition}
  Let $(\Sigma,j,\nu)$ be a nodal Riemann surface and let $(\pi\co \sX \to A,\star,\iota)$ and $(\rho\co \sY\to B,\dagger,\kappa)$ be two deformations of $(\Sigma,j,\nu)$.
  A pair of holomorphic maps  $\Phi\co \sX\to \sY$ and $\phi\co A \to B$ forms a \defined{morphism} $(\Phi,\phi)\co (\rho,\star,\iota) \to (\sY,\dagger,\kappa)$ of deformations if
  \begin{equation*}
    \phi(\star) = \dagger, \quad
    \rho\circ\Phi = \phi\circ\pi, \quad
    \Phi\circ\iota = \kappa
  \end{equation*}
  and for every $a \in A$ the restriction $\Phi\co \pi^{-1}(a) \to \rho^{-1}(\phi(a))$ induces a nodal, biholomorphic map.
\end{definition}

\begin{definition}
  \label{Def_VersalDeformation}
  A deformation $(\rho\co \sY\to B,\dagger,\kappa)$ of $(\Sigma,j,\nu)$ is \defined{(uni)versal} if for every deformation ($\pi\co \sX \to A,\star,\iota)$ of $(\Sigma,j,\nu)$ there exists an open neighborhood $U$ of $\star \in A$ and a (unique) morphism of deformations $(\pi\co \pi^{-1}(U) \to U,\star,\iota) \to (\rho,\dagger,\kappa)$.
\end{definition}

A nodal Riemann surface $(\Sigma,j,\nu)$ admits a universal deformation if and only if it is stable \cites{Deligne1969}[Chapter XI Theorem 4.3]{Arbarello2011}[Theorem A]{Robbin2006}.
However, every nodal Riemann surface $(\Sigma,j,\nu)$ admits a versal deformation.
This will be discussed in detail in \autoref{Sec_VersalDeformations}.

\begin{definition}
  \label{Def_Framing}
  Let $(\pi\co \sX\to A,\star,\iota)$ be a deformation of a nodal Riemann surface $(\Sigma,j,\nu)$.
  Denote by $S$ the nodal set of $\nu$.
  A \defined{framing} of $(\pi,\star,\iota)$ is a smooth embedding $\Psi\co \paren*{\Sigma\setminus S}\times A \to \sX$ such that
  \begin{equation*}
    \pi\circ\Psi = \pr_A \qandq
    \Psi(\cdot,\star) = \iota.
    \qedhere
  \end{equation*}
\end{definition}

\subsection{The Gromov topology}
\label{Sub_GromovTopology}

Let $X$ be a manifold and denote by $\sH(X)$ the set of almost Hermitian structures $(J,h)$ on $X$ equipped with the $C^\infty$ topology.
The following defines a topology on
\begin{equation*}
  \overline\sM_{A,g}(X)
  \coloneq
  \coprod_{(J,h) \in \sH(X)} \overline\sM_{A,g}(X,J).
\end{equation*}

\begin{definition}
  Let $(X,J,h)$ be an almost Hermitian manifold.
  Let $(\Sigma,j,\nu)$ be a closed, nodal Riemann surface.
  The \defined{energy} of a nodal map $u \co (\Sigma,\nu) \to X$ is 
  \begin{equation*}
    E(u) \coloneq \frac{1}{2}\int_\Sigma \abs{\rd u}^2 \,\vol.
  \end{equation*}
  Implicit in this definition is a choice of Riemannian metric in the conformal class determined by $j$.
  The right-hand side, however, is independent of this choice.
\end{definition}

\begin{definition}
  \label{Def_GromovTopology}
  Let $(J_0,h_0) \in \sH(X)$.
  Let $[u_0\co (\Sigma_0,j_0,\nu_0) \to (X,J_0)] \in \overline\sM_{A,g}(X,J_0)$,
  let $(\pi \co \sX \to A, \star,\iota)$ be a versal deformation of $(\Sigma_0,j_0,\nu_0)$,
  let $\Psi$ be a framing of $(\pi,\star,\iota)$,
  let $\epsilon > 0$.
  let $U_0 \subset C^\infty(\Sigma_0\setminus S,X)$ be an open neighborhood of $u_\infty|_{\Sigma_0\setminus S}$ in the $C_\loc^\infty$ topology, and
  let $U_\sH$ be an open neighborhood of $(J_0,h_0)$ in $\sH(X)$.
  Define
  \begin{equation*}
    \sU(u_0,\epsilon,U_0,U_\sH)
    \subset
    \overline\sM_{A,g}(X)
  \end{equation*}
  to be the subset of the equivalences classes of nodal $J$--holomorphic maps $u\co (\Sigma,j,\nu) \to (X,J)$ satisfying the following:
  \begin{enumerate}
  \item
    $(J,h) \in U_\sH$,
  \item
    $\abs{E(u) - E(u_0)} < \epsilon$,
  \item
    $(\Sigma,j,\nu) = \pi^{-1}(a)$ for some $a\in A$, and
  \item   
    $\tilde u \coloneq u\circ\Psi(\cdot,a) \in U_0$,
  \end{enumerate}
  The \defined{Gromov topology} on $\overline\sM_{A,g}(X)$ is the coarsest topology with respect to which every subset of the form $\sU(u_0,\epsilon,U_0,U_\sH)$ is open.
\end{definition}

In practice, it is more convenient to use the notion of \defined{Gromov convergence} defined on the level of nodal maps.

\begin{definition}
  \label{Def_GromovConvergence}
  Let $(X,J_\infty,h_\infty)$ be an almost Hermitian manifold and let $(J_k,h_k)_{k \in \N}$ be a sequence of almost Hermitian structures on $X$ converging to $(J_\infty,h_\infty)$ in the $C^\infty$ topology.
  For every $k \in \N\cup\set{\infty}$ let $u_k \co (\Sigma_k,j_k,\nu_k) \to (X,J_k)$ be a nodal $J_k$--holomorphic map.
  Denote by $S$ the nodal set of $(\Sigma_\infty,\nu_\infty)$.
  The sequence $(u_k,j_k)_{k\in\N}$ \defined{Gromov converges} to $(u_\infty,j_\infty)$ if
  \begin{enumerate}
  \item
    \label{Def_GromovConvergence_Energy}
    $\lim_{k\to \infty} E(u_k) = E(u_\infty)$
    and
  \item
    \label{Def_GromovConvergence_Maps}
    there are:
    \begin{enumerate}
    \item a deformation $(\pi\co \sX\to A,a_\infty,\iota_\infty)$ of $(\Sigma_\infty,j_\infty,\nu_\infty)$ together with a framing $\Psi$,
    \item
      a sequence $(a_k)_{k\in\N}$ in $A$ converging to $a_\infty$, and
    \item
     a nodal, biholomorphic map $\iota_k\co (\Sigma_k,j_k,\nu_k) \to \pi^{-1}(a_k)$  for every sufficiently large $k\in\N$,
    \end{enumerate}
    such that the sequence of maps
    \begin{equation*}
      \tilde u_k \coloneq u_k \circ \iota_k^{-1}\circ \Psi(\cdot,a_k)\circ\iota_\infty  \co \Sigma_\infty\setminus S \to X
    \end{equation*}
    converges to $u_\infty|_{\Sigma_\infty\setminus S}$ in the $C_\loc^\infty$ topology.
    \qedhere
  \end{enumerate}
\end{definition}

\begin{remark}  
  If $(\pi,\star,\iota)$ is a versal deformation of $(\Sigma_\infty,j_\infty,\nu_\infty)$ and $\Psi$ is a framing of this deformation,
  then for every sequence $(u_k,j_k)_{k\in\N}$ which Gromov converges to $(u_\infty,j_\infty)$ the deformation in \autoref{Def_GromovConvergence} can be assumed to be $(\pi,\star,\iota)$ and the framing can be assumed to be $\Psi$.
  This is an immediate consequence of the definition of a versal deformation.
\end{remark}

\begin{theorem}[{\citet{Gromov1985}; see also \cites{Parker1993}{Ye1994}{Hummel1997}[Chapters 4 and 5]{McDuff2012}}]
  \label{Thm_GromovCompactness}
  Let $(X,J_\infty,h_\infty)$ be a closed almost Hermitian manifold and let $(J_k,h_k)_{k \in \N}$ be a sequence of almost Hermitian structures on $X$ converging to $(J_\infty,h_\infty)$ in the $C^\infty$ topology.
  For every $k \in \N$ let 
  \begin{equation*}
    u_k \co (\Sigma_k,j_k,\nu_k) \to (X,J_k)
  \end{equation*}
  be a stable nodal $J$--holomorphic map.
  Denote by $\#\pi_0(\Sigma_k)$ the number of connected components of $\Sigma_k$. 
  If
  \begin{equation*}
    \limsup_{k\to\infty} \#\pi_0(\Sigma_k)  < \infty, \quad
    \limsup_{k\to\infty} p_a(\Sigma_k,\nu_k) < \infty, \qandq
    \limsup_{k\to\infty} E(u_k) < \infty,
  \end{equation*}
  then there exists a stable nodal $J_\infty$--holomorphic map $u_\infty \co (\Sigma_\infty,j_\infty,\nu_\infty) \to (X,J_\infty)$ and a subsequence of $(u_k,j_k)_{k\in\N}$ which Gromov converges to $(u_\infty,j_\infty)$.
  The limit $(u_\infty,j_\infty,\mu_\infty)$ is unique up to automorphism.
\end{theorem}

\begin{remark}
  The Gromov topology  on $\overline\sM_{A,g}(X)$ is metrizable, which can be seen as follows. 
  \autoref{Thm_GromovCompactness} implies that it is Hausdorff and the projection map $\overline\sM_{A,g}(X) \to \sH(X)\times\R$ is is proper and closed. 
  This implies, in particular, that $\overline\sM_{A,g}(X)$ is a regular topological space. 
   (In general, if $A$ is a Hausdorff space,  $B$ is a regular space, and $f \co A \to B$ is a proper, closed map, then $A$ is a regular space.)
  Urysohn's metrization theorem says that a second countable, Hausdorff, regular space is metrizable.
\end{remark}
%
%
%
Henceforth,
let $(X,\omega)$ be a symplectic manifold.
The set $\sJ(X,\omega)$ of almost complex structures compatible with $\omega$ injects into $\sH(X)$.

\begin{prop}[{\citet{Gromov1985}; see also \cite[Lemma 2.2.1]{McDuff2012}}]
 Let $(X,\omega)$ be a symplectic manifold and $J\in\sJ(X,\omega)$. 
  Let $(\Sigma,\nu,j)$ be a closed, nodal Riemann surface.
  For every nodal map $u\co(\Sigma,\nu)\to X$ 
  \begin{equation*}
    E(u) \geq \inner{u^*[\omega]}{[\Sigma]},
  \end{equation*}
  and the equality holds if and only if $u$ is $J$--holomorphic. 
  Here $E$ is to be understood with respect to the Riemannian metric $h = \omega(\cdot,J\cdot)$ on $X$. 
\end{prop}

Set
\begin{equation*}
  \overline\sM_{A,g}(X,\omega)
  \coloneq
  \coprod_{J \in \sJ(X,\omega)} \overline\sM_{A,g}(X,J).
\end{equation*}
By the above energy identity,
in the symplectic context,
\autoref{Thm_GromovCompactness} is equivalent to the map
\begin{equation*}
  \pi\co
  \overline\sM_{A,g}(X,\omega)
  \to \sJ(X,\omega)
\end{equation*}
being proper.

\subsection{Behavior near the vanishing cycles}
\label{Sec_BehaviorNearTheVanishingCycles}

The results of this subsection will be important for proving the surjectivity of the gluing construction in \autoref{Sec_KuranishiModel}. 
Assume the situation of \autoref{Def_GromovConvergence}.
By condition \autoref{Def_GromovConvergence_Energy} for every $\delta > 0$ there are $K \in \N_0$ and $r > 0$ such that for every $k \geq K$
\begin{equation*}
  E\paren[\big]{u_k|_{N_k^r}} \leq \delta
\end{equation*}
with
\begin{equation}
  \label{Eq_NeckRegion}
  N_k^r \coloneq \Sigma_k \setminus \set*{ \Psi(z,a_k) : z \in \Sigma_0 \text{ with } d(z,S) \geq r }.
\end{equation}
The subset $N_k^r$ can be partitioned into regions $N_{k,n}^r$ corresponding to the nodes $n \in S$.
If $n$ is not smoothed out in $\Sigma_k$,
then the corresponding region is biholomorphic to
\begin{equation*}
  B_1(0) \amalg B_1(0)
\end{equation*}
with $\nu_k$ identifying the origins.
If $n$ is smoothed out in $\Sigma_k$,
then the corresponding region is biholomorphic to
\begin{equation*}
  S^1\times(-L_k,L_k)
\end{equation*}
with $\lim_{k\to \infty} L_k = \infty$.

The behavior of $J$--holomorphic maps from such domains and with small energy can be understood through the following two results.

\begin{lemma}[{\cite[Lemma 4.3.1]{McDuff2012}}]
  \label{Lem_EpsilonRegularity}
  Let $(X,J,h)$ be an almost Hermitian manifold.
  There is a constant $\delta = \delta(X,J,h) > 0$, depending continuously on $(J,h)$, such that for every $r > 0$ the following holds.
  If $u\co (B_{2r}(0),i) \to (X,J)$ is a $J$--holomorphic map with
  \begin{equation*}
    E(u) \leq \delta,
  \end{equation*}
  then
  \begin{equation*}
    \Abs{\rd u}_{L^\infty(B_r(0))} \leq cr^{-1}E(u)^{1/2}.
  \end{equation*}
\end{lemma}

\begin{lemma}[{\cite[Lemma 4.7.3]{McDuff2012}}]
  \label{Lem_EnergyDecayOnCylinders}
  Let $(X,J,h)$ be an almost Hermitian manifold.
  For every $\mu \in (0,1)$ there are constants: $\delta = \delta(X,J,h,\mu) > 0$,  depending continuously on $(J,h)$,  and $c = c(\mu) > 0$ such that for every $L > 0$ the following holds.
  If $u\co (S^1\times(-L,L),j_\cyl) \to (X,J)$ is a $J$--holomorphic map with
  \begin{equation*}
    E(u) \leq \delta,
  \end{equation*}
  then for every $\ell \in (0,L)$
  \begin{equation*}
    E\paren*{u|_{S^1\times(L+\ell,L-\ell)}} \leq ce^{-2\mu\paren{L-\ell}}E(u).
  \end{equation*}
  and for every $\theta \in S^1$ and $\ell \in [-L+1,L-1]$  
  \begin{equation*}
    \abs{\rd u}(\theta,\ell)
    \leq
    ce^{-\mu\paren*{L-\abs{\ell}}}E(u)^{1/2}.
  \end{equation*}
\end{lemma}

\begin{proof}
  The first assertion is \cite[Lemma 4.7.3]{McDuff2012}.
  The second assertion follows from the first by \autoref{Lem_EpsilonRegularity}. 
\end{proof}

The following is an important consequence of the previous two lemmas. 

\begin{prop}
  \label{Prop_SmallEnergyNeck}
  Let $\paren*{u_k \co (\Sigma_k,j_k,\nu_k) \to (X,J_k)}_{k\in\N}$ be a sequence of nodal pseudo-holomorphic maps which Gromov converges to $u_\infty \co (\Sigma_\infty,j_\infty,\nu_\infty) \to (X,J_\infty)$. 
  Denote by $S$ the nodal set of $(\Sigma_\infty,\nu_\infty)$ and let $N^r_k$ be as in \autoref{Eq_NeckRegion}. 
  For every $\delta > 0$ there are $r>0$ and $K\in\N$ such that for every $k\geq K$ and $n \in S$
  \begin{equation*}
    u_k(N^r_{k,n}) \subset B_\delta(u_\infty(n));
  \end{equation*}
  in particular, provided $\delta$ is sufficiently small,
  \begin{equation*}
    (u_k)_*[\Sigma_k] = (u_\infty)_*[\Sigma_\infty].
  \end{equation*}
\end{prop}



\section{Versal deformations of nodal Riemann surfaces}
\label{Sec_VersalDeformations}

The purpose of this section is to construct a versal deformation of a nodal Riemann surface in a rather explicit manner.

\subsection{Deformations of nodal curves}

Let us briefly review parts of the deformation theory of nodal curves in the complex analytic category.
For further details and proofs we refer the reader to \cite[Chapter XI Section 3]{Arbarello2011}.
A thorough discussion of deformation theory in the algebraic category can be found in \cite{Hartshorne2010}. 

\begin{definition}
  \label{Def_DeformationOfNodalCurve}
  Let $C$ be a nodal curve.
  A \defined{deformation} of $C$ consists of
  \begin{enumerate}
    \item
     a proper flat\footnote{%
     A morphism $f \co A \to B$ between two analytic spaces is \defined{flat} if it makes the stalk $\sO_{A,a}$ into a flat $\sO_{B,f(a)}$--module for every $a$, that is: tensoring by $\sO_{A,a}$ preserves short exact sequences of $\sO_{B,f(a)}$--modules.}
     morphism $\pi \co \sX \to A$ between analytic spaces such that every fiber of $\pi$ is a nodal curve,
    \item
    a base-point $\star \in A$, and
    \item
      an isomorphism $\iota\co C \to \pi^{-1}(\star)$.
      \qedhere
  \end{enumerate}
\end{definition}

\begin{prop}
  Every nodal family $\pi\co \sX \to A$ is flat.
  In particular,
  a deformation of a nodal Riemann surface $(\Sigma,j,\nu)$ is also a deformation of the associated nodal curve $C$.
\end{prop}

\begin{definition}
  Let $C$ be a nodal Riemann surface and let $(\pi\co \sX \to A,\star,\iota)$ and $(\rho\co \sY\to B,\dagger,\kappa)$ be two deformations of $C$.
  A pair of analytic maps $\Phi\co \sX\to \sY$ and $\phi\co A \to B$ forms a \defined{morphism} $(\Phi,\phi)\co (\rho,\star,\iota) \to (\sY,\dagger,\kappa)$ of deformations if
  \begin{equation*}
    \phi(\star) = \dagger, \quad
    \rho\circ\Phi = \phi\circ\pi, \quad
    \Phi\circ\iota = \kappa,
  \end{equation*}
  and for every $a \in A$ the restriction $\Phi\co \pi^{-1}(a) \to \rho^{-1}(\phi(a))$ induces an analytic isomorphism.
\end{definition}

\begin{definition}
  A deformation $(\rho\co \sY \to B,\dagger,\kappa)$ of $C$ is \defined{(uni)versal} if for every deformation ($\pi\co \sX \to A,\star,\iota)$ of $(\Sigma,j,\nu)$ there exists an open neighborhood $U$ of $\star \in A$ and a (unique) morphism of deformations $(\pi\co \pi^{-1}(U) \to A,\star,\iota) \to (\rho,\dagger,\kappa)$.
\end{definition}

\begin{definition}
  Denote by $\C[\epsilon]/\epsilon^2$ the ring of dual numbers and set $D \coloneq \Spec\paren*{\C[\epsilon]/\epsilon^2}$.
  A \defined{first order deformation} is a deformation over $D$.
\end{definition}

Let $C$ be a nodal curve.
Every first order deformation $(\pi\co \sX \to D,0,\iota)$ of $C$ induces a short exact sequence
\begin{equation*}
  0 \to \sO_C \iso \pi^*\Omega_D^1 \to \Omega_{\sX}^1 \otimes \sO_C \xrightarrow{\iota^*} \Omega_C^1 \to 0.
\end{equation*}
The map $\pi^*\Omega_D^1 \to \Omega_{\sX}^1 \otimes \sO_C$ is given by pulling-back forms  from $D$ to $\sX$ and $\Omega_{\sX}^1 \otimes \sO_C \to \Omega_C^1$ is given by restricting forms on $\sX$ to $C$. 
The extension class $\delta \in \Ext^1(\Omega_C^1,\sO_C)$ of this sequence depends on the first order deformation only up to isomorphism of deformations.
Indeed,
two first order deformation of $C$ are isomorphic if and only if they yield the same extension class $\delta$.

\begin{definition}
  Let $C$ be a nodal curve and let $(\pi\co \sX\to A,\star,\iota)$ be a deformation of $C$.
  Every $v \in T_\star A$ corresponds to an analytic map $\phi\co D \to A$ mapping $0$ to $\star$.
  The pullback of $(\pi,\star,\iota)$ via $\phi$ is a first order deformation.
  Denote by $\delta(v) \in \Ext^1(\Omega_C^1,\sO_C)$ the corresponding extension class.
  The map $\delta\co T_\star A \to \Ext^1(\Omega_C^1,\sO_C)$ thus defined is called the \defined{Kodaira--Spencer map}.
\end{definition}

It is instructive to analyze $\Ext^1(\Omega_C^1,\sO_C)$ more closely.
The local-to-global $\Ext$ spectral sequence yields a short exact sequence
\begin{equation*}
  0 \to H^1(C,\sHom(\Omega_C^1,\sO_C)) \to \Ext^1(\Omega_C^1,\sO_C) \to H^0(C,\sExt^1(\Omega_C^1,\sO_C)) \to 0.
\end{equation*}
This can be interpreted in terms of the normalization $\pi\co \tilde C \to C$ as follows.
Denote by $S$ the set of nodes of $C$ and set $\tilde S \coloneq \pi^{-1}(S)$.
It can be shown that
\begin{equation*}
  \sHom(\Omega_C^1,\sO_C) = \pi_*\sT_{\tilde C}(-\tilde S);
  \quad\text{hence:}\quad
  H^1(C,\sHom(\Omega_C^1,\sO_C)) = H^1(\tilde C,\sT_{\tilde C}(-\tilde S)).
\end{equation*}
The space $H^1(\tilde C,\sT_{\tilde C}(-\tilde S))$ parametrizes the deformations of the marked curve $(\tilde C,\tilde S)$;
that is: deformations of $\tilde C$ which fix $\tilde S$ point-wise.
The sheaf $\sExt^1(\Omega_C^1,\sO_C)$ is supported on the nodes of $C$:
\begin{equation*}
  \sExt^1(\Omega_C^1,\sO_C)
  =
  \bigoplus_{n \in S} \Ext^1(\Omega_{C,n}^1,\sO_{C,n}).
\end{equation*}
For every $n \in \tilde S$ and $\set{n_1,n_2} = \pi^{-1}(n)$
\begin{equation*}
  \Ext^1(\Omega_{C,n}^1,\sO_{C,n})
  = T_{n_1}\tilde C\otimes T_{n_2}\tilde C.
\end{equation*}
By considering the deformation $\set{ zw = \epsilon} $ of the node $\set{ zw = 0 }$,
the space $\Ext^1(\Omega_{C,n}^1,\sO_{C,n})$ can be seen to parametrize smoothings of the node $n$.
The above discussion show that to first order all deformations of $C$ arise from smoothing nodes and deforming its normalization while fixing the points mapping to the nodes.
In the following we construct a deformation of $C$ which induces all of these deformations to first order.

\subsection{Smoothing nodal Riemann surfaces}
\label{Sec_SmoothingNodalRiemannSurfaces}

Let $(\Sigma_0,j_0,\nu_0)$ be a closed, nodal Riemann surface with nodal set $S$.
Let $g_0$ be a Riemannian metric on $\Sigma_0$ in the conformal class determined by $j_0$ and such that there is a constant $R_0 > 0$ such that for every $n \in S$ the restriction of $g_0$ to $B_{4R_0}(n)$ is flat and for every $n_1,n_2 \in S$ the balls $B_{4R_0}(n_1)$ and $B_{4R_0}(n_2)$ are disjoint.
For every $n \in S$ define the holomorphic charts $\phi_n\co B_{4R_0}(n)\subset T_n\Sigma_0 \to \Sigma_0$ by
\begin{equation*}
  \phi_n(v) \coloneq \exp_n(v)
\end{equation*}
and define $r_n \co \Sigma_0 \to [0,\infty)$ by
\begin{equation*}
  r_n(z)
  \coloneq
  \max\set{d(n,z),4R_0}.
\end{equation*}

Given a pair of complex vector spaces $V$ and $W$,
denote by $\sigma \co V\otimes W \to W\otimes V$ the isomorphism defined by $\sigma(v\otimes w) \coloneq w\otimes v$.

\begin{definition}
 \label{Def_SmoothingParameter}
 A \defined{smoothing parameter} for $(\Sigma_0,j_0,\nu_0)$ is an element
 \begin{equation*}
    \tau
    =
    (\tau_n)_{n \in S}
    \in
    \prod_{n \in S} T_n\Sigma_0\otimes_\C T_{\nu(n)}\Sigma_0
  \end{equation*}
  such that for every $n \in S$
  \begin{equation*}
    \tau_{\nu(n)} = \sigma(\tau_n)
    \qandq
    \abs{\tau_n} < R_0^2.
  \end{equation*}
  Given a smoothing parameter $\tau$,
  for every $n \in S$
  set
  \begin{equation*}
    \epsilon_n \coloneq \abs{\tau_n} \qandq
    \hat\tau_n \coloneq \tau_n/\abs{\tau_n} \qifq \tau_n \neq 0;
  \end{equation*}
  furthermore,
  set
  \begin{equation*}
    \epsilon \coloneq \max\set{ \epsilon_n : n \in S }.
    \qedhere
  \end{equation*}
\end{definition}

Henceforth,
let $\tau = (\tau_n)_{n\in S}$ be a smoothing parameter for $(\Sigma_0,j_0,\nu_0)$.

\begin{definition}
  \label{Def_IotaTau}
  Set
  \begin{equation*}
    A_\tau
    \coloneq
    \set*{
      w \in \Sigma_0
      :
      \epsilon_n/R_0 < r_n(w) < R_0
      \text{ for some }
      n \in S
      \text{ with }
      \epsilon_n \neq 0
    }
  \end{equation*}
  and denote by $\iota_\tau\co A_\tau \to A_\tau$ the biholomorphic map characterized by
  \begin{equation*}
    \phi_{\nu(n)}^{-1}\circ \iota_\tau\circ\phi_{n}(v) \otimes v
    =
    \tau_n
  \end{equation*}
  for every $n \in S$ and $v \in T_n\Sigma_0$ with $\epsilon_n/R_0 < \abs{v} < R_0$.
\end{definition}

\begin{definition}
  \label{Def_SmoothingNodalRiemannSurface}
  Consider the Riemann surface with boundary
  \begin{align*}
    \Sigma_\tau^\circ
    \coloneq
    \set[\big]{
      z \in \Sigma_0
      :
      r_n(z) \geq \epsilon_n^{1/2}
      \text{ for every }
      n \in S
    }.
  \end{align*}
  Denote by $\sim_\tau$ the equivalence relation on $\Sigma_\tau^\circ$ generated by
  identifying the boundary components via $\iota_\tau$.
  The quotient
  \begin{equation*}
    \Sigma_\tau \coloneq \Sigma_\tau^\circ/\sim_\tau
  \end{equation*}
  is a closed surface.
  The restrictions of the complex structure $j_0$ and the nodal structure $\nu_0$ to $\Sigma_\tau^\circ$ descends to a complex structure $j_\tau$ and a nodal structure $\nu_\tau$ on $\Sigma_\tau$.
  The nodal Riemann surface $(\Sigma_\tau,j_\tau,\nu_\tau)$ is called the \defined{partial smoothing} of $(\Sigma_0,j_0,\nu_0)$ associated with $\tau$.
\end{definition}

\begin{remark}
  The above construction smooths out every node with $\epsilon_n > 0$.
  In particular,
  if all of the  $\epsilon_n$ are positive,
  then $\nu_\tau$ is the trivial nodal structure and $(\Sigma_\tau,j_\tau,\nu_\tau)$ is simply the Riemann surface $(\Sigma_\tau,j_\tau)$.
\end{remark}

\begin{definition}
  Denote by $\Delta$ the space of smoothing parameters for $(\Sigma_0,j_0,\nu_0)$.
  Set
  \begin{equation*}
    \sX
    \coloneq
    \set*{
      (z,\tau) \in \Sigma_0\times\Delta
      :
      z \in \Sigma_\tau^\circ
    }/\sim
  \end{equation*}
  with $(z_1,\tau_1)\sim(z_2,\tau_2)$ if and only if $\tau_1 = \tau_2$ and $z_1 \sim_{\tau_1} {z_2}$ or $z_1,z_2 \in S$, $\nu(z_1) = z_2$, and $\epsilon_{z_1} = \epsilon_{z_2} = 0$.
  Denote by $\pi\co \sX \to \Delta$ the canonical projection.
\end{definition}

The following example will be important in the proof of \autoref{Thm_LimitConstraints} in \autoref{Sec_ProofOfLimitConstraints}. 

\begin{example}
  \label{Ex_SphericalBubble}
  Let $(\Sigma_1,j_1,\nu_1)$ and $(\Sigma_2,j_2,\nu_2)$ be two nodal Riemann surfaces with nodal sets $S_1$ and $S_2$.
  Given $x_i \in \Sigma_i \setminus S_i$ for $i=1,2$, we define a new nodal Riemann surface $(\Sigma_\clubsuit,j_\clubsuit,\nu_\clubsuit)$ by setting $\Sigma_\clubsuit = \Sigma_1\amalg\Sigma_2$ and $\nu_\clubsuit(x_1) = x_2$ (and otherwise agreeing with $\nu_1$ and $\nu_2$).
  The nodal set of $(\Sigma_\clubsuit,j_\clubsuit,\nu_\clubsuit)$ is
  \begin{equation*}
    S_\clubsuit =  \set{x_1,x_2} \amalg S_1 \amalg S_2.
  \end{equation*}
  Accordingly, the space of smoothing parameters for $(\Sigma_\clubsuit,j_\clubsuit,\nu_\clubsuit)$ is
  \begin{equation*}
    \Delta_\clubsuit = \Delta_0 \times \Delta_1 \times \Delta_2,
  \end{equation*} 
  where $\Delta_0$ is an open neighborhood of zero in $T_{x_1} \Sigma_1 \otimes_\C T_{x_2}\Sigma_2$. 
  
  Suppose now that $(\Sigma_1,j_1,\nu_1)$ is a tree of spheres. 
  It is easy to see that for every smoothing parameter $\tau_\clubsuit = (\tau_0,\tau_1,\tau_2)$ such that $\tau_0 \neq 0$ and $\tau_{1,n} \neq 0$ for every node $n\in S_1$,  there is a biholomorphism
  \begin{equation*}
   (\Sigma_{\clubsuit, \tau_\clubsuit}, j_{\tau_\clubsuit}, \nu_{\tau_\clubsuit} ) \iso (\Sigma_{2,\tau_2}, j_{\tau_2}, \nu_{\tau_2}).
  \end{equation*}
 In particular, if $\tau_0 \neq 0$, $\tau_{1,n} \neq 0$ for every $n\in S_1$, and $\tau_2 = 0$, there is a biholomorphism
   \begin{equation*}
   (\Sigma_{\clubsuit, \tau_\clubsuit}, j_{\tau_\clubsuit}, \nu_{\tau_\clubsuit} ) \iso (\Sigma_2, j_2, \nu_2). \qedhere
  \end{equation*}
\end{example}

\begin{prop}
  \label{Prop_PartialSmoothingsFormNodalFamily}
  $\sX$ is a smooth manifold and the complex structure on $\Sigma_0\times\Delta$ induces a complex structure on $\sX$ such that $\pi$ is a nodal family and for every $\tau \in \Delta$ the canonical map $\Sigma_\tau \to \pi^{-1}(\tau)$ induces a nodal, biholomorphic map $\iota_\tau\co (\Sigma_\tau,j_\tau,\nu_\tau) \to \pi^{-1}(\tau)$.
\end{prop}

\begin{proof}
  It suffices to consider the local model of a node $C_0 \coloneq \set{ (z,w) \in \C^2 : zw = 0 }$.
  $\tilde \sX\coloneq \set{ (z,w,\tau) \in \C^2\times \C : zw = \tau }$ is a complex manifold.
  The map $\tilde\pi\co \tilde \sX \to \C$ defined by $\tilde \pi(z,w,\tau) \coloneq t$ has only nodal critical points and its fiber over $0$ is $C_0$.
  The nodal Riemann surface associated with $C_0$ is $\Sigma_0 = \C\amalg\C$ with the complex structure $i$ on both components and the nodal structure which interchanges the origins of the components.
  The partial smoothing defined in \autoref{Def_SmoothingNodalRiemannSurface} is
  \begin{equation*}
    \Sigma_\tau
    =
    \paren[\big]{
      \set[\big]{ z \in \C : \abs{z} \geq \abs{\tau}^{1/2} }
      \amalg
      \set[\big]{ w \in \C : \abs{{w}} \geq \abs{\tau}^{1/2} }
    }\big/\!\sim_\tau.
  \end{equation*}
  The map $\Phi\co \sX \to \tilde \sX$ defined by $\Phi([z],\tau) \coloneq (z,\tau/z,\tau)$ and $\Phi([w],\tau) \coloneq (\tau/z,z,\tau)$ is biholomorphic.
  This implies the assertion.
\end{proof}

\subsection{Construction of a versal deformation}
\label{Sec_VersalDeformations_Construction}

Let $(\Sigma_0,j_0,\nu_0)$ be a nodal Riemann surface with nodal set $S$.
Denote by $\sJ(\Sigma_0)$ the space of almost complex structures on $\Sigma_0$ and by $\Diff_0(\Sigma_0,\nu_0)$ the group of diffeomorphism of $\Sigma_0$ which are isotopic to the identity and commute with $\nu_0$.
Denote by
\begin{equation*}
  \sT \coloneq \sJ(\Sigma_0)/\Diff_0(\Sigma_0,\nu_0)
\end{equation*}
the corresponding Teichmüller space.
This is a complex manifold and there is an open neighborhood $\Delta_1$ of $0 \in \C^{\dim_\C \sT}$ together with a map $\jmath\co \Delta_1 \to \sJ(\Sigma_0)$ such that:
\begin{enumerate}
\item
  $\jmath(0) = j_0$,
\item
  for every $\sigma \in \Delta_1$
  the almost complex structure $\jmath(\sigma)$ agrees with $j_0$ in some neighborhood $U$ of $S$, and
\item
  the map $[\jmath] \co \Delta_1 \to \sT$ is an embedding.
\end{enumerate}
For every $\sigma \in \Delta_1$
set
\begin{equation*}
  \Sigma_{\sigma,0} \coloneq \Sigma_0, \quad
  j_{\sigma,0} \coloneq \jmath(\sigma), \qandq
  \nu_{\sigma,0} \coloneq \nu_0.
\end{equation*}
Choose a family of metrics $(g_{\sigma,0})_{\sigma \in \Delta_1}$ whose restriction to the neighborhood $U$ of $S$ is independent of $\sigma$ and such that $g_{\sigma,0}$ is in the conformal class determined by $j_{\sigma,0}$ for every $\sigma \in \Delta_1$.
Let $R_0 > 0$ be such that the conditions at the beginning of \autoref{Sec_SmoothingNodalRiemannSurfaces} hold for every $\sigma \in \Delta_1$ and $B_{4R_0}(S) \subset U$.

Denote by $\Delta_2$ the space of elements
\begin{equation*}
  \tau
  =
  (\tau_n)_{n \in S}
  \in
  \prod_{n \in S} T_n\Sigma_0\otimes_\C T_{\nu(n)}\Sigma_0
\end{equation*}
such that for every $n \in S$
\begin{equation*}
  \tau_{\nu(n)} = \sigma(\tau_n)
  \qandq
  \abs{\tau_n} < R_0^2.
\end{equation*}

\begin{definition}
  Set $\Delta \coloneq \Delta_1\times \Delta_2$.
  Set
  \begin{equation*}
    \sX
    \coloneq
    \set*{
      (z;\sigma,\tau) \in \Sigma_0\times\Delta_1\times\Delta_2
      :
      z \in \Sigma_{\sigma,\tau}^\circ
    }/\sim
  \end{equation*}
  with $(z_1;\sigma_1,\tau_1)\sim(z_2;\sigma_2,\tau_2)$ if and only if $\sigma_1 = \sigma_2$, $\tau_1 = \tau_2$ and $z_1 \sim_{\tau_1} {z_2}$ or $\tau_1 = \tau_2$ and $z_1 \sim_{\tau_1} {z_2}$ or $z_1,z_2 \in S$, $\nu(z_1) = z_2$, and $\epsilon_{z_1} = \epsilon_{z_2} = 0$.
  Denote by $\pi\co \sX \to \Delta$ the canonical projection.
\end{definition}

\begin{prop}
  \label{Prop_PartialSmoothingsInFamiliesFormNodalFamily}
  $\sX$ is a smooth manifold and the complex structure on $\Sigma_0\times\Delta$ induces a complex structure on $\sX$ such that $\pi$ is a nodal family and for every $(\sigma,\tau) \in \Delta$ the canonical map $\Sigma_{\sigma,\tau} \to \pi^{-1}(\sigma,\tau)$ induces a nodal, biholomorphic map $\iota_{\sigma,\tau}\co (\Sigma_{\sigma,\tau},j_{\sigma,\tau},\nu_{\sigma,\tau}) \to \pi^{-1}(\sigma,\tau)$.
\end{prop}

\begin{theorem}[{cf. \cite[Chapter XI Theorem 3.17 and Section 4]{Arbarello2011}}]
  \label{Thm_VersalFamily}
  Set $\star \coloneq (0,0)$ and $\iota \coloneq \iota_{0,0}$.
  The deformation $\paren*{\pi,\star,\iota}$ of $\paren*{\Sigma_0,j_0,\nu_0}$ is versal.
\end{theorem}

\begin{proof}  
  Denote by $C$ the nodal curve associated with $(\Sigma_0,j_0,\nu_0)$.
  It is proved in \cite[Chapter XI Theorem 3.17]{Arbarello2011} that the Kodaira--Spencer map $\delta\co T_0\Delta_1\times T_0\Delta_2 \to \Ext^1(\Omega_C^1,\sO_C)$ is an isomorphism.
  This implies that the deformation is versal.
  Indeed,
  $C$ has some versal family $(\rho\co \sY \to B,\dagger,\kappa)$ for which the Kodaira--Spencer map is an isomorphism.
  Therefore,
  after possibly shrinking $\Delta$,
  there exists a morphism of deformations $(\Phi,\phi)\co (\pi,\star,\iota) \to (\rho,\dagger,\kappa)$.
  Since both Kodaira--Spencer maps are isomorphism,
  after possibly shrinking $\Delta$,
  $\phi$ is a holomorphic embedding.
  Therefore,
  after possibly shrinking both $\Delta$ and $A$,
  both deformations become isomorphic.
\end{proof}

To define a framing of the deformation $(\pi,\star,\iota)$,
choose an increasing, smooth function $\eta\co [0,2] \to [1,2]$ such that
\begin{equation*}
  \eta(0) = 1 \qandq
  \eta(r) = r \qforeveryq 3/2 \leq r \leq 2.
\end{equation*}

\begin{definition}
  \label{Def_ConcreteFraming}
  Define the framing $\Psi\co \Sigma_0\setminus S\times \Delta \to \sX$ of $(\pi,\star,\iota)$ by
  \begin{equation*}
    \Psi(z;\sigma,\tau)
    \coloneq
    \begin{cases}
     \psi_n(z) & \text{if } r_n(z) \leq 2\epsilon_n^{1/2} \text{ for some } n \in S \\
      (z;\sigma,\tau) & \text{otherwise},
    \end{cases}
  \end{equation*}
  with $\psi_n(z)$ defined by
  \begin{equation*}
    \psi_n(z) = \phi_n\paren*{\eta\paren*{r_n(z)/\epsilon_n^{1/2}}\cdot \frac{\phi_n^{-1}(z)}{r_n(z)/\epsilon_n^{1/2}}}.
  \end{equation*}  
  Observe that $\psi_n$ is defined so that:
  \begin{enumerate}
  \item $r_n(\psi_n(z)) \geq \epsilon_n^{1/2}$, so that indeed $\psi_n(z)$ corresponds to a point in $\sX$,
  \item $\psi_n$ defines an embedding from a punctured neighborhood of $n$ in $\Sigma_0$ to $\sX$,
  \item $\psi_n(z) = z$ when $r_n(z) \geq 3/2 \epsilon_n^{1/2}$, so that $\Psi$ is continuous.
    \qedhere
  \end{enumerate} 

\end{definition}

\begin{remark}
  \label{Rmk_DecompositionOfSurface}
  Let $(\sigma,\tau) \in \Delta$ and $r \in (2\epsilon^{1/2},R_0)$.
  Set
  \begin{equation*}
    \Sigma_0^r \coloneq \set{ z \in \Sigma_0 : r_n(z) \geq r \text{ for every } n \in S }.
  \end{equation*}
  Denote by
  \begin{equation*}
    N_{\sigma,\tau}^r
    \coloneq
    \Sigma_{\sigma,\tau} \setminus \Psi\paren*{\Sigma_0^r\times\set{(\sigma,\tau)}}
  \end{equation*}
  the part of $\Sigma_{\sigma,\tau}$ not covered by $\Sigma_0^r$ under the framing $\Psi$,
  cf. \autoref{Sec_BehaviorNearTheVanishingCycles}.
  By construction,
  \begin{equation*}
    N_{\sigma,\tau}^r
    =
    \bigcup_{n \in S} N_{\sigma,\tau;n}^r
  \end{equation*}
  with
  \begin{equation*}
    N_{\sigma,\tau;n}^r
    =
    N_{\sigma,\tau;\nu(n)}^r
    \coloneq
    \set[\big]{
      z \in \Sigma_{\sigma,\tau}^\circ
      :
      r_n(z) < r \text{ or } r_{\nu(n)}(z) < r
    }/\sim_\tau.
  \end{equation*}
  If $\epsilon_n = 0$,
  then $N_{\sigma,\tau;n}^r$ is biholomorphic to
  \begin{equation*}
    B_r(0) \amalg B_r(0)
  \end{equation*}
  and the nodal structure $\nu_{\sigma,\tau}$ identifies the two origins.
  If $\epsilon_n \neq 0$,
  then $N_{\sigma,\tau;n}^r$ is biholomorphic to
  \begin{equation*}
    \set{ z \in \C : \epsilon_n/r < \abs{z} < r }
    \iso
    S^1 \times \paren[\big]{-\log\paren{r\epsilon_n^{-1/2}}, \log\paren{r\epsilon_n^{-1/2}}}.
    \qedhere
  \end{equation*}
\end{remark}



\section{Smoothing nodal \texorpdfstring{$J$}{J}--holomorphic maps}
\label{Sec_SmoothingNodalMaps}

The purpose of this section is to prove  \autoref{Thm_LimitConstraints}.
The strategy is to construct a Kuranishi model for a Gromov neighborhood of $u_\infty$ and analyze the obstruction map.
This idea goes back to \citet{Ionel1998} and has been used by \citet{Zinger2009,Niu2016} to give a sharp compactness results for genus one and two pseudo-holomorphic maps.

Throughout this section,
fix a smooth function $\chi \co [0,\infty) \to [0,1]$ with
\begin{equation}
  \label{Eq_CutoffFunction}
  \chi|_{[0,1]} = 1 \qandq
  \chi|_{[2,\infty)} = 0
\end{equation}
and, moreover, fix $p \in (2,\infty)$.

\subsection{Riemannian metrics on smoothings}

Let $(\Sigma_0,j_0,\nu_0)$ be a nodal Riemann surface with nodal set $S$.
  Denote by $g_0$ a Riemannian metric on $\Sigma_0$ as at the beginning of \autoref{Sec_SmoothingNodalRiemannSurfaces}.
In \autoref{Sec_SmoothingNodalRiemannSurfaces} we discussed the construction of a smoothing $\Sigma_\tau$ of $\Sigma_0$ for every smoothing parameter $\tau$. 
In this section we construct a Riemannian metric $g_\tau$ on $\Sigma_\tau$ which is uniformly equivalent to the metric $g_0$ on $\Sigma_0$ in the smoothing region. 
This property will be useful for proving estimates in the construction of a smoothing of a nodal pseudo-holomorphic map from $\Sigma_0$. 

\begin{definition}
  \label{Def_GTau}
  Given a smoothing parameter $\tau$,
  let $\Sigma_\tau^\circ$ be as in \autoref{Def_SmoothingNodalRiemannSurface}.
  Recall that for every node $n\in S$ we have the corresponding  number $\epsilon_n = \abs{\tau_n}$, the size of the smoothing parameter at $n$, and local radial coordinate $r_n \co \Sigma_0 \to [0,\infty)$. 

  Define the Riemannian metric $g_\tau^\circ$ on $\Sigma_\tau^\circ$ by
  \begin{equation*}
    g_\tau^\circ
    \coloneq
    g_0
    +
    \sum_{n \in S}
    \chi\paren*{\frac{r_n}{2\epsilon_n^{1/2}}}
    \cdot
    \paren[\big]{
      \epsilon_n\cdot(\phi_n)_*\paren*{r^{-2}\rd r\otimes \rd r + \theta\otimes\theta}
      -
      g_0
    }
  \end{equation*}
  with $r$ denoting the distance from origin in $T_n\Sigma_0 \iso \C$ and $\theta = - \rd r \circ j_0$. 
  Since the Riemannian metric $r^{-2}\rd r\otimes \rd r + \theta\otimes\theta$ on $\C^*$ is invariant under the involution $z \mapsto \epsilon/z$,
  $g_\tau^\circ$ descends to a Riemannian metric $g_\tau$ on $\Sigma_\tau$.
\end{definition}

\begin{prop}
  \label{Prop_MetricsAreUniformlyComparable}
  There is a constant $c > 1$ such that for every nodal Riemann surface and every smoothing parameter $\tau$
  \begin{equation*}
    c^{-1}g_0 < g_\tau^\circ < cg_0.
  \end{equation*}
\end{prop}

\begin{proof}
  Let $n \in S$ and let $r$ and $\theta$ be as in \autoref{Def_GTau}.
  On the annulus $\set[\big]{ z \in \Sigma_0: \epsilon_n^{1/2} \leq r_n(z) \leq 4\epsilon_n^{1/2} }$,
  \begin{align*}
    \phi_n^*g_0
    &=
      \rd r\otimes\rd r + r^2\theta\otimes\theta
  \end{align*}
  and, therefore,
  \begin{equation*}
    g_\tau^\circ
    =
    \paren*{F_{\epsilon_n}\circ r_n} \cdot g_0
    \qwithq
    F_{\epsilon_n}(r)
    \coloneq
    1+
    \chi\paren*{\frac{r}{2\epsilon_n^{1/2}}}\cdot \paren[\big]{\epsilon_n r^{-2}-1}.
  \end{equation*}
  This implies the assertion because $c^{-1} < F_{\epsilon_n}(r) < c$ for $\epsilon_n^{1/2} \leq r \leq 4\epsilon_n^{1/2}$.
\end{proof}

Henceforth,
the $L^p$ and $W^{1,p}$ norms of all sections and differential forms on $\Sigma_\tau$ are understood with respect to the metric $g_\tau$.
The above proposition will often be implicitly used to bound these norms by estimating various expressions with respect to $g_0$ over the corresponding region in $\Sigma_0$.

\subsection{Approximate smoothing of nodal \texorpdfstring{$J$}{J}--holomorphic maps}

Throughout the next four sections,
let $(X,J,h)$ be an almost Hermitian manifold,
let $c_u > 0$,
let $u_0\co (\Sigma_0,j_0,\nu_0) \to (X,J)$ be a nodal map,
and let $\tau$ be a smoothing parameter.
Furthermore, choose $g_0$ and $R_0$ as at the beginning of \autoref{Sec_SmoothingNodalRiemannSurfaces}.

\begin{definition}
  \label{Def_ExpInverse}
  For every point $x \in X$,
  denote by $\tilde U_x \subset T_x X$ the segment/injectivity domain and set $U_x \coloneq \exp_x(\tilde U_x)$ and $\frac12 U_x \coloneq \exp_x(\frac12\tilde U_x)$.
  The map $\exp_x \co \tilde U_x \to U_x$ is a diffeomorphism and its inverse is denoted by $\exp_x^{-1}\co U_x \to \tilde U_x$.
\end{definition}

Furthermore, we assume the following.

\begin{hypothesis}
  \label{Hyp_UTau}
  The map $u_0$ and $R_0 > 0$ satisfy
  \begin{equation*}
    \Abs{u_0}_{C^2} \leq c_u \qandq   
    u_0(B_{4R_0}(n)) \subset U_{u_0(n)}
    \qforeveryq n \in S.
  \end{equation*}
\end{hypothesis}

\begin{convention}
  Henceforth,
  constants may depend on $p$, $(\Sigma_0,j_0,\nu_0)$, $(X,J,h)$, $c_u$, and $R_0$, 
  but not on $\tau$.
\end{convention}

\begin{definition}
  \label{Def_TildeUTau}
  For $n \in S$ define $\chi_\tau^n\co \Sigma_\tau \to [0,1]$ by
  \begin{equation*}
    \chi_\tau^n(z)
    \coloneq
    \chi\paren*{\frac{r_n(z)}{R_0}}.
  \end{equation*}
  Let $\iota_\tau$ be the map defined in \autoref{Def_IotaTau}.
  Define $\tilde u_\tau^\circ\co \Sigma_\tau^\circ \to X$ by
  \begin{equation*}
    \tilde u_\tau^\circ(z)
    \coloneq
    \begin{cases}
      \exp_{u_0(n)}\paren[\big]{{\exp_{u_0(n)}^{-1}}\circ u_0(z) + \chi_\tau^n(z)\cdot {\exp_{u_0(n)}^{-1}}\circ u_0(\iota_\tau(z))}
      & \text{if } r_n(z) \leq 2R_0 \\
      u_0(z)
      & \text{otherwise}.
    \end{cases}
  \end{equation*}
  Since $u_0(\nu_0(n)) = u_0(n)$, the restriction of $\tilde u_\tau^\circ$ to
  \begin{equation*}
    \set[\big]{ z \in \Sigma_\tau^\circ : r_n(z) \leq R \text{ for some } n \in S } 
  \end{equation*}
  is invariant under $\iota_\tau$.
  Therefore, $\tilde u_\tau^\circ$ descends to a smooth map
  \begin{equation*}
    \tilde u_\tau\co \Sigma_\tau \to X.
  \end{equation*}
  This map is called the \defined{approximate smoothing} of $u$ associated with $\tau$.  
\end{definition}

\begin{remark}
  This construction differs from that found, for example, in \cites[Section 10.2]{McDuff2012}[Section B.3]{Pardon2016} in which the approximate smoothing is constant in the middle of the neck region.
  The above construction is very similar to that in \cite[Section 2.1]{Gournay2009}.
  It leads to a smaller error term and significantly simplifies the discussions in \autoref{Sec_LeadingOrderTermOfObstructionGhostComponents}.
  Morally, this section analyzes how the interaction between the different components of $u_0$ affects whether $u_0$ can be smoothed or not.
  The constructions in \cites[Section 10.2]{McDuff2012}[Section B.3]{Pardon2016} make it difficult to see such interactions.
\end{remark}

\begin{prop}
  \label{Prop_ErrorEstimate}
  For every nodal map $u_0 \co \Sigma_0 \to X$ (not necessarily $J$--holomorphic)  the map $\tilde u_\tau$ satisfies
  \begin{equation}
    \Abs{\delbar_J(\tilde u_\tau,j_\tau)}_{L^p}
    \leq
    c\Abs{\delbar_J(u_0,j_0)}_{L^p}
    +
    c\epsilon^{\frac12+\frac1p},
  \end{equation}
  with $\epsilon=\max\set{ \epsilon_n : n\in S}$ as in \autoref{Def_SmoothingParameter}. 
\end{prop}

For the proof of this result and for future reference let us observe that for every $k \geq 1$
\begin{equation}
  \label{Eq_AlmostDivergentIntegralEstimate}
  \paren*{
    \int_{\epsilon_n^{1/2} \leq r_n \leq 2R_0}
    r_n^{-kp}
  }^{\frac1p}
  \leq
  \paren*{\frac{2\pi}{kp-2}}^{\frac1p}
  \epsilon_n^{\frac1p-\frac{k}2}.
\end{equation}

\begin{proof}[Proof of \autoref{Prop_ErrorEstimate}]
  The map $\tilde u_\tau^\circ$ agrees with $u_0$ in the region where $r_n \geq 2R_0$ for every $n \in S$.
  Therefore, it suffices to consider the regions where $r_n \leq 2R_0$ for some $n \in S$.
  To simplify notation, identify $U_x$ with $\tilde U_x$ via $\exp_x$ for $x \coloneq u(n)$.
  Here $U_x$ and $\tilde U_x$ are as in \autoref{Def_ExpInverse}.
  Having made this identification,
  in such a region,
  $\tilde u_\tau^\circ$ is given by
  \begin{equation*}
    \tilde u_\tau^\circ = u_0 + \chi_\tau^n \cdot u_0\circ \iota_\tau.
  \end{equation*}
  Note that addition is well-defined since, with respect to the above identification, $u_0$ takes values in an sufficiently small open subset of $T_xX$. 
  Therefore,
  \begin{align*}
    \delbar_J (\tilde u_\tau^\circ,j_\tau)
    &=
      \frac12\paren*{\rd \tilde u_\tau^\circ + J(\tilde u_\tau^\circ)\circ \rd \tilde u_\tau^\circ \circ j} \\
    &=
      \underbrace{\delbar_J (u_0,j_0) + \chi_\tau^n \cdot \delbar_J (u_0\circ\iota_\tau,j_0)}_{\eqcolon \rI} \\
    &\quad
      + 
      \underbrace{\frac12 \paren*{J(\tilde u_\tau^\circ)-J(u_0)}\circ \rd u_0\circ j_0}_{\eqcolon \rII_1}
      +
      \underbrace{\chi_\tau^n\cdot \frac12 \paren*{J(\tilde u_\tau^\circ)-J(u_0\circ \iota_\tau)}\circ \rd (u_0\circ\iota_\tau)\circ j_0}_{\eqcolon \rII_2} \\
    &\quad
      + \underbrace{\delbar \chi_\tau^n \cdot u_0\circ\iota_\tau}_{\eqcolon \rIII}.
  \end{align*}
  
  The $L^p$ norm of the term $\rI$ is controlled by the $L^p$ norm of $\delbar_J(u_0,j_0)$ over the regions of $\Sigma_0$ where $\epsilon_n/2R_0 \leq r_n \leq 2R_0$ for some $n \in S$.
  By Taylor expansion
  \begin{align*}
    \abs{\rII_1}
    \leq
    c \Abs{J}_{C^1}\cdot\abs{u_0\circ \iota_\tau}\cdot\abs{\rd u_0}
    \leq
    c \epsilon_n/r_n
  \end{align*}
  and
  \begin{align*}
    \abs{\rII_2}
    &\leq
      c \Abs{J}_{C^1}\cdot\paren*{\abs{u_0} + (1-\chi_\tau^n)\cdot\abs{u_0\circ \iota}}\cdot\abs{\rd(u_0\circ\iota_\tau)} \\
    &\leq
      c\cdot\paren*{r_n + (1-\chi_\tau^n)\epsilon_n/r_n}\cdot \epsilon_n/r_n^2
      \leq
      c \epsilon_n/r_n.
  \end{align*}
  On $\Sigma_\tau^\circ$, by definition, $r_n \geq \epsilon_n^{1/2}$.
  Therefore and by \autoref{Eq_AlmostDivergentIntegralEstimate},
  \begin{equation*}
    \Abs{\rII_1+\rII_2}_{L^p} \leq c\epsilon_n^{\frac12+\frac1p}.
  \end{equation*}
  
  The term $\rIII$ is supported in the region where $R_0 \leq r_n \leq 2R_0$, whose area is independent of $\epsilon_n$, and satisfies
  \begin{equation*}
    \abs{\rIII} \leq c \abs{\rd \chi_\tau^n } \abs{u_0\circ\iota_\tau} \leq c r_n \cdot (\epsilon_n/r_n) = c\epsilon_n, 
  \end{equation*}
  where in the last inequality we use that, with respect to the identifications introduced earlier, $u_0(0)= 0$ so $\abs{u_0(z)} \leq c\abs{z}$ in a neighborhood of $0$. 
 Therefore,
    \begin{equation*}
    \Abs{\rIII}_{L^p} \leq c \epsilon_n. \qedhere
  \end{equation*}
\end{proof}

\subsection{Fusing nodal vector fields}

The purpose of this section is to introduce the fusing operator.
This operator assigns to every vector field $\xi$ along $u_0$ a vector field $\fuse_\tau(\xi)$ along $\tilde u_\tau$, which agrees with $\xi$ outside the gluing region. 
The construction of the fusing operator makes use of the following local trivializations of $TX$.

\begin{definition}
  \label{Def_Phi}
  For every $x \in X$ and $y \in U_x$ define an isomorphism $\Phi_y = \Phi_y^x \co T_xX \to T_yX$ by
  \begin{equation*}
    \Phi_y^x(v) \coloneq \rd_{\exp_x^{-1}(y)} \exp_x(v)
  \end{equation*}
  As $y$ varies in $U_x$,
  these maps define a trivialization $\Phi = \Phi^x\co U_x \times T_xX \to TX|_{U_x}$.
\end{definition}

The definition of the fusing operator uses a different cutoff function than the definition of the approximate smoothing $u_\tau$.

\begin{definition}
  \label{Def_FuseCutoffFunction}
  For $n \in S$ define $\rho_\tau^n\co \Sigma_\tau \to [0,1]$ by
  \begin{equation*}
    \rho_\tau^n(z)
    \coloneq
    \chi\paren*{\frac{r_n(z)}{\epsilon_n^{1/4}}}.
  \end{equation*}
  Here $\chi$ is the cutoff function  \autoref{Eq_CutoffFunction}; that is: $\rho_\tau^n = 1$ in the region where $r_n \leq \epsilon_n^{1/4}$ and $\rho_\tau^n = 0$ in the region where $r_n \geq 2\epsilon_n^{1/4}$.
\end{definition}

\begin{definition}
  \label{Def_Fuse}
  Define $\fuse_\tau^\circ \co W^{1,p}\Gamma(\Sigma_0,\nu_0;u_0^*TX) \to W^{1,p}\Gamma(\Sigma_\tau^\circ,\nu_0;(\tilde u_\tau^\circ)^*TX)$ by
  \begin{equation*}
    \fuse_\tau^\circ(\xi)(z)
    \coloneq
    \begin{cases}
      \Phi_{\tilde u_\tau([z])}\paren*{
        \Phi_{u_0(z)}^{-1}\xi(z) + \rho_\tau^n(z)\cdot \paren{\Phi_{u_0\circ\iota_\tau (z)}^{-1} \xi(\iota_\tau(z)) - \Phi_{u_0}^{-1}\xi(n)}
      }
      & \text{if } r_{n}(z) \leq 2\epsilon_n^{1/4} \\
      \xi(z)
      & \text{otherwise}.
    \end{cases}
  \end{equation*}  
  In the above formula, $\Phi = \Phi^x$ with $x = u_0(n)$.
  For every $n \in S$ the restriction of $\fuse_\tau^\circ(\xi)$ to
  \begin{equation*}
    \set*{ z\in\Sigma_\tau^\circ : r_n(z) \leq \epsilon_n^{1/4} \text{ for some } n \in S } 
  \end{equation*}
  is invariant under $\iota_\tau$.
  Therefore,
  $\fuse_\tau^\circ$ induces a map
  \begin{equation*}
    \fuse_\tau \co W^{1,p}\Gamma(\Sigma_0,\nu_0;u_0^*TX) \to W^{1,p}\Gamma(\Sigma_\tau,\nu_\tau;\tilde u_\tau^*TX). \qedhere
  \end{equation*}
\end{definition}

The following is a counterpart of \autoref{Prop_ErrorEstimate}.

\begin{prop}
  \label{Prop_DFuse}
  For every $\xi \in W^{1,p}\Gamma(\Sigma_0,\nu_0;u_0^*TX)$
  \begin{equation*}
    \Abs{\fd_{\tilde u_\tau}\fuse_\tau(\xi)}_{L^p}
    \leq
    c\Abs{\fd_{u_0}\xi}_{L^p} +
    c
    \sum_{n\in S} \paren*{\epsilon_n^{\frac{1}{2p}}+\epsilon_n^{\frac{1}{2}-\frac1p}}    
    \Abs{\xi}_{W^{1,p}}.
  \end{equation*}
\end{prop}

The proof requires the following results as a preparation.

\begin{prop}
  \label{Prop_CutOffEstimate}
  For every $n \in S$ and $\xi \in W^{1,p}\Gamma(\Sigma_0,\nu;u_0^*TX)$
  \begin{equation*}
    \Abs{\rd \rho_\tau^n \cdot (\xi\circ \iota_\tau - \xi(n))}_{L^p}
    \leq
    c\epsilon_n^{\frac12 - \frac1p}\Abs{\xi}_{W^{1,p}}.
  \end{equation*}
\end{prop}

\begin{proof}
  Morrey's embedding theorem asserts that $W^{1,p} \hookrightarrow C^{0,1-2/p}$.
  Hence,
  \begin{align*}
    \abs{\xi\circ \iota_\tau(z) - \xi(n)}
    \leq
    c (\epsilon_n/r_n(z))^{1-2/p}\Abs{\xi}_{W^{1,p}}.
  \end{align*}
   The term $\rd\rho_\tau^n$ is supported in the annulus $P_\tau^n = \set{\epsilon_n^{1/4} \leq r_n \leq 2\epsilon_n^{1/4}}$ 
   and satisfies 
  \begin{equation*}
   \abs{\rd\rho_\tau^n} \leq c\epsilon_n^{-1/4}.
   \end{equation*}
   Since the area of $P_\tau^n$ is proportional to $\epsilon_n^{1/2}$,  \begin{equation*}
      \Abs{\rd \rho_\tau^n \cdot (\xi\circ \iota_\tau - \xi(n))}_{L^p} \leq c \epsilon_n^{-1/4} \epsilon_n^{3/4 (1-2/p)} \epsilon_n^{1/2p} \Abs{\xi}_{W^{1,p}} = c\epsilon_n^{\frac12-\frac1p}  \Abs{\xi}_{W^{1,p}}. \qedhere
    \end{equation*}
\end{proof}

\begin{prop}
  \label{Prop_CompareDU1WithDU2}
  Let $U \subset \Sigma_0$ be an open subset.
  Let $u_1,u_2\co U \to U_x$ and set
  \begin{equation*}
    v \coloneq {\exp_x^{-1}}\circ u_2 - {\exp_x^{-1}}\circ u_1.
  \end{equation*}
  For every $\xi \in C^\infty(U,T_xX)$
  \begin{equation*}
    \abs*{\paren*{\Phi_{u_1}\circ\fd_{u_1}\circ\Phi_{u_1}^{-1} - \Phi_{u_2}\circ \fd_{u_2}\circ\Phi_{u_2}^{-1}}\xi}
    \leq
    c\paren*{\abs{v}\abs{\rd\xi}+\abs{\rd v}\abs{\xi}+\abs{\rd u_1}\abs{\xi}\abs{v}}.
  \end{equation*}
\end{prop}

\begin{proof}
  To simplify notation, identify $U_x$ with $\tilde U_x$ via $\exp_x$.
  Having made this identification, $\Phi$ becomes the identity map and $v = u_2 - u_1$.
  Therefore,
  \begin{align*}
    \fd_{u_1}\xi - \fd_{u_2}\xi
    &=      
      \frac12\paren*{J(u_1) - J(u_2)} \circ\nabla \xi \circ j \\
    &\quad
      +
      \frac12\paren*{(\nabla_\xi J)(u_1) - (\nabla_\xi J)(u_2)} \circ \rd u_1 \circ j\\
    &\quad
      +      
      \frac12(\nabla_\xi J)(u_2)\circ(\rd u_1-\rd u_2)\circ j.
  \end{align*}
  This implies the asserted inequality.
\end{proof}

\begin{proof}[Proof of \autoref{Prop_DFuse}]
  Outside the regions where $r_n \leq 2R_0$ for some $n \in S$ the operators $\fd_{u_0}$ and $\fd_{\tilde u_\tau}$ agree.
  Within such a region and with the usual identifications
  \begin{equation*}
    \fd_{\tilde u_\tau^\circ}\fuse_\tau^\circ(\xi) 
    =
    \fd_{u_0}\xi
    +
    \underbrace{(\fd_{\tilde u_\tau^\circ}-\fd_{u_0})\xi}_{\eqcolon \rI}
    +
    \underbrace{\delbar \rho_\tau^n \cdot \paren*{\xi\circ\iota_\tau - \xi(n)}}_{\eqcolon \rII}
    +
    \underbrace{\rho_\tau^n\cdot\fd_{\tilde u_\tau^\circ}(\xi\circ\iota_\tau)}_{\eqcolon \rIII}
    -
    \underbrace{\rho_\tau^n\cdot\fd_{\tilde u_\tau^\circ}\xi(n)}_{\eqcolon \rIV}.
  \end{equation*}
  The difference $v \coloneq \tilde u_\tau^\circ - u_0 = \chi_\tau^n\cdot u_0\circ\iota_\tau$
  satisfies
  \begin{align*}
    \abs{v}
    &\leq
      c\epsilon_n/r_n
      \leq
      c \epsilon_n^{1/2} \qand \\
    \abs{\rd v}
    &\leq
      \abs{\rd\chi_\tau^n\cdot u_0\circ\iota_\tau} + \abs{\chi_\tau^n\rd(u_0\circ\iota_\tau)} 
      \leq
      c \epsilon_n/r_n^2.
  \end{align*}
  Therefore, by \autoref{Prop_CompareDU1WithDU2} and \autoref{Eq_AlmostDivergentIntegralEstimate},
  \begin{equation*}
    \Abs{\rI}_{L^p} \leq c\epsilon_n^{\frac1p}\Abs{\xi}_{W^{1,p}}.
  \end{equation*}
  By \autoref{Prop_CutOffEstimate},
  \begin{equation*}
    \Abs{\rII}_{L^p} \leq c \epsilon_n^{\frac12-\frac1p} \Abs{\xi}_{W^{1,p}}.
  \end{equation*}
  The term $\rIII$ can be written as
  \begin{equation*}
    \rIII
    =
    \rho_\tau^n \cdot \paren*{\fd_{\tilde u_\tau^\circ}\xi}\circ \iota_\tau
    +
    \rho_\tau^n\cdot\paren*{\fd_{\tilde u_\tau^\circ}-\fd_{\tilde u_\tau^\circ\circ\iota_\tau}}(\xi\circ\iota_\tau).
  \end{equation*} 
   The first term in this sum satisfies
  \begin{equation*}
  \Abs{\rho_\tau^n \cdot \paren*{\fd_{\tilde u_\tau^\circ}\xi}\circ \iota_\tau}_{L^p} \leq   \Abs{\fd_{u_0}\xi}_{L^p} +  \Abs{\rI}_{L^p}.
  \end{equation*}
  To estimate the second term, consider the difference
  \begin{equation*}
    w \coloneq \tilde u_\tau^\circ - \tilde u_\tau^\circ\circ\iota_\tau = (1-\chi_\tau^n)(u_0 - u_0\circ\iota_\tau).
  \end{equation*}
  It satisfies
  \begin{align*}
    \abs{w} &\leq c r_n \qand \\
    \abs{\rd w} &\leq c.
  \end{align*}
  Since $\rho_\tau^n$ is supported in the region where $\epsilon_n^{1/4} \leq r_n \leq 2\epsilon_n^{1/4}$,  \autoref{Prop_CompareDU1WithDU2}  implies that 
  \begin{equation*}
   \Abs{ \rho_\tau^n\cdot\paren*{\fd_{\tilde u_\tau^\circ}-\fd_{\tilde u_\tau^\circ\circ\iota_\tau}}(\xi\circ\iota_\tau)}_{L^p} \leq  c \epsilon_n^{1/2p}\Abs{\xi}_{W^{1,p}}.
  \end{equation*}
  Therefore,
  \begin{equation*}
    \Abs{\rIII}_{L^p} \leq  \Abs{\fd_{u_0}\xi}_{L^p} +  c \epsilon_n^{1/2p} \Abs{\xi}_{W^{1,p}}.
  \end{equation*}
  To estimate the term $\rIV$, write it in the form
  \begin{equation*}
    \rIV
    =
    \rho_\tau^n\cdot\paren*{\fd_{\tilde u_\tau^\circ}-\fd_{u(n)}}\xi(n),
  \end{equation*}
  with $\fd_{u(n)}$ denoting the operator associated with the constant map with value $u(n)$.  
  Since the difference $u_\tau^\circ - u(n)$ and its derivative are bounded, again from \autoref{Prop_CompareDU1WithDU2} we conclude that 
  \begin{equation*}
    \Abs{\rIV}_{L^p}
    \leq
    c\epsilon_n^{1/2p}\Abs{\xi}_{W^{1,p}}.
    \qedhere
  \end{equation*}
\end{proof}

\subsection{Construction of right inverses}
\label{Sec_ConstructionOfRightInverses}

Throughout this subsection,
let $\sO \subset L^p\Omega^{0,1}(\Sigma_0,u_0^*TX)$ be a finite dimensional subspace such that
\begin{equation}
  \label{Eq_OFillsCokernel}
  \im \fd_{u_0} + \sO = L^p\Omega^{0,1}(\Sigma_0,u_0^*TX).
\end{equation}
In particular, $\sO$ surjects onto $\coker\fd_{u_0}$.

\begin{definition}
  \label{Def_PullTau}
  Define $\pull_\tau \co L^p\Omega^{0,1}(\Sigma_0,u_0^*TX) \to L^p\Omega^{0,1}(\Sigma_\tau,\tilde u_\tau^*TX)$ by
  \begin{equation*}
    \pull_\tau(\eta)([z])
    \coloneq
    \begin{cases}
      \Phi_{\tilde u_\tau([z])}\Phi_{u_0(z)}^{-1}\eta(z) & \text{if } \epsilon_n^{1/2} \leq r_{n}(z) \leq 2R_0 \\
      \eta(z) & \text{otherwise}.
    \end{cases}
    \qedhere
  \end{equation*}
\end{definition}

Recall that $\Sigma_\tau$ is defined in \autoref{Def_SmoothingNodalRiemannSurface} by identifying the boundary components of $\set{r_n \geq \epsilon_n^{1/2}}$ and $\set{ r_{\nu(n)} \geq \epsilon_n^{1/2}}$ using $\iota_\tau$. 
The operator $\pull_\tau$ is obtained by simply restricting $(0,1)$--forms to these regions. 
The resulting $(0,1)$--form on $\Sigma_\tau$ is typically not continuous but it is still in $L^p$.
In particular, the ambiguity at $r_n = \epsilon_n^{1/2}$ in \autoref{Def_PullTau} is immaterial.
The reader should contrast \autoref{Def_PullTau} with the definition of $\fuse_\tau$, cf. \autoref{Def_Fuse}, which produces sections of class $W^{1,p}$, and therefore continuous. 

\begin{definition}
  \label{Def_BarD}
  Define $\overline\fd_{u_0} \co W^{1,p}\Gamma(\Sigma_0,\nu_0;u_0^*TX)\oplus \sO \to L^p\Omega^{0,1}(\Sigma_0,u_0^*TX)$ by
  \begin{equation*}
    \overline\fd_{u_0}(\xi,o) \coloneq \fd_{u_0}\xi + o.
  \end{equation*}
  Define $\overline\fd_{\tilde u_\tau} \co W^{1,p}\Gamma(\Sigma_\tau,\nu_\tau;\tilde u_\tau^*TX)\oplus \sO \to L^p\Omega^{0,1}(\Sigma_\tau,\tilde u_\tau^*TX)$ by
  \begin{equation*}
    \overline\fd_{\tilde u_\tau}(\xi,o) \coloneq \fd_{\tilde u_\tau}\xi + \pull_\tau(o).
    \qedhere
  \end{equation*}
\end{definition}

By construction,
$\overline\fd_{u_0}$ is surjective and, hence, has a right inverse $\fr_{u_0} \co L^p\Omega^{0,1}(\Sigma_0,u_0^*TX) \to W^{1,p}\Gamma(\Sigma_0,\nu_0;u_0^*TX)\oplus \sO$ of $\overline \fd_{u_0}$.
Henceforth, fix a choice of $\fr_{u_0}$.
The purpose of this subsection is to construct a right inverse $\fr_{\tilde u_\tau}$ to $\overline\fd_{\tilde u_\tau}$ for sufficiently small $\epsilon$.

\begin{definition}
  Define $\push_\tau\co L^p\Omega^{0,1}(\Sigma_\tau,\tilde u_\tau^*TX) \to L^p\Omega^{0,1}(\Sigma_0,u_0^*TX)$ by
  \begin{equation*}
    \push_\tau(\eta)(z)
    \coloneq
    \begin{cases}
      0 & \text{if } r_{n}(z) < \epsilon_n^{1/2} \\
      \Phi_{u_0(z)}\Phi_{\tilde u_\tau([z])}^{-1}\eta([z]) & \text{if } \epsilon_n^{1/2} \leq r_{n}(z) \leq 2R_0 \\
      \eta([z]) & \text{otherwise}.
    \end{cases}
    \qedhere
  \end{equation*}
\end{definition}

\begin{definition}
  Define $\tilde\fr_{\tilde u_\tau} \co L^p\Omega^{0,1}(\Sigma_\tau,\tilde u_\tau^*TX) \to W^{1,p}\Gamma(\Sigma_\tau,\nu_\tau;\tilde u_\tau^*TX)\oplus \sO$ by
  \begin{equation*}
    \tilde\fr_{\tilde u_\tau}
    \coloneq
    (\fuse_\tau \oplus \id_\sO) \circ \fr_{u_0} \circ \push_\tau.
    \qedhere
  \end{equation*}
\end{definition}

\begin{prop}
  \label{Prop_ApproximateRightInverse}
  The linear operator $\tilde\fr_{\tilde u_\tau}$ satisfies
  \begin{align}
    \label{Prop_ApproximateRightInverse_Estimate}
    \Abs[\big]{\overline\fd_{\tilde u_\tau}\circ \tilde\fr_{\tilde u_\tau} - \id}
    &\leq
      c\sum_{n \in S} \paren*{\epsilon_n^{\frac{1}{2p}}+\epsilon_n^{\frac12-\frac1p}}\Abs{\fr_{u_0}}
      \qand \\
    \notag
    \Abs{\tilde\fr_{\tilde u_\tau}}
    &\leq
      c\Abs{\fr_{u_0}}.
  \end{align}
\end{prop}

\begin{proof}
  The map $\push_\tau$ is bounded by a constant independent of $\tau$ and,
  by \autoref{Prop_CutOffEstimate},
  so is $\fuse_\tau$.
  This implies the estimate on $\Abs{\tilde\fr_{\tilde u_\tau}}$.
  
  Let $\eta \in L^p\Omega^{0,1}(\Sigma_\tau,\tilde u_\tau^*TX)$.
  To prove \autoref{Prop_ApproximateRightInverse_Estimate},
  we estimate $\Abs[\big]{\overline\fd_{\tilde u_\tau}\tilde\fr_{\tilde u_\tau}\eta - \eta}_{L^p}$ as follows.
  Set
  \begin{equation*}
    (\xi,o) \coloneq \fr_{u_0} \circ \push_\tau(\eta),
  \end{equation*}
  so that
  \begin{equation*}
    \fd_{u_0}\xi + o = \push_\tau(\eta).
  \end{equation*}
  By \autoref{Prop_CompareDU1WithDU2} applied to $\tilde u_\tau^\circ$ and $u_0$ and using \autoref{Eq_AlmostDivergentIntegralEstimate} and the fact that on $\Sigma_\tau^\circ$ we have $\pull_\tau(o)=o$, 
  \begin{align*}
    \Abs{\fd_{\tilde u_\tau^\circ}\xi +\pull_\tau(o) - \eta}_{L^p}
    &\leq
      c\epsilon^{\frac1p}\Abs{\xi}_{W^{1,p}} \\
    &\leq
      c\epsilon^{\frac1p}\Abs{\fr_{u_0}}\Abs{\eta}_{L^p}.
  \end{align*}
  Therefore,
  it remains to estimate
  \begin{equation}
    \label{Eq_DelFuseEstimate}
    \fd_{\tilde u_\tau^\circ} \paren*{\rho_\tau^n\cdot \paren*{\xi\circ\iota_\tau - \xi(n)}}
    =
    \underbrace{\delbar \rho_\tau^n \cdot \paren*{\xi\circ\iota_\tau - \xi(n)}}_{\eqcolon \rI} 
    +
    \underbrace{\rho_\tau^n\cdot \fd_{\tilde u_\tau^\circ}(\xi\circ\iota_\tau)}_{\eqcolon \rII}
    -
    \underbrace{\rho_\tau^n\cdot \fd_{\tilde u_\tau^\circ}\xi(n)}_{\eqcolon \rIII}.
  \end{equation}
  By \autoref{Prop_CutOffEstimate},
  \begin{equation*}
    \Abs{\rI}_{L^p}
    \leq
    c\epsilon^{\frac12-\frac1p}\Abs{\xi}_{W^{1,p}} \leq c\epsilon^{\frac12-\frac1p}\Abs{\fr_{u_0}}\Abs{\eta}_{L^p}.    
  \end{equation*}
  To estimate the second term, observe that 
  in the region where $r_n \geq \epsilon_n^{1/2}$, 
  \begin{equation*}
    \fd_{u_0 \circ \iota_\tau} (\xi\circ \iota_\tau)
    =
    \iota_\tau^*( \fd_{u_0} \xi) 
    = 
    \iota_\tau^*(\push_\tau(\eta))
    =
    0.
  \end{equation*}
  To understand the last identity,
  observe that $r_{n}(\iota_\tau(z))  = \epsilon_n r_{\nu(n)}^{-1}(z)$
  and $\push_\tau(\eta)$ is defined to vanish in the region of $\Sigma_0$ where $r_{\nu(n)} \leq \epsilon_n^{1/2}$. 
  Thus, by \autoref{Prop_CompareDU1WithDU2} applied to $\tilde u_\tau$ and $u_0  \circ \iota_\tau$, and using the fact that $\rho_\tau^n$ is supported in the region where $\epsilon_n^{1/4} \leq r_n \leq 2\epsilon_n^{1/4}$ whose area is proportional to $\epsilon_n^{1/2}$, 
  \begin{equation*}
    \Abs{\rII}_{L^p}
    \leq
    c\epsilon^{\frac{1}{2p}}\Abs{\xi}_{W^{1,p}}
    \leq 
    c\epsilon^{\frac{1}{2p}}\Abs{\fr_{u_0}}\Abs{\eta}_{L^p}.
  \end{equation*}  
  The vector field $\xi(n)$ is constant with respect the chosen trivialization.
  Since the operator $\fd_{u_0(n)}$ associated with the constant map agrees with the standard $\delbar$--operator,
  \begin{equation*}
    \fd_{u_0(n)}\xi(n) = 0.
  \end{equation*}   
  Therefore,
  using \autoref{Prop_CompareDU1WithDU2} applied to $\tilde u_\tau$ and the constant map $u(n)$, and the estimate on the area of the support of $\rho_\tau^n$, we arrive at
  \begin{equation*}
    \Abs{\rIII}_{L^p}
    \leq
    c\epsilon^{\frac{1}{2p}}\Abs{\fr_{u_0}}\Abs{\eta}_{L^p}.
    \qedhere
  \end{equation*}
\end{proof}

Throughout the remainder of this subsection,
suppose the following.

\begin{hypothesis}
  \label{Hyp_TauRightInverse}
  The smoothing parameter $\tau$ is such that the right-hand side of \autoref{Prop_ApproximateRightInverse_Estimate} is at most $1/2$.
\end{hypothesis}

\begin{definition}
  \label{Def_RightInverse}
  Define the \defined{right inverse} $\fr_{\tilde u_\tau} \co L^p\Omega^{0,1}(\Sigma_\tau,\tilde u_\tau^*TX) \to W^{1,p}\Gamma(\Sigma_\tau,\nu_\tau;\tilde u_\tau^*TX)\oplus \sO$ associated with $\fr_{u_0}$ by
  \begin{equation*}
    \fr_{\tilde u_\tau}
    \coloneq
    \tilde\fr_{\tilde u_\tau} \paren*{\overline\fd_{\tilde u_\tau} \tilde\fr_{\tilde u_\tau}}^{-1}
    =
    \tilde\fr_{\tilde u_\tau} \sum_{k=0}^\infty \paren*{\id - \overline\fd_{\tilde u_\tau} \tilde\fr_{\tilde u_\tau}}^k.
    \qedhere
  \end{equation*}
\end{definition}

The following is an immediate consequence of the definition.

\begin{prop}
  \label{Prop_RightInverse}
  The right inverse $\fr_{\tilde u_\tau} \co L^p\Omega^{0,1}(\Sigma_\tau,\tilde u_\tau^*TX) \to W^{1,p}\Gamma(\Sigma_\tau,\nu_\tau;\tilde u_\tau^*TX)\oplus \sO$ satisfies
  \begin{align*}
    \overline\fd_{u_\tau}\fr_{\tilde u_\tau} &= \id \qand \\
    \Abs{\fr_{\tilde u_\tau}} &\leq c \Abs{\fr_{u_0}};
  \end{align*}
  furthermore,
  \begin{equation*}
    \im \fr_{\tilde u_\tau} = \im \tilde\fr_{\tilde u_\tau}.
  \end{equation*}
\end{prop}

\subsection{Complements of the image of \texorpdfstring{$\fr_{\tilde u_\tau}$}{the right inverse}}

\begin{prop}
  \label{Prop_FuseKComplementsImR}
  Given $c_f > 0$ there is a constant $\delta = \delta(c_f) > 0$ such that the following holds.
  If $\tau$ satisfies $\epsilon < \delta$ and $K \subset W^{1,p}\Gamma(\Sigma_0,\nu_0;u_0^*TX)$ is a subspace with $\dim K = \dim \ker \overline\fd_{u_0}$ and such that for every $\kappa \in K$
  \begin{equation*}
    \Abs{\fd_{u_0}\kappa}_{L^p}
    \leq
    \delta\Abs{\kappa}_{W^{1,p}} \qandq
    \Abs{\kappa}_{W^{1,p}}
    \leq      
    c_f\Abs{\fuse_\tau(\kappa)}_{W^{1,p}},
  \end{equation*}
  then every $(\xi,o) \in W^{1,p}\Gamma(\Sigma_\tau,\nu_\tau;\tilde u_\tau^*TX) \oplus \sO$ can be uniquely written as
  \begin{equation*}
    (\xi,o) = \fr_{\tilde u_\tau}\eta + (\kappa,0)
  \end{equation*}
  with $\eta \in L^p\Omega^{0,1}(\Sigma_\tau,\tilde u_\tau^*TX)$ and $\kappa \in K$;
  moreover,
  \begin{equation*}
    \Abs{\eta}_{L^p} + \Abs{\kappa}_{W^{1,p}}
    \leq
    c(c_f)\paren*{\Abs{\xi}_{W^{1,p}} + \abs{o}}.
  \end{equation*}
  Here $\abs{o} = \Abs{o}_{L^p}$ is the norm of $o$ induced by the inclusion $\sO \subset L^p\Omega^{0,1}(\Sigma_0,u_0^*TX)$.
\end{prop}

\begin{proof}
  Because $\fr_{\tilde u_\tau}$ and $\fuse_\tau|_K$ are injective and given the hypothesis on $\fuse_\tau|_K$,
  it suffices to show that $W^{1,p}\Gamma(\Sigma_\tau,\nu_\tau;\tilde u_\tau^*TX) \oplus \sO$ is the direct sum of $\im(\fr_{\tilde u_\tau})$ and $\im(\fuse_\tau|_K)\oplus 0$.
  
  By the index formula \autoref{Eq_IndexFormula}, \autoref{Rmk_ArithmeticGenus_Smoothing}, and \autoref{Prop_SmallEnergyNeck},
  \begin{align*}
    \ind\fd_{u_0}
    &=
      2\inner{(u_0^*c_1(X,J)}{[\Sigma_0]} + 2n\paren*{1-p_a(\Sigma_0,\nu_0)} \\
    &=
      2\inner{(\tilde u_\tau)^*c_1(X,J)}{[\Sigma_\tau]} + 2n\paren*{1-p_a(\Sigma_\tau,\nu_\tau)} \\
    &=
      \ind\fd_{\tilde u_\tau}.
  \end{align*}
  Therefore and because $\overline \fd_{u_0}$ is surjective and $\fr_{\tilde u_\tau}$ is injective,
  \begin{equation*}
    \codim \im(\fr_{\tilde u_\tau})
    =
    \ind\overline \fd_{\tilde u_\tau}
    =
    \ind\overline \fd_{u_0}
    =
    \dim\ker\overline\fd_{u_0}.
  \end{equation*}
  Hence, it remains to prove that $\im(\fr_{\tilde u_\tau})$ and $\im(\fuse_\tau|_K)\oplus 0$ intersect trivially.
  
  Suppose that $\eta \in L^p\Omega^{0,1}(\Sigma_{\sigma,\tau},\tilde u_\tau^*TX)$ and $\kappa \in K$ satisfy
  \begin{equation*}
    \fr_{\tilde u_\tau}(\eta) = (\fuse_\tau(\kappa),0).
  \end{equation*}
  By \autoref{Prop_DFuse} as well as the hypothesis on $\fuse_\tau$ and for sufficiently small $\delta$,
  \begin{align*}
    \Abs{\eta}_{L^p}
    &=
      \Abs{\fd_{\tilde u_\tau}\fuse_\tau(\kappa)}_{L^p} \\
    &\leq
      c\sqparen*{
      \delta +  \#S\cdot\paren*{\delta^{\frac{1}{2p}}+\delta^{\frac12-\frac1p}}
      }
      \Abs{\kappa}_{W^{1,p}} \\
    &\leq
      cc_f\sqparen*{
      \delta +  \#S\cdot\paren*{\delta^{\frac{1}{2p}}+\delta^{\frac12-\frac1p}}
      }
      \Abs{\eta}_{L^p} \\
    &\leq
      \frac12
      \Abs{\eta}_{L^p}.
  \end{align*}
  Therefore, $\eta$ vanishes.
\end{proof}

\subsection{Kuranishi model for a neighborhood of nodal maps}
\label{Sec_KuranishiModel}

Throughout,
let $(\Sigma_0,j_0,\nu_0)$ be a nodal Riemann surface with nodal set $S$,
let $(X,J_0,h)$ be an almost Hermitian manifold, and 
let $u_0 \co (\Sigma_0,j_0,\nu_0) \to (X,J_0)$ be a nodal $J_0$--holomorphic map.
Let $(\pi\co \sX \to \Delta, \star = (0,0), \iota)$ be the versal deformation of $(\Sigma_0,j_0,\nu_0)$ constructed in \autoref{Sec_VersalDeformations_Construction} with fibers
\begin{equation*}
  (\Sigma_{\sigma,\tau},j_{\sigma,\tau},\nu_{\sigma,\tau}) = \pi^{-1}(\sigma,\tau).
\end{equation*}
Let $\delta_\sJ > 0$ and let
\begin{equation*}
  \sU
  \subset
  \set*{
    J \in \sJ(X)
    :
    \Abs{J-J_0}_{C^1} < \delta_\sJ
  }
\end{equation*}
be such that for every $k \in \N$
\begin{equation*}
  \sup_{J \in \sU} \Abs{J-J_0}_{C^k} < \infty.
\end{equation*}
In the upcoming discussion we may implicitly shrink $\Delta$ and $\delta_\sJ$,
in order to ensure that \autoref{Hyp_UTau} and \autoref{Hyp_TauRightInverse} hold and various expressions involving $\abs{\sigma}$, $\epsilon \coloneq \max\set{\epsilon_n : n \in S}$ with $\epsilon_n \coloneq \abs{\tau_n}$, and $\Abs{J-J_0}_{C^1}$ are sufficiently small.

The purpose of this subsection is to analyze whether $u_0$ can be slightly deformed to a $J$--holomorphic map $u_{\sigma,\tau} \co (\Sigma_{\sigma,\tau},j_{\sigma,\tau},\nu_{\sigma,\tau}) \to (X,J)$ with $J \in \sU$.
More precisely, we show that  a Gromov neighborhood of $u_0$ in the space of nodal $J$--holomorphic maps with $J \in \sU$ is homeomorphic to the zero set of a continuous map
\begin{equation*}
  \ob\co \Delta\times\sU\times\sI \to \sO,
\end{equation*}  
where $\sI$ an open subset of the deformation space $\ker\fd_{u_0,J_0}$  and $\sO \iso \coker\fd_{u_0,J_0}$ is the obstruction space. 
This is a local Kuranishi model at $u_0$ for the universal moduli space of pseudo-holomorphic nodal maps. 

\begin{remark}
  Since we are not interested here in the global properties of the universal moduli space, we do not require that $u_0$ is a stable map. 
  A local Kuranishi model can be constructed around any pseudo-holomorphic map. 
  However, the Gromov limit, as defined in \autoref{Def_GromovConvergence}, is not necessarily unique for unstable maps, so the universal moduli space of all nodal pseudo-holomorphic maps is not a Hausdorff space.  
\end{remark}

To facilitate the discussion in \autoref{Sec_LeadingOrderTermOfObstructionGhostComponents} (and although it makes the present discussion somewhat more awkward than it needs to be) the construction of the Kuranishi model proceeds in two steps.
Choose a partition
\begin{equation*}
  S = S_1 \amalg S_2 \qwithq \nu_0(S_1) = S_1 \qandq \nu_0(S_2) = S_2
\end{equation*}
and write every smoothing parameter $\tau$ as
\begin{equation*}
  \tau = (\tau_1,\tau_2)
  \qwithq
  \tau_1 = (\tau_{1,n})_{n \in S_1}
  \qandq
  \tau_2 = (\tau_{2,n})_{n \in S_2}.
\end{equation*}
The first step of our construction varies $\sigma$ and $\tau_1$ but $\tau_2 = 0$ is fixed.
The second step holds $\sigma$ and $\tau_1$ fixed and varies $\tau_2$.

Denote by $u_{\sigma,0} \co \Sigma_{\sigma,0} \to X$ the smooth map underlying $u_0$.
Denote by
\begin{equation*}
  \fd_{u_0;J_0} \co W^{1,p}\Gamma(\Sigma_0,\nu_0;u_0^*TX) \to L^p\Omega^{0,1}(\Sigma_0,u_0^*TX)
\end{equation*}
the linear operator associated with $u_0$ defined in \autoref{Def_JHolomorphicLinearization}.
Let
\begin{equation*}
  \sO \subset \Omega^{0,1}(\Sigma_0,u_0^*TX) \subset L^p\Omega^{0,1}(\Sigma_0,u_0^*TX)
\end{equation*}
be a lift of $\coker\fd_{u_0;J_0}$;
that is: $\dim \sO = \dim\coker\fd_{u_0}$ and \autoref{Eq_OFillsCokernel} holds.
We will assume that all $1$--forms in $\sO$ are smooth on each component of $\Sigma_0$. 
(The canonical choice is $\sO = \ker \fd_{u_0;J_0}^*$, but this choice is not always the most convenient.)
Let $\Delta_1$ parametrize complex structures on $(\Sigma_0,\nu_0)$  as at the beginning of \autoref{Sec_VersalDeformations_Construction}, and let $\sU$ be an open neighborhood of $J_0$ as above. 
Trivialize the bundle over $\Delta_1\times \sU$ whose fiber over $(\sigma,J) \in \Delta_1 \times \sU$ is $\Omega^{0,1}(\Sigma_{\sigma,0},u_{\sigma,0}^*TX)$ with the $(0,1)$--part taken with respect to $j_{\sigma,0}$ and $J$. 
This identifies $\Omega^{0,1}(\Sigma_0,u_0^*TX)$ and $\Omega^{0,1}(\Sigma_{\sigma,0},u_{\sigma,0}^*TX)$ and thus exhibits $\sO$ as a subset of $L^p\Omega^{0,1}(\Sigma_{\sigma,0},u_{\sigma,0}^*TX)$ for which \autoref{Eq_OFillsCokernel} holds for $\fd_{\tilde u_{\sigma,0};J}$ instead of $\fd_{u_0}$.
Define
\begin{equation*}
  \overline \fd_{\tilde u_{\sigma,\tau_1,0};J}
  \co
  W^{1,p}\Gamma(\Sigma_{\sigma,\tau_1,0},\nu_{\sigma,0};\tilde u_{\sigma,\tau_1,0}^*TX) \oplus \sO
  \to
  L^p\Omega^{0,1}(\Sigma_{\sigma,\tau_1,0},\tilde u_{\sigma,\tau_1,0}^*TX)
\end{equation*}
as in \autoref{Def_BarD}.
The construction in \autoref{Sec_ConstructionOfRightInverses} yields a right inverse
\begin{equation*}
  \fr_{\tilde u_{\sigma,\tau_1,0};J}
  \co
  L^p\Omega^{0,1}(\Sigma_{\sigma,\tau_1,0},\tilde u_{\sigma,\tau_1,0}^*TX)
  \to
  W^{1,p}\Gamma(\Sigma_{\sigma,\tau_1,0},\nu_{\sigma,\tau_1,0};\tilde u_{\sigma,\tau_1,0}^*TX) \oplus \sO.
\end{equation*}
of $\overline\fd_{\tilde u_{\sigma,\tau_1,0};J}$.

The next two propositions build a Kuranishi model for a neighborhood of $u_0$ in the Gromov compactification of the moduli space of pseudo-holomorphic maps.
In essence, they assert that for every smoothing parameter $\tau$ and an infinitesimal deformation $\kappa$ the pseudo-holomorphic map equation can be solved modulo obstructions.
The construction of the Kuranishi model proceeds in two steps, described by \autoref{Prop_KuranishiModel_Step1} and \autoref{Prop_KuranishiModel_Step2}, in order to obtain better control of the obstruction map. 
The first step is to smooth the nodes in $S_1$. 

\begin{prop}
  \label{Prop_KuranishiModel_Step1}
  There are constants $\delta_\kappa,\Lambda > 0$ such that for every $(\sigma,\tau_1,0) \in \Delta$ and $\kappa \in \ker\fd_{u_0}$ with $\abs{\kappa} < \delta_\kappa$ there exists a unique pair
  \begin{equation*}
    \paren*{\xi(\sigma,\tau_1;J;\kappa),o(\sigma,\tau_1;J;\kappa)}
    \in
    \im \fr_{\tilde u_{\sigma,\tau_1,0};J} \subset W^{1,p}\Gamma(\Sigma_{\sigma,\tau_1,0},\nu_{\sigma,\tau_1,0};\tilde u_{\sigma,\tau_1,0}^*TX)\oplus \sO
  \end{equation*}
  with
  \begin{equation*}
    \Abs{\xi(\sigma,\tau_1;J;\kappa)}_{W^{1,p}} + \abs{o(\sigma,\tau_1;J;\kappa)} \leq \Lambda
  \end{equation*}
  satisfying
  \begin{equation}
    \label{Eq_KuranishiModel_Step1}
    \fF_{\tilde u_{\sigma,\tau_1,0};J}(\fuse_{\tau_1,0}\kappa+\xi(\sigma,\tau_1;J;\kappa)) + \pull_{\tau_1,0}(o(\sigma,\tau_1;J;\kappa)) = 0,
  \end{equation}
  with $\fF$ as in \autoref{Def_FU}.
  Furthermore,
  \begin{equation}
    \label{Eq_KuranishiModel_Step1_Estimate}
    \Abs{\xi(\sigma,\tau_1;J;\kappa)}_{W^{1,p}} + \abs{o(\sigma,\tau_1;J;\kappa)}
    \leq
    c\paren[\big]{\abs{\sigma}+\abs{\tau_1}^{\frac12+\frac1p}+\Abs{J-J_0}_{C^0}+\abs{\kappa}}.
  \end{equation}
\end{prop}

\begin{proof}
  Since $\fr_{\tilde u_{\sigma,\tau_1,0};J}$ is injective,
  \autoref{Eq_KuranishiModel_Step1} is equivalent to the fixed-point equation
  \begin{equation*}
    \eta
    =
    \bF(\eta)
    \coloneq
    \eta - \fF_{\tilde u_{\sigma,\tau_1,0};J}\paren[\big]{\fuse_{\tau_1,0}\kappa + \pr_1\fr_{\tilde u_{\sigma,\tau_1,0};J}\eta} - \pull_{\tau_1,0}(\pr_2\fr_{\tilde u_{\sigma,\tau_1,0};J}\eta).
  \end{equation*}
  Here $\pr_1$ and $\pr_2$ denote the projections to the first and second summand of
  \begin{equation*}
    W^{1,p}\Gamma(\Sigma_{\sigma,\tau_1,0},\nu_{\sigma,\tau_1,0};\tilde u_{\sigma,\tau_1,0}^*TX)\oplus \sO
  \end{equation*}
  respectively.
  By \autoref{Prop_QuadraticEstimate},
  \begin{equation*}
    \bF(\eta)
    =
    - \delbar_J(\tilde u_{\sigma,\tau_1,0},j_{\sigma,\tau_1,0})
    - \fd_{\tilde u_{\sigma,\tau_1,0};J} \fuse_{\tau_1,0}\kappa
    - \fn_{\tilde u_{\sigma,\tau_1,0};J}\paren[\big]{\fuse_{\tau_1,0}\kappa + \pr_1 \circ \fr_{\tilde u_{\sigma,\tau_1,0};J} \eta}.
  \end{equation*}

  By \autoref{Prop_ErrorEstimate} and \autoref{Prop_QuadraticEstimate},
  \begin{equation*}
    \Abs{\bF(0)}_{L^p}
    \leq
    c\paren*{
      \abs{\sigma}
      + \Abs{J-J_0}_{C^0}
      + \abs{\tau_1}^{\frac12+\frac1p}
      + \abs{\kappa}
      + \abs{\kappa}^2
    }.
  \end{equation*}
  Moreover,
  by \autoref{Prop_RightInverse} and \autoref{Prop_QuadraticEstimate},
  \begin{equation*}
    \Abs{\bF(\eta_1)-\bF(\eta_2)}_{L^p}
    \leq    
    c\paren*{
      \abs{\kappa} + \Abs{\eta_1}_{L^p} + \Abs{\eta_2}_{L^p}
    }
    \Abs{\eta_1-\eta_2}_{L^p}.
  \end{equation*}
  Therefore,
  provided $\delta_\kappa$ is sufficiently small,
  there is an $R > 0$ such that $\Abs{\bF(0)}_{L^p} \leq R/2$ and for every $\eta_1, \eta_2 \in \bar B_R(0) \subset L^P\Omega^{0,1}(\Sigma_{\sigma,\tau},\tilde u_{\sigma,\tau}^*TX)$
  \begin{equation*}
    \Abs{\bF(\eta_1)-\bF(\eta_2)}_{L^p} \leq \frac12\Abs{\eta_1-\eta_2}_{L^p}.
  \end{equation*}
  This shows that $\bF$ maps $\bar B_R(0)$ into $\bar B_R(0)$ and $\bF\co \bar B_R(0) \to \bar B_R(0)$ is a contraction.
  Thus, the first assertion follows from Banach's fixed-point theorem. 
  The second follows from the above and \autoref{Prop_ErrorEstimate}.
\end{proof}

This completes the first step.
In the second step, we smooth the nodes in $S_2$.
This step is analogous to the first one, with $u_0$ being replaced by the maps obtained from \autoref{Prop_KuranishiModel_Step1}.
For $(\sigma,\tau) \in \Delta$ and $\kappa \in \ker\fd_{u_0}$ with $\Abs{\kappa}_{W^{1,p}} < \delta_\kappa$ set
\begin{equation*}
  u_{\sigma,\tau_1,0;J;\kappa}
  \coloneq
  \exp_{\tilde u_{\sigma,\tau_1,0}}(\fuse_{\tau_1,0}\kappa + \xi(\sigma,\tau_1;J;\kappa)) \qandq
  \tilde u_{\sigma,\tau;J;\kappa}
  \coloneq
  \widetilde{\paren*{u_{\sigma,\tau_1,0;J;\kappa}}}_\tau;
\end{equation*}
that is: $\tilde u_{\sigma,\tau;J;\kappa}$ is obtained from $u_{\sigma,\tau_1,0;J;\kappa}$ by the construction in \autoref{Def_TildeUTau}.

\begin{definition}
  Define $\pull_{\sigma,\tau_1,0;J;\kappa}\co L^p\Omega^{0,1}(\Sigma_0,u_0^*TX) \to L^p\Omega^{0,1}(\Sigma_{\sigma,\tau_1,0},u_{\sigma,\tau_1,0;J;\kappa}^*TX)$ to be the composition of $\pull_{\tau_1,0}$ with the map induced by parallel transport along the geodesics
  \begin{equation*}
    t\mapsto \exp_{\tilde u_{\sigma,\tau_1,0}}\paren*{t\paren{\fuse_{\tau_1,0}\kappa + \xi(\sigma,\tau_1;J;\kappa)}}.
  \end{equation*}
  Furthermore,
  denote by $\pull_{\sigma,\tau;J;\kappa}\co L^p\Omega^{0,1}(\Sigma_0,u_0^*TX) \to L^p\Omega^{0,1}(\Sigma_{\sigma,\tau},\tilde u_{\sigma,\tau;J;\kappa}^*TX)$ the composition of $\pull_{\sigma,\tau_1,0;J;\kappa}$ with $\pull_{\tau_2}\co L^p\Omega^{0,1}(\Sigma_{\sigma,\tau_1,0},u_{\sigma,\tau_1,0;J;\kappa}^*TX) \to L^p\Omega^{0,1}(\Sigma_{\sigma,\tau},\tilde u_{\sigma,\tau;J;\kappa}^*TX)$ defined in \autoref{Def_PullTau}.  
\end{definition}

The subspace $\pull_{\sigma,\tau_1,0;J;\kappa}(\sO)$ satisfies \autoref{Eq_OFillsCokernel} for $u_{\sigma,\tau_1,0;J;\kappa}$ instead of $u_0$.
Define
\begin{equation*}
  \overline \fd_{\tilde u_{\sigma,\tau;J;\kappa}}
  \co
  W^{1,p}\Gamma(\Sigma_{\sigma,\tau},\nu_{\sigma,\tau};\tilde u_{\sigma,\tau;J;\kappa}^*TX) \oplus \sO
  \to
  L^p\Omega^{0,1}(\Sigma_{\sigma,\tau},\tilde u_{\sigma,\tau;J;\kappa}^*TX)
\end{equation*}
as in \autoref{Def_BarD}.
The construction in \autoref{Sec_ConstructionOfRightInverses} yields a right inverse
\begin{equation*}
  \fr_{\tilde u_{\sigma,\tau;J;\kappa}}
  \co
  L^p\Omega^{0,1}(\Sigma_{\sigma,\tau},\tilde u_{\sigma,\tau;J;\kappa}^*TX)
  \to
  W^{1,p}\Gamma(\Sigma_{\sigma,\tau},\nu_{\sigma,\tau};\tilde u_{\sigma,\tau;J;\kappa}^*TX) \oplus \sO
\end{equation*}
of $\overline\fd_{\tilde u_{\sigma,\tau;J;\kappa}}$.

\begin{prop}
  \label{Prop_KuranishiModel_Step2}
  There are constants $\delta_\kappa,\Lambda > 0$ such that for every $(\sigma,\tau;J) \in \Delta \times \sU$ and $\kappa \in \ker\fd_{u_0}$ with $\Abs{\kappa}_{W^{1,p}} < \delta_\kappa$ there exists a unique pair
  \begin{equation*}
    \paren{\hat\xi(\sigma,\tau;J;\kappa),\hat o(\sigma,\tau;J;\kappa)}
    \in \im \fr_{\tilde u_{\sigma,\tau;\kappa};J} \subset W^{1,p}\Gamma(\Sigma_{\sigma,\tau},\nu_{\sigma,\tau};\tilde u_{\sigma,\tau;\kappa}^*TX)\oplus \sO
  \end{equation*}
  with
  \begin{equation*}
    \Abs{\hat\xi(\sigma,\tau;J;\kappa)}_{W^{1,p}} + \abs{\hat o(\sigma,\tau;J;\kappa)} \leq \Lambda
  \end{equation*}
  satisfying
  \begin{equation}
    \label{Eq_KuranishiModel_Step2}
    \fF_{\tilde u_{\sigma,\tau;J;\kappa}}(\hat\xi(\sigma,\tau;J;\kappa)) + \pull_{\sigma,\tau;J;\kappa}(o(\sigma,\tau_1;J;\kappa) + \hat o(\sigma,\tau;J;\kappa)) = 0.
  \end{equation}
  Furthermore,
  \begin{equation}
    \label{Eq_KuranishiModel_Step2_Estimate}
    \Abs{\hat\xi(\sigma,\tau;J;\kappa)}_{W^{1,p}} + \abs{\hat o(\sigma,\tau;J;\kappa)}
    \leq
    c\Abs{\delbar_J(\tilde u_{\sigma,\tau;J;\kappa},j_{\sigma,\tau}) + \pull_{\sigma,\tau;J;\kappa}(o(\sigma,\tau_1;J;\kappa))}_{L^p}.
  \end{equation}
\end{prop}

\begin{proof}
  This is similar to the proof of \autoref{Prop_KuranishiModel_Step1}.
\end{proof}

\begin{definition}
  \label{Def_ObstructionMap}
  Set $\sI \coloneq B_{\delta_\kappa}(0) \subset \ker \fd_{u_0}$.
  The \defined{Kuranishi map} $\ob\co \Delta\times\sU\times\sI \to \sO$ is defined by
  \begin{equation*}
    \ob(\sigma,\tau;J;\kappa)
    \coloneq
    o(\sigma,\tau_1;J;\kappa) + \hat o(\sigma,\tau;J;\kappa),
  \end{equation*}
  with $o$ and $\hat o$ as in \autoref{Prop_KuranishiModel_Step1} and \autoref{Prop_KuranishiModel_Step2}.
\end{definition}

The upshot of the preceding discussion is that $u_0$ can be slightly deformed to a $J$--holomorphic map $u_{\sigma,\tau} \co (\Sigma_{\sigma,\tau},j_{\sigma,\tau},\nu_{\sigma,\tau}) \to (X,J)$; if and only if there is a $\kappa \in \sI$ with $\ob(\sigma,\tau;J;\kappa) = 0$.
The following shows that this Kuranishi model indeed describes a Gromov neighborhood of $u_0\co (\Sigma_0,j_0,\nu_0) \to (X,J_0)$.

\begin{prop}
  \label{Prop_KuranishiModelCoversGromovNeighborhood}
  Let $(\sigma_k,\tau_k)_{k\in\N}$ be a sequence in $\Delta$ converging to $(0,0)$ and let $(J_k)_{k \in \N}$ be a sequence in $\sU$ converging to $J_0$. 
  If
  \begin{equation*}
    \paren*{u_k \co (\Sigma_{\sigma_k,\tau_k},j_{\sigma_k,\tau_k},\nu_{\sigma_k,\tau_k}) \to (X,J_k)}_{k\in\N}
  \end{equation*}
  is a sequence of nodal pseudo-holomorphic maps which Gromov converges to $u_0\co (\Sigma_0,j_0,\nu_0) \to (X,J_0)$
  then there is a $K \in \N$ such that for every $k \geq K$ there are $\kappa_k \in \ker\fd_{u_0}$ and $(\xi_k,0) \in \im \fr_{\tilde u_{\sigma_k,\tau_k;\kappa_k};J_k}$ with
  \begin{equation*}
    u_k = \exp_{\tilde u_{\sigma_k,\tau_k;\kappa_k}}(\xi_k);
  \end{equation*}
  moreover,
  \begin{align*}
    \lim_{k\to\infty} \abs{\kappa_k} = 0 \qandq \lim_{k\to\infty} \Abs{\xi_k}_{W^{1,p}} = 0.
  \end{align*}
  In particular,
  \begin{equation*}
    \ob(\sigma_k,\tau_k;J_k;\kappa_k) = 0.
  \end{equation*}
\end{prop}

The proof of this proposition relies on the following result.

\begin{prop}
  \label{Prop_GromovNeighborhoodGraph}
  Assume the situation of \autoref{Prop_KuranishiModelCoversGromovNeighborhood}.  
  There are $K \in \N$, $\delta_\kappa > 0$, and $c > 0$ such that for every $k \geq K$ and $\kappa \in \ker\fd_{u_0}$ with $\Abs{\kappa}_{W^{1,p}} < \delta_\kappa$ there is a $\zeta_{k;\kappa} \in \Gamma(\Sigma_{\sigma_k,\tau_k},\tilde u_{\sigma_k,\tau_k;J_k;\kappa}^*TX)$ with
  \begin{equation*}
    u_k = \exp_{\tilde u_{\sigma_k,\tau_k;J_k;J_k;\kappa}}(\zeta_{k;\kappa});
  \end{equation*}
  moreover,
  \begin{align*}
    \limsup_{k\to \infty}\Abs{\zeta_{k;\kappa}}_{W^{1,p}}
    \leq
    c\abs{\kappa}.
  \end{align*}
\end{prop}

\begin{proof}
  The proof has two steps:
  the construction of $\zeta_{k;\kappa}$ and the proof of the estimate.
  Recall the definition of $U_x$ from \autoref{Def_ExpInverse}.

  \setcounter{step}{0}
  \begin{step}
    There are $K \in \N$ and $\delta_\kappa > 0$ such that for every $k \geq K$, $\kappa \in \ker\fd_{u_0}$ with $\Abs{\kappa}_{W^{1,p}} < \delta_\kappa$, and $z \in \Sigma_{\sigma_k,\tau_k}$
    \begin{equation*}
      u_k(z) \in U_{\tilde u_{\sigma_k,\tau_k;J_k;\kappa}(z)};
    \end{equation*}
    in particular,
    there is a section $\zeta_{k;\kappa} \in \Gamma(\Sigma_{\sigma_k,\tau_k},\tilde u_{\sigma_k,\tau_k;J_k}^*TX)$ given by
    \begin{equation*}
      \zeta_{k;\kappa} \coloneq {\exp^{-1}_{\tilde u_{\sigma_k,\tau_k;J_k;\kappa}}}\circ u_k.
    \end{equation*}
  \end{step}

  By \autoref{Eq_KuranishiModel_Step1_Estimate}, \autoref{Eq_KuranishiModel_Step2_Estimate}, and \autoref{Prop_ErrorEstimate},
  \begin{equation}
    \label{Eq_TildeUSigmaTauKappaTildeUSigmaTau0}
    d(\tilde u_{\sigma_k,\tau_k;J_k;\kappa},\tilde u_{\sigma_k,\tau_k;J_k;0})
    \leq
    c\paren[\big]{\abs{\sigma_k} + \epsilon_k^{\frac12+\frac1p} + \Abs{J_k-J_0}_{C^0} + \abs{\kappa}}.
  \end{equation}
  Therefore, it suffices to consider $\kappa=0$ and prove that there exists a $K \in \N$ such that for every $k \geq K$
  \begin{equation*}
    u_k(z) \in \frac12 U_{\tilde u_{\sigma_k,\tau_k;J_k;0}(z)}.
  \end{equation*}
  
  Using the framing $\Psi$ from \autoref{Def_ConcreteFraming}, define $v_k\co \Sigma_0\setminus S \to X$ and $\tilde v_{\kappa,k}\co \Sigma_0\setminus S \to X$ by
  \begin{align*}
    v_k
    &\coloneq
      u_k \circ \iota_k^{-1}\circ \Psi(\,\cdot\,;\sigma_k,\tau_k)\circ\iota_0  \qandq \\
    \tilde v_k
    &\coloneq
      \tilde u_{\sigma_k,\tau_k;J_k;0} \circ \iota_k^{-1}\circ \Psi(\,\cdot\,;\sigma_k,\tau_k)\circ\iota_0,
  \end{align*}
  cf. \autoref{Def_GromovConvergence}.
  Both of the sequences $(v_k)_{k\in \N}$ and $(\tilde v_k)_{k \in \N}$ converge to $u_0\co \Sigma_0\setminus S \to X$ in the $C^\infty_\loc$ topology---the former by \autoref{Def_GromovConvergence} and the latter by construction.

  With the notation of \autoref{Rmk_DecompositionOfSurface}
  for $r > 0$ and $n \in S$ set
  \begin{equation*}
    N_{k,n}^r \coloneq N_{\sigma_k,\tau_k;n}^r.    
  \end{equation*}  
  Choose $r > 0$ as in \autoref{Prop_SmallEnergyNeck} with $\delta \coloneq \frac18\inj_g(X)$.
  By the preceding paragraph,
  the assertion holds for sufficiently large $k$ and $z \notin N_{k,n}^r$.
  By \autoref{Prop_SmallEnergyNeck} and by construction of $\tilde u_{\sigma,\tau}$,
  for sufficiently large $k$
  \begin{equation*}
    u_k(N_{k,n}^r) \subset B_\delta(u_0(n)) \qandq
    \tilde u_{\sigma_k,\tau_k;J_k;0}(N_{k,n}^r) \subset B_\delta(u_0(n));
  \end{equation*}
  hence,
  for every $z \in N_{k,n}^r$
  \begin{equation*}
    u_k(z) \in \frac12U_{\tilde u_{\sigma_k,\tau_k;J_k;0}(z)}.
  \end{equation*}
  
  \begin{step}
    There is a constant $c > 0$ such that the sections $\zeta_{k;\kappa}$ defined in the preceding step satisfy
    \begin{equation*}
      \limsup_{k\to \infty}\Abs{\zeta_{k;\kappa}}_{W^{1,p}} \leq c\abs{\kappa}.
    \end{equation*}
  \end{step}

  By \autoref{Eq_KuranishiModel_Step1_Estimate}, \autoref{Eq_KuranishiModel_Step2_Estimate}, and \autoref{Prop_ErrorEstimate},
  we can restrict to $\kappa = 0$.
  Furthermore, it suffices to prove that for every $n \in S$
  \begin{equation}
    \label{Eq_ZetaBoundedByKappa}
    \lim_{s\downarrow 0}\limsup_{k\to\infty} \Abs{\zeta_{k;0}}_{W^{1,p}(N_{k,n}^s)} = 0.
  \end{equation}
  The case when $n$ is not smoothed is straightforward.
  The framing extends to identify a neighborhood of $n$ in $\Sigma_{\sigma_k,\tau_k}$ with a neighborhood of $n$ in $\Sigma_0$.
  It follows from \autoref{Lem_EpsilonRegularity} and elliptic regularity that, on this subset, the maps $u_k$ converge to $u_0$ in the $C_\loc^\infty$ topology.
  Let us therefore assume that $n$ is smoothed out;
  that is: $\epsilon_{k;n} \neq 0$ for sufficiently large $k$.

  Define $\rho_k \in C^\infty(N_{k;n}^r,T_{u_0(n)}X)$ and $\tilde \rho_k \in C^\infty(N_{k;n}^r,T_{u_0(n)}X)$ by
  \begin{equation*}
    \rho_k \coloneq {\exp_{u_0(n)}^{-1}}\circ u_k \qandq
    \tilde \rho_{k} \coloneq {\exp_{u_0(n)}^{-1}}\circ \tilde u_{\sigma_k,\tau_k;J_k;0}.
  \end{equation*}
  By construction,
  \begin{equation*}
    \lim_{k\to\infty} \Abs{\tilde \rho_k}_{W^{1,p}} = 0.
  \end{equation*}
  Therefore, it suffices to prove that
  \begin{equation*}
    \lim_{s\downarrow 0}\limsup_{k\to\infty} \Abs{\rho_k}_{W^{1,p}(N_{k,n}^s)} = 0
  \end{equation*}
  
  As explained in \autoref{Rmk_DecompositionOfSurface},
  the subset $N_{k;n}^r$ is biholomorphic to the cylinder
  \begin{equation*}
    S^1\times (-L_k,L_k) \qwithq L_k \coloneq \log(r\epsilon_{n;k}^{-1/2}).
  \end{equation*}
  Hence, $\rho_k$ can be thought of as a map $\rho_k^\cyl \co S^1 \times (-L_k,L_k) \to T_{u_0(n)}X$.
  More concretely, the canonical chart $\phi_n$ defines a holomorphic embedding
  \begin{equation*}
    \phi_n \co \set[\big]{ v \in T_n\Sigma_0 : \epsilon_{n;k}^{1/2} \leq \abs{v} < r } \to N_{k;n}^r
  \end{equation*}
  which glues via $\iota_\tau$ with the embedding $\phi_{\nu(n)}$ to a biholomorphic map
  \begin{equation*}
    B_r(0) \setminus \bar B_{\epsilon_{n;k}/r}(0) \iso N_{k;n}^r.
  \end{equation*}
  Choose identifications $T_n\Sigma_0 \iso \C \iso T_{\nu(n)}\Sigma_0$ such that $\iota_\tau(z) = \epsilon_n/z$.
  The map $\rho_k^\cyl$ is then defined by
  \begin{equation*}
    \rho_k^\cyl(\theta,\ell)
    \coloneq
    \begin{cases}
      \rho_k\circ \phi_n\paren[\big]{\epsilon_{n;k}^{1/2}e^{\ell+i\theta}} & \text{if } t \geq 0 \\
      \rho_k\circ \phi_{\nu(n)}\paren[\big]{\epsilon_{n;k}^{1/2}e^{-\ell-i\theta}} & \text{if } t \leq 0.
    \end{cases}
  \end{equation*}
  
  Since $u_k$ is $J_k$--holomorphic,
  $\rho_k^\cyl$ is $\exp_{u_0(n)}^*(J_k)$--holomorphic.
  Since the energy is conformally invariant, 
  \begin{equation*}
    E\paren[\big]{\rho_k^\cyl} = E\paren[\big]{u_{k|N_{k;n}^r}}.
  \end{equation*}
  Choose $\mu \in (1-2/p,1)$.
  By \autoref{Lem_EnergyDecayOnCylinders},
  \begin{equation*}
    \abs{\nabla\rho_k^\cyl(\theta,\ell)}
    \leq
    c
    e^{-\mu\paren{L_k-\abs{\ell}}}
    E\paren[\big]{u_k|_{N_{k;n}^r}}^{1/2}.
  \end{equation*}
  By the above and \autoref{Prop_MetricsAreUniformlyComparable},
  for $z \in \Sigma_{\sigma,\tau}^\circ$ with $r_n(z) < r$
  \begin{equation*}
    \abs{\nabla\rho_k(z)}
    \leq
    c r^{-\mu}r_n(z)^{\mu-1}
    E\paren[\big]{u_k|_{N_{k;n}^r}}^{1/2}.
  \end{equation*}
  There is a corresponding estimate with $n$ replaced with $\nu(n)$.
  Hence,
  \begin{equation*}
    \Abs{\nabla\rho_k(z)}_{L^p(N_{k;n}^s)}^p
    \leq
    \frac{cr^{-\mu p}}{(\mu-1)p+2}
    s^{(\mu-1)p+2}
    E\paren[\big]{u_k|_{N_{k;n}^r}}^{p/2}.
  \end{equation*}
  Since $(\mu-1)p + 2 > 0$,
  the right-hand converges to zero as $s$ converges to zero.
\end{proof}

\begin{proof}[Proof of \autoref{Prop_KuranishiModelCoversGromovNeighborhood}]
  Let $k \geq K$ and $\kappa \in \ker\fd_{u_0}$ with $\abs{\kappa} < \delta_\kappa$.
  Let $\zeta_{k;\kappa}$ be as in \autoref{Prop_GromovNeighborhoodGraph}.
  By \autoref{Prop_FuseKComplementsImR} the latter can be uniquely written as  
  \begin{equation*}
    \zeta_{k,\kappa} = \fuse_{\tau_k}(\lambda_{k;\kappa}) + \pr_1\circ\fr_{\tilde u_{\sigma_k,\tau_k;J_k;\kappa}}\eta_{k;\kappa}
    \qwithq \lambda_{k;\kappa} \in \ker\fd_{u_0}.
  \end{equation*}
  It remains to be proved that after possibly increasing $K$ for every $k \geq K$ there exists a $\kappa \in \ker \fd_{u_0}$ with $\abs{\kappa} < \delta_\kappa$ and $\lambda_{k;\kappa} = 0$.
  The following statement is a consequence of \autoref{Eq_TildeUSigmaTauKappaTildeUSigmaTau0},  \autoref{Eq_ZetaBoundedByKappa}, and the fact that $\fr_{\tilde u_{\sigma_k,\tau_k;J_k;\kappa}}$ depends smoothly on $\kappa$ when interpreted as a family of operators on a fixed Banach space $L^p\Omega^{0,1}(\Sigma_{\sigma_k,\tau_k}, \tilde u_{\sigma_k,\tau_k;J_k;0}^*TX)\oplus\sO$ using parallel transport along geodesics.
  If $\delta_\kappa$ is sufficiently small, then for every $\kappa$, $u_k$ can be written in the form
  \begin{equation*}
    u_k = \exp_{\tilde u_{\sigma_k,\tau_k;J_k,0}}\paren[\big]{\fuse_{\tau_k}(\kappa + \lambda_{k;\kappa}) + \fr_{\tilde u_{\sigma_k,\tau_k;J_k;0}}\hat\eta_{k;\kappa} + \fe_{k;\kappa}},
  \end{equation*}
  with $\fe_{k;\kappa}$ satisfying $\limsup_{k\to\infty}\Abs{\fe_{k;0}}_{W^{1,p}} = 0$ and a quadratic estimate
  \begin{equation}
    \label{Eq_ErrorTermIsQuadraticInKappa}
    \Abs{\fe_{k;\kappa_1} - \fe_{k;\kappa_2}}_{W^{1,p}} \leq c \paren*{\abs{\kappa_1} + \abs{\kappa_2}}\abs{\kappa_1-\kappa_2}
  \end{equation}
  It follows from \autoref{Prop_FuseKComplementsImR}  that for $\abs{\kappa} \leq \delta_\kappa$, 
  \begin{equation*}
    \kappa + \lambda_{k;\kappa} + \pi(\fe_{k;\kappa}) = 0,
  \end{equation*}
  where $\pi$ denotes the projection on $\fuse_{\tau_k}(\ker\fd_{u_0})$ precomposed $\fuse_{\tau_k}^{-1}$;
  the latter is defined because $\fuse_{\tau_k}$ is injective on $\ker\fd_{u_0}$ provided $k$ is sufficiently large.
  Thus, the existence of a unique small $\kappa$ such that $\lambda_{k;\kappa} = 0$ is a consequence of \autoref{Eq_ErrorTermIsQuadraticInKappa} and the Banach fixed point theorem applied to the map $\kappa \mapsto -\pi(\fe_{k;\kappa})$. 
\end{proof}

\subsection{The leading order term of the obstruction on ghost components}
\label{Sec_LeadingOrderTermOfObstructionGhostComponents}

Assume the situation of \autoref{Sec_KuranishiModel}.
The purpose of this subsection is to analyze the leading order term of part of the obstruction map $\ob$ constructed in \autoref{Sec_KuranishiModel}.
This construction requires a choice of partition of $S$ and a choice of lift $\sO \subset L^p\Omega^{0,1}(\Sigma_0,u_0^*TX)$ of $\coker\fd_{u_0}$. 
The following paragraphs introduce a particular choice tailored to the upcoming discussion.

Let $C \subset \Sigma_0$  be a ghost component of $u_0$; see \autoref{Sub_GhostComponents} for the definitions of a ghost component and related notation.
Denote by
\begin{equation*}
  x_0 \in X
\end{equation*}
the constant value which $u_0$ takes on $C$.

To simplify the upcoming discussion, we will make the following assumption, which will be satisfied in the situation considered in the proof of \autoref{Thm_LimitConstraints}.

\begin{hypothesis}
  \label{Hyp_BubbleAtOnePoint}
  $S_{C}^\rext$ consists of one point, that is: 
  $C$ and $\Sigma_0\setminus C$ meet at one node.
\end{hypothesis}

Denote by $B \subset C$ the base-locus of the dualizing sheaf of $\check C$, cf.~\autoref{Prop_BaseLocusOfDualizingSheaf}.
If $B$ does not contain the node at which $C$ and $\Sigma_0\setminus C$ meet,
then set $B_0 \coloneq \emptyset$;
otherwise,
denote by $B_0$ the connected component of $B$ containing the node.
Set
\begin{equation*}
  C_\basefree \coloneq C \setminus B_0 \qandq \Sigma_\clubsuit \coloneq \Sigma_0 \setminus C_\basefree
\end{equation*}
and abbreviate
\begin{equation*}
  \nu_\basefree
  \coloneq
  \nu_{C_\basefree}
  \qandq
  \nu_\clubsuit
  \coloneq
  \nu_{\Sigma_\clubsuit}.
\end{equation*}
The significance of this construction is as follows.
By \autoref{Prop_BaseLocusOfDualizingSheaf},
every connected component of $C_\basefree$ attaches to $\Sigma_\clubsuit$ at a unique node;
moreover: these nodes are not contained in the base-locus of the dualizing sheaf of $C_\basefree$.

The partition of the set of nodes we choose is,
with the notation from \autoref{Sub_GhostComponents},
\begin{equation}
  \label{Eq_SetsOfNodes}
  S = S_1 \amalg S_2
  \qwithq
  S_1
  \coloneq
  S_{\Sigma_\clubsuit}^\rint
  \amalg
  S_{C_\basefree}^\rint
  \qandq
  S_2
  \coloneq
  S_{\Sigma_\clubsuit}^\rext
  \amalg
  S_{C_\basefree}^\rext,
\end{equation}
that is: first, we will smooth the interior nodes of $\Sigma_\clubsuit$ and $C_\bullet$ and then the exterior nodes connecting them.

Next we  discuss the choice of the obstruction space $\sO$.
Set $\fd_{u_0,\clubsuit} \coloneq \fd_{u_0|_{\Sigma_\clubsuit}}$ and let
\begin{equation*}
  \sO_\clubsuit
  \subset
  \Omega^{0,1}(\Sigma_\clubsuit,u_0^*TX)
  \subset
  L^p\Omega^{0,1}(\Sigma_\clubsuit,u_0^*TX)
\end{equation*}
be a lift of $\coker\fd_{u_0,\clubsuit}$ such that every $o \in \sO_\clubsuit$ vanishes in a neighborhood of $S_{\Sigma_\clubsuit}^\rext$ if $B_0$ is empty, and over all of $B_0$ if $B_0$ is non-empty. 
Furthermore,
let
\begin{equation*}
  \sO_\basefree
  \subset
  \Omega^{0,1}(C_\basefree,\C)\otimes_\C T_{x_0}X
  \subset
  L^p\Omega^{0,1}(C_\basefree,u_0^*TX)
\end{equation*}
be a lift of $\coker\paren{\delbar\otimes_\C\one}$.

Every $\xi \in  W^{1,p}\Gamma(\Sigma_\clubsuit,\nu_\clubsuit;u_0^*TX)$ can be extended to $\Sigma_0$ in the following way.
Given $n \in S_{\Sigma_\clubsuit}^\rext$, extend $\xi$ to a constant section taking value $\xi(n)$ over the connected component of the nodal curve $(C_\basefree,\nu_\basefree)$ containing $\nu(n)$. 
This defines an inclusion
\begin{equation}
  \label{Eq_ExtendW1PStar}
  W^{1,p}\Gamma(\Sigma_\clubsuit,\nu_\clubsuit;u_0^*TX) \subset W^{1,p}\Gamma(\Sigma_0,\nu_0;u_0^*TX).
\end{equation}
Furthermore,
extension by zero defines inclusions
\begin{equation*}
  L^p\Omega^{0,1}(\Sigma_\clubsuit,u_0^*TX) \subset L^p\Omega^{0,1}(\Sigma_0,u_0^*TX) \qandq
  L^p\Omega^{0,1}(C_\basefree,u_0^*TX) \subset L^p\Omega^{0,1}(\Sigma_0,u_0^*TX).
\end{equation*}
Set
\begin{equation*}
  \sO \coloneq \sO_\clubsuit \oplus \sO_\basefree.
\end{equation*}

\begin{prop}
  \label{Prop_DescriptionOfKernelAndCokernel}
   The map \autoref{Eq_ExtendW1PStar} induces an isomorphism $\ker\fd_{u_0,\clubsuit} \iso \ker \fd_{u_0}$ and $\sO$ is a lift of $\coker\fd_{u_0}$.
\end{prop}

\begin{proof}
  Denote by $\nu_\amalg$ the nodal structure on $\Sigma_0$ which agrees with $\nu_0$ on the complement of $S_2$ and is the identity on $S_2$.
  This nodal structure disconnects $\Sigma_\clubsuit$ and $C_\basefree$.
  Denote by
  \begin{equation*}
    \fd_{u_0,\amalg} \co W^{1,p}\Gamma(\Sigma_0,\nu_\amalg;u_0^*TX) \to L^p\Omega^{0,1}(\Sigma_0,u^*TX)
  \end{equation*}
  the operator induced by $\fd_{u_0}$.
  Define $V_-$ and $\diff\co \ker\fd_{u_0,\amalg} \to V_-$ as in \autoref{Rmk_NodalKernelAndCokernel} with $S_2$ instead of $S$.
  As is explained in \autoref{Rmk_NodalKernelAndCokernel},
  \begin{equation*}
    \ker\fd_{u_0} = \ker\diff
  \end{equation*}
  and there is a short exact sequence
  \begin{equation*}
    \begin{tikzcd}[column sep=small]
      0 \ar[r] & \coker\diff  \ar[r] & \coker\fd_{u_0} \ar[r] & \coker\fd_{u_0,\amalg} \ar[r] & 0.
    \end{tikzcd}
  \end{equation*}

  The domain and codomain of $\fd_{u_0,\amalg}$ decompose as
  \begin{align*}
    W^{1,p}\Gamma(\Sigma_0,\nu_\amalg;u_0^*TX)
    &=
      W^{1,p}\Gamma(\Sigma_\clubsuit,\nu_\clubsuit;u_0^*TX) \oplus W^{1,p}\Gamma(C_\basefree,\nu_\basefree;\C)\otimes_\C T_{x_0} X \qand
    \\
    L^p\Omega^{0,1}(\Sigma_0,u_0^*TX)
    &=
      L^p\Omega^{0,1}(\Sigma_\clubsuit,u_0^*TX) \oplus L^p\Omega^{0,1}(C_\basefree,\C)\otimes_\C T_{x_0} X.
  \end{align*}
  With respect to these decompositions
  \begin{equation*}
    \fd_{u_0,\amalg}
    =
    \begin{pmatrix}
      \fd_{u_0,\clubsuit} & 0 \\
      0 & \delbar\otimes_\C\one
    \end{pmatrix}
  \end{equation*}
  with $\fd_{u_0,\clubsuit}=\fd_{u_0|_{\Sigma_\clubsuit}}$ and $\delbar\otimes_\C\one = \fd_{u_0|_{C_\basefree}}$ is the standard Cauchy--Riemann operator since $u_0$ is constant on $C_\basefree$. 
  Therefore,
  \begin{equation*}
    \ker \fd_{u_0,\amalg} = \ker \fd_{u_0,\clubsuit} \oplus \ker(\delbar\otimes_\C\one )\qandq
    \coker \fd_{u_0,\amalg} = \coker \fd_{u_0,\clubsuit} \oplus \coker\paren{\delbar\otimes_\C \one}.  
  \end{equation*}
  The task at hand is to understand $\ker\fd_{u_0}$ and $\coker \fd_{u_0}$ in terms of the above.

  Since elements of $\ker(\delbar\otimes_\C\one)$ are locally constant,
  $\ker(\delbar\otimes_\C\one)$ has a direct summand $T_{x_0} X$ for every connected component of $(C_\basefree,\nu_\basefree)$. 
  \autoref{Hyp_BubbleAtOnePoint} and \autoref{Prop_BaseLocusOfDualizingSheaf} imply that there is one connected component for each node in $S_{\Sigma_\clubsuit}^\rext$.
  Therefore,
  \begin{equation*}
    \ker(\delbar\otimes_\C\one) = V_- = \Map(S_{\Sigma_\clubsuit}^\rext,T_{x_0}X).
  \end{equation*}
  With respect to this identification the map $\diff\co \ker \fd_{u_0,\clubsuit}\oplus  \ker(\delbar\otimes_\C\one) \to V_-$ is given by
  \begin{equation*}
    \diff(\kappa,v)(n) = \kappa(n) - v(n).
  \end{equation*}
  Therefore, $\ker\diff \iso \ker \fd_{u_0,\clubsuit}$ and  $\coker\diff = \set{0}$, which, by \autoref{Rmk_NodalKernelAndCokernel}, completes the proof of the proposition. 
\end{proof}

Construct the Kuranishi model as in \autoref{Sec_KuranishiModel} for the above choices of $S = S_1 \amalg S_2$ and $\sO$.

As a final piece of preparation,
let us make the following observation,
which by \autoref{Rmk_NodalKernelAndCokernel}, in particular,
gives an explicit description of $\sO_\basefree^* = \coker\paren{\delbar\otimes_\C\one}^*$.

\begin{prop}
  \label{Prop_HGamma=SectionsOfDualizingSheaf}
  Let $(C,\nu)$ be a nodal Riemann surface with nodal set $S$.
  Denote the corresponding nodal curve by $\check C$ and its dualizing sheaf by $\omega_{\check C}$.
  Let $q \in (1,2)$ be such that $1/p+1/q = 1$.
  Define 
  \begin{equation*}
    \sH \subset L^q\Omega^{0,1}(C,\C)
  \end{equation*}
  to be the subspace of solutions  $\bar\zeta$ of the distributional equation
  \begin{equation}
    \label{Eq_Delbar*Eta=Deltas}
    \delbar^*\bar\zeta = \sum_{n \in S} f(n)\delta(n)
  \end{equation}
  for some weight function $f \co S \to \C$ with $f\circ \nu = -f$.
  Here $\delta(n)$ is the Dirac delta distribution at $n$.
  The subspace $\sH$ satisfies
  \begin{equation*}    
    \overline \sH= H^0(\check C,\omega_{\check  C}).
  \end{equation*}
\end{prop}

\begin{proof}
  If $\check C$ is smooth,
  then $\check C = C$ and the dualizing sheaf $\omega_C$ is simply the canonical sheaf $K_C$.
  By the Kähler identities,
  \begin{align*}
    \overline \sH
    &=
      \overline{\ker\paren[\big]{\delbar^*\co \Omega^{0,1}(C,\C) \to \Omega^0(C,\C) }} \\
    &=
      \overline{\ker\paren[\big]{\del\co \Omega^{0,1}(C,\C) \to \Omega^{1,1}(C,\C) }} \\
    &\iso
      \ker\paren[\big]{\delbar\co \Omega^{1,0}(C,\C) \to \Omega^{1,1}(C,\C)} \\
    &\iso
      H^0(C,K_C).
  \end{align*}
  Recall from the proof of \autoref{Prop_BaseLocusOfDualizingSheaf},
  that the dualizing sheaf of $\check C$ is constructed as follows;
  Denote by $\pi\co C \to \check C$ the normalization map.
  Denote by $\tilde \omega_{\check C}$ the subsheaf of $K_C(S)$ whose sections $\zeta$ satisfy
  \begin{equation*}
    \Res_n \zeta + \Res_{\nu(n)}\zeta = 0
  \end{equation*}
  for every $n \in S$, with $\Res_n\eta$ being the residue of the meromorphic $1$--form $\eta$ at $n$. 
  The dualizing sheaf $\omega_{\check C}$ then is
  \begin{equation*}
    \omega_{\check C} = \pi_*\tilde \omega_{\check C}.
  \end{equation*}
  Therefore,
  $H^0(\check C,\omega_{\check C}) = H^0(C,\tilde\omega_{\check C})$.
  By definition every $\zeta \in H^0(C,\tilde\omega_{\check C})$ is smooth away from $S$ and
  blows-up at most like $1/\mathrm{dist}(n,\cdot)$ at $n$ for $n \in S$; hence: $\zeta \in L^q\Omega^{0,1}(C,\C)$.
  The residue condition amounts to \autoref{Eq_Delbar*Eta=Deltas}.
  This shows that $H^0(\check C,\omega_{\check C}) \subset \sH$.
  Conversely, by elliptic regularity every $\zeta \in \sH$ defines an element of $H^0(\check C,\omega_{\check C})$.
\end{proof}

The following is the technical backbone of the proof of \autoref{Thm_LimitConstraints}.
The reader is advised to recall \autoref{Def_ObstructionMap} and \autoref{Prop_KuranishiModelCoversGromovNeighborhood} because these are the main ingredients of the proof.

\begin{lemma}
  \label{Lem_KuranishiModelGhostComponent}
  Denote by $\check C_\basefree$ the nodal curve corresponding to $(C_\basefree,\nu_\basefree)$.
  There is a constant $c > 0$ such that the obstruction map defined in \autoref{Def_ObstructionMap} satisfies the following. For every $(\sigma,\tau;J;\kappa) \in \Delta\times\sU\times\sI$,
  $\zeta \in H^0(\check C_\basefree,\omega_{\check C_\basefree})$,
  and $v \in T_{x_0}X$
  \begin{equation*}
    \Inner[\big]{
      \pull_{\sigma,\tau;J;\kappa}(\ob(\sigma,\tau;J;\kappa)),
      \pull_{\sigma,\tau;J;\kappa}(\bar\zeta\otimes_\C v)
    }_{L^2}    
    =
    \sum_{n \in S_{C_\basefree}^\rext}
    \pi
    \Inner[\big]{
      \paren*{\zeta(n)\otimes_\C \rd_{\nu_0(n)}u_{\sigma,\tau_1,0;J;\kappa}}(\tau_n),
      v
    }
    +
    \fe
  \end{equation*}
  with
  \begin{align*}
    \abs{\fe}
    &\leq
      c\abs{\zeta}\abs{v}
      \epsilon_\basefree^{\frac12}
      \sum_{n \in S_{C_\basefree}^\rext}
      \epsilon_n \\
    \intertext{and}
    \epsilon_\basefree
    &\coloneq
      \max\set[\big]{ \epsilon_n : n \in S_{C_\basefree}^\rint \cup S_{C_\basefree}^\rext }.
  \end{align*}
\end{lemma}

\begin{example}
  To understand better the significance of \autoref{Lem_KuranishiModelGhostComponent}, it is helpful to consider the following example. 
  Suppose that $(\Sigma_0,j_0,\nu_0)$ consists of two components: the higher genus ghost $C = C_\bullet$ and a spherical bubble $\Sigma_\clubsuit = \CP^1$ meeting at points
  \begin{equation*}
    n \in C \qandq \nu(n) \in \CP^1.
  \end{equation*}
 (In that case, $B=B_0=\emptyset$.) 
  In that case, there is only one smoothing parameter $\tau$, which, after trivializing the tangent spaces of $C$ and $\Sigma_\clubsuit$ at the nodes, can be thought of as a complex number, and $\epsilon = \abs{\tau}$. 
  By \autoref{Lem_KuranishiModelGhostComponent}, for every holomorphic $1$-form $\zeta \in H^0(C,\omega_C)$, 
\begin{equation}
  \label{Eq_KuranishiMapEstimate}
    \Inner[\big]{
      \pull_{\sigma,\tau;J;\kappa}(\ob(\sigma,\tau;J;\kappa)),
      \pull_{\sigma,\tau;J;\kappa}(\bar\zeta\otimes_\C v)
    }_{L^2}    
    =
    \pi
    \Inner[\big]{
      \paren*{\zeta(n)\otimes_\C \rd_{\nu_0(n)}u_{\sigma,0;J;\kappa}}(\tau),
      v
    }
    +
    \fe
  \end{equation}
  with
  \begin{equation*}
    \abs{\fe} \leq c\abs{\zeta}\abs{v}\epsilon^{3/2}.
  \end{equation*}
  Since $C$ has positive genus, there exists $\zeta \in H^0(C,\omega_c)$ such that $\zeta(n) \neq 0$. 
   If the restriction of $u_0$ to the bubble $\CP^1$ is unobstructed, then $u_{\sigma,0;J;\kappa} = u_0$ on $\CP^1$ for all $\sigma$ and $\kappa$.
   If, moreover, $\rd_{\nu(n)}u_0 \neq 0$, it follows that  the right-hand side of \eqref{Eq_KuranishiMapEstimate} is never zero unless $\epsilon=0$, and so $\ob(\sigma,\tau,J;\kappa) \neq 0$ for $\tau\neq 0$. 
   We conclude that in that case $u_0$ cannot be smoothed.
\end{example}

\begin{proof}[Proof of \autoref{Lem_KuranishiModelGhostComponent}]
  The proof is based on analyzing the expression
  \begin{equation*}
    0
    =
    \inner{
      \fF_{\tilde u_{\sigma,\tau;J;\kappa}}(\hat\xi(\sigma,\tau;J;\kappa))
      + \pull_{\sigma,\tau;J;\kappa}\paren*{o(\sigma,\tau_1;J;\kappa) + \hat o(\sigma,\tau;J;\kappa)}}{\pull_{\sigma,\tau;J;\kappa}(\bar\zeta\otimes_\C v)}_{L^2}
  \end{equation*}
  and the identity
  \begin{equation*}
    \ob(\sigma,\tau;J;\kappa) \coloneq o(\sigma,\tau_1;J;\kappa) + \hat o(\sigma,\tau;J;\kappa).
  \end{equation*}

  \setcounter{step}{0}
  \begin{step}
    \label{Step_KappaConstantOnC}
    The vector field $\xi(\sigma,\tau_1;J;\kappa)$ is constant on $C_\basefree$ and $o(\sigma,\tau_1;J;\kappa)$ is supported on $\Sigma_\clubsuit$;
    in particular,
    \begin{equation*}
      \inner{\pull_{\sigma,\tau;J;\kappa}(o(\sigma,\tau_1;J;\kappa))}{\pull_{\sigma,\tau;J;\kappa}(\bar\zeta\otimes_\C v)}_{L^2}
      =
      0.
    \end{equation*}
  \end{step}

  The construction in \autoref{Prop_KuranishiModel_Step1} can be carried out for $u_0|_{\Sigma_\clubsuit}$ with the choice of $\sO=\sO_\clubsuit$ and $S_2 = \emptyset$.
  For every $(\sigma,\tau_1,0;J) \in \Delta\times\sU$ and $\kappa \in \ker \fd_{u_0,\star}$ with $\abs{\kappa} < \delta_\kappa$ denote by $\xi(\sigma,\tau_1;J;\kappa)$ and $o(\sigma,\tau_1;J;\kappa)$ the solution of \autoref{Eq_KuranishiModel_Step1} obtained in this way.

  Henceforth,
  regard $\xi(\sigma,\tau_1;J;\kappa)$ as an element of $W^{1,p}\Gamma(\Sigma_0\nu_0;u_0^*TX)$ and $o(\sigma,\tau_1;J;\kappa)$ as an element of $\sO$.
  By construction these satisfy \autoref{Eq_KuranishiModel_Step1} for $u_0$ and with the choices of $\sO$ and $S = S_1 \amalg S_2$ made in the discussion preceding \autoref{Lem_KuranishiModelGhostComponent}. 
  Therefore and since $\ker \fd_{u_0,\star} = \ker \fd_{u_0}$,
  $\xi(\sigma,\tau_1;J;\kappa)$ and $o(\sigma,\tau_1;J;\kappa)$ are precisely the output produced by \autoref{Prop_KuranishiModel_Step1}.
  The first part of the assertion thus holds by construction.

  \begin{step}
    \label{Step_ExpansionOfDelbarTildeUSigmaTauKappa}
    The term
    \begin{equation*}
      \delbar_J(\tilde u_{\sigma,\tau;J;\kappa},j_{\sigma,\tau}) + \pull_{\sigma,\tau;J;\kappa}(o(\sigma,\tau_1;J;\kappa))
    \end{equation*}
    is supported in the regions where $r_n \leq 2R_0$ for some $n \in S_{C_\basefree}^\rext$.
    Set $x_n \coloneq u_{\sigma,\tau_1,0;J;\kappa}(n)$.
    Identifying $U_{x_n}$ with $\tilde U_{x_n}$ via $\exp_{x_n}$ in the region where $r_n \leq 2R_0$ the error term can be written as
    \begin{equation*}
      \delbar \chi_{\tau_2}^n \cdot u_{\sigma,\tau_1,0;J;\kappa}\circ \iota_{\tau_2}
      + \fq
    \end{equation*}
    with
    \begin{equation}
      \label{Eq_FirstErrorEstimate}
      \abs{\delbar \chi_{\tau_2}^n \cdot u_{\sigma,\tau_1,0;J;\kappa}\circ \iota_{\tau_2}}
      \leq
      c
      \epsilon_n
      \qandq
      \abs{\fq}
      \leq
      c
      \epsilon_n^2.
    \end{equation}
  \end{step}

  The proof is a refinement of that of \autoref{Prop_ErrorEstimate}.
  A priori,
  the error term $\delbar_J(\tilde u_{\sigma,\tau;J;\kappa},j_{\sigma,\tau}) +  \pull_{\sigma,\tau;J;\kappa}(o(\sigma,\tau_1;J;\kappa))$ is supported in the in the regions where $r_n \leq 2R_0$ for some $n \in S_2$.
  If $n \in S_{\Sigma_\clubsuit}^\rext$,  
  then it is immediate from \autoref{Def_TildeUTau} that $\tilde u_{\sigma,\tau;J;\kappa}$ agrees with $u_{\sigma,\tau_1,0;J;\kappa}$ in the region under consideration;
  hence, the error term vanishes.
  For $n \in S_{C_\basefree}^\rext$,  
  in the region under consideration and with the identifications having been made,
  \begin{equation*}
    \tilde u_{\sigma,\tau;J;\kappa}^\circ = \chi_{\tau_2}^n \cdot u_{\sigma,\tau_1,0;J;\kappa} \circ \iota_{\tau_2}.
  \end{equation*}
  Therefore,
  \begin{align*}
    \delbar_J (\tilde u_{\sigma,\tau;J;\kappa}^\circ,j_{\sigma,\tau})
    &=
      \delbar \chi_{\tau_2}^n \cdot u_{\sigma,\tau_1,0;J;\kappa} \circ \iota_{\tau_2} \\
    &\quad+
      \underbrace{\chi_{\tau_2}^n \cdot \frac12 \paren[\big]{J(\tilde u_{\sigma,\tau;J;\kappa}^\circ)-J(u_{\sigma,\tau_1,0;J;\kappa} \circ \iota_{\tau_2})}\circ \rd (u_{\sigma,\tau_1,0;J;\kappa}\circ\iota_{\tau_2})\circ j_{\sigma,\tau_1,0}}_{\eqcolon \rI} \\
    &\quad+      
      \underbrace{\chi_{\tau_2}^n\cdot\delbar_J (u_{\sigma,\tau_1,0;J;\kappa}\circ \iota_{\tau_2},j_{\sigma,\tau_1,0})}_{\eqcolon \rII}.
  \end{align*}
   (Observe  that by elliptic regularity and \autoref{Eq_KuranishiModel_Step1}, the map $u_{\sigma,\tau;J;\kappa}$ is smooth in the region in question, so we can take its derivative.
   We will use this fact in the remaining part of the proof.)
  The term $\rI$ is supported in the region where $R_0 \leq r_n \leq 2R_0$.
  By Taylor expansion at $\nu_0(n)$,
  in this region
  \begin{align*}
    \abs{u_{\sigma,\tau_1,0;J;\kappa}\circ\iota_{\tau_2}}
    &\leq
      c\epsilon_n/r_n \qandq \\
    \abs{\rd(u_{\sigma,\tau_1,0;J;\kappa}\circ\iota_{\tau_2})}
    &\leq
      c\epsilon_n/r_n^2.
  \end{align*}
  Therefore,
  \begin{equation*}
    \abs{\rI}
    \leq
    c
    \epsilon_n^2
    \qandq
    \abs{\delbar \chi_{\tau_2}^n \cdot u_{\sigma,\tau_1,0;J;\kappa}\circ \iota_{\tau_2}}
    \leq
    c
    \epsilon_n.
  \end{equation*}
  Since $\iota_{\tau_2}$ is holomorphic and $o(\sigma,\tau_1;\kappa)$ is defined by \autoref{Eq_KuranishiModel_Step1}, 
  \begin{align*}
    \rII = \chi_{\tau_2}^n \cdot \iota_{\tau_2}^*\delbar_J (u_{\sigma,\tau_1,0;J;\kappa},j_{\sigma,\tau_1,0}) 
    = -  \chi_{\tau_2}^n \cdot \iota_{\tau_2}^*o(\sigma,\tau_1;J;\kappa),
  \end{align*}
  and thus $\rII$ vanishes by our choice of $\sO$.
  
  \begin{step}
    \label{Step_IntegralComputation}
    For every $n \in S_{C_\basefree}^\rext$
    \begin{equation*}
      \inner{\delbar \chi_{\tau_2}^n\cdot u_{\sigma,\tau_1,0;J;\kappa} \circ\iota_{\tau_2}}{\pull_{\sigma,\tau;J;\kappa}(\bar\zeta\otimes_\C v)}_{L^2}
      =
      -  
      \pi
      \Inner[\big]{
        \paren*{\zeta\otimes_\C \rd_{\nu_0(n)}u_{\sigma,\tau_1,0;J;\kappa}}(\tau_n),
        v
      }
      + \fe_2
    \end{equation*}
    with
    \begin{equation*}
      \abs{\fe_2}
      \leq
      c
      \epsilon_n^2\abs{\zeta}\abs{v}.
    \end{equation*}   
  \end{step}

  To simplify the notation,
  we choose an identification $T_nC = \C$ and work in the canonical holomorphic coordinate $z$ on $C$ at $n$ and the coordinate system at $\nu_0(n)$ with respect to which $w = \iota_{\tau_2}(z) = \epsilon_n/z$.
  In particular, with respect to the induced identification $T_{\nu_0(n)}C=\C$ the gluing parameter is simply $\tau_n = \epsilon_n \cdot 1\otimes_\C 1$.

  Since $u_{\sigma,\tau_1,0;J;\kappa}$ is $J$--holomorphic,
  by Taylor expansion,
  \begin{equation*}
    u_{\sigma,\tau_1,0;J;\kappa}(\epsilon_n/z)
    =
    \del_w u_{\sigma,\tau_1,0;J;\kappa}(0) \cdot \epsilon_n/z
    + \fr
    \qwithq
    \abs{\fr}
    \leq c\epsilon_n^2/\abs{z}^2.
  \end{equation*}
  The term
  \begin{equation*}
    \fe_2'
    \coloneq
    \inner{\delbar \chi_{\tau_2}^n\cdot \fr}{\pull_{\sigma,\tau;J;\kappa}(\zeta\otimes_\C v)}_{L^2}
  \end{equation*}
  satisfies
  \begin{equation*}
    \abs*{\fe_2'}
    \leq
    c \epsilon_n^2\abs{\zeta}\abs{v}.
  \end{equation*}

  Since $\zeta$ is holomorphic,
  \begin{equation*}
    \int_{S^1} \zeta(re^{i\alpha}) \,\rd \alpha
    =
    2\pi \cdot \zeta(0).
  \end{equation*}
  Therefore,  
  \begin{align*}
    \inner{\delbar \chi_{\tau_2}^n\cdot z^{-1}}{\bar\zeta}_{L^2}
    &=
      \int_{R_0\leq \abs{z} \leq 2R_0} \frac12\chi'\paren*{\frac{\abs{z}}{R_0}} \frac{\zeta(z)}{R_0\abs{z}}  \, \vol \\
    &=
      \int_{R_0}^{2R_0} \frac12\chi'\paren*{\frac{r}{R_0}} \frac{1}{R_0}\cdot\paren*{\int_{S^1} \zeta(r e^{i\alpha})\,\rd \alpha}\,\rd r \\
    &=
      \int_{R_0}^{2R_0} \chi'\paren*{\frac{r}{R_0}} \frac{1}{R_0} \,\rd r \cdot \pi \zeta(r e^{i\alpha}).
  \end{align*}
  The integral evaluates to $-1$.
  Thus the assertion follows because the term $\inner{\zeta(0)\cdot\del_wu_{\sigma,\tau_1,0;J;\kappa}(0)}{v}$ can be written in coordinate-free form as
  \begin{equation*}
    \pi
    \Inner[\big]{
      \paren*{\zeta\otimes_\C \rd_{\nu_0(n)}u_{\sigma,\tau_1,0;J;\kappa}}(\tau_n),
      v
    }.
  \end{equation*}

  \begin{step}
    \label{Step_ErrorEstimate}
    The term
    \begin{equation*}
      \fe_3
      \coloneq
      \inner{\fd_{\tilde u_{\sigma,\tau;J;\kappa}}\hat\xi(\sigma,\tau;J;\kappa)+ \fn_{\tilde u_{\sigma,\tau;J;\kappa}}(\hat\xi(\sigma,\tau;J;\kappa))}{\pull_{\sigma,\tau;J;\kappa}(\bar\zeta\otimes_\C v)}_{L^2}
    \end{equation*}
    satisfies
    \begin{equation*}
      \abs{\fe_3}
      \leq
      c
      \epsilon_\basefree^{\frac12}
      \sum_{n \in S_{C_\basefree}^\rext}
      \epsilon_n
      \abs{\zeta}\abs{v}
    \end{equation*}
  \end{step}
  
  By \autoref{Step_ExpansionOfDelbarTildeUSigmaTauKappa} and \autoref{Prop_KuranishiModel_Step2},
  \begin{equation*}
    \Abs{\hat\xi(\sigma,\tau;J;\kappa)}_{W^{1,p}}
    \leq
    c
    \epsilon_n.
  \end{equation*}
  This immediately implies that
  \begin{equation*}
    \fe_3'
    \coloneq
    \inner{\fn_{\tilde u_{\sigma,\tau;J;\kappa}}(\hat\xi(\sigma,\tau;J;\kappa))}{\pull_{\sigma,\tau;J;\kappa}(\bar\zeta\otimes_\C v)}_{L^2}
  \end{equation*}
  satisfies
  \begin{equation*}
    \abs{\fe_3'}
    \leq
    c
    \sum_{n \in S_{C_\basefree}^\rext}
    \epsilon_n^2
    \abs{\zeta}\abs{v}.
  \end{equation*}
  It remains to estimate
  \begin{equation*}
    \fe_3''
    \coloneq
    \Inner{
      \fd_{\tilde u_{\sigma,\tau;J;\kappa}}\hat\xi(\sigma,\tau;J;\kappa),
      \pull_{\sigma,\tau;J;\kappa}(\bar\zeta\otimes_\C v)
    }_{L^2}.
  \end{equation*}
  
  Set
  \begin{equation*}
    C_{\sigma,\tau}^\circ
    \coloneq
    C_\basefree \cap \Sigma_{\sigma,\tau}^\circ.
  \end{equation*}
  Since $\kappa$ and $\xi(\sigma,\tau_1;\kappa)$ are constant on $C$, \autoref{Prop_CompareDU1WithDU2} implies that
  the term 
  \begin{equation*}
    \fd_{\tilde u_{\sigma,\tau;J;\kappa}}\hat\xi(\sigma,\tau;J;\kappa) - \delbar\hat\xi(\sigma,\tau;J;\kappa),
  \end{equation*}
  defined over $C_{\sigma,\tau}^\circ$, is supported in the regions where $\epsilon_n^{1/2} \leq r_n \leq 2R_0$ for some $n \in S_{C_\basefree}^\rext$ and satisfies
  \begin{align*}
    &\sum_{n \in S_{C_\basefree}^\rext} \int_{\epsilon_n^{1/2} \leq r_n \leq 2R_0} \abs{\fd_{\tilde u_{\sigma,\tau;J;\kappa}}\hat\xi(\sigma,\tau;J;\kappa) - \delbar\hat\xi(\sigma,\tau;J;\kappa)} \\
    &\quad\leq
      c
      \sum_{n \in S_{C_\basefree}^\rext}
      \int_{\epsilon_n^{1/2} \leq r_n \leq 2R_0}
      \abs{\hat\xi(\sigma,\tau;J;\kappa)}\abs{\nabla\hat\xi(\sigma,\tau;J;\kappa)}
      + \abs{\hat\xi(\sigma,\tau;J;\kappa)}^2 \\
    &\quad\leq
      c
      \Abs{\hat\xi(\sigma,\tau;J;\kappa)}_{W^{1,p}}^2 \\
    &\quad\leq
      c
      \sum_{n \in S_{C_\basefree}^\rext}
      \epsilon_n^2.
  \end{align*}
  Therefore,
  \begin{equation*}
    \abs{\fe_3''}
    \leq
    c
    \sum_{n \in S_{C_\basefree}^\rext}
    \epsilon_n^2
    +
    \abs[\big]{
      \Inner[\big]{
        \delbar \hat\xi(\sigma,\tau;J;\kappa),
        \bar\zeta\otimes_\C v}_{L^2(C_{\sigma,\tau}^\circ)}}.
  \end{equation*}
  Since $\delbar^*\zeta = 0$ on $C_{\sigma,\tau}^\circ$,
  integration by parts yields
  \begin{align*}
    \abs[\big]{
    \Inner[\big]{
    \delbar \hat\xi(\sigma,\tau;J;\kappa),
    \bar\zeta\otimes_\C v}_{L^2(C_{\sigma,\tau}^\circ)}}
    &\leq
      c
      \epsilon_\basefree^{\frac12}
      \sum_{n \in S_{C_\basefree}^\rext}
      \Abs{\hat\xi(\sigma,\tau;J;\kappa)}_{W^{1,p}}\abs{\zeta}\abs{v} \\
    &\leq
      c
      \epsilon_\basefree^{\frac12}
      \sum_{n \in S_{C_\basefree}^\rext}
      \epsilon_n.
  \end{align*}
  Combining the above estimates yields the asserted estimate on $\fe_3$.
  
  \begin{step}
    \label{Step_Conclusion}
    Conclusion of the proof.
  \end{step}

  By \autoref{Step_KappaConstantOnC}, \autoref{Step_ExpansionOfDelbarTildeUSigmaTauKappa}, \autoref{Step_IntegralComputation}, and \autoref{Step_ErrorEstimate} the term
  \begin{equation*}
    \fo \coloneq \inner{\pull_{\sigma,\tau;J;\kappa}(\ob(\sigma,\tau;J;\kappa))}{\pull_{\sigma,\tau;J;\kappa}(\bar\zeta\otimes_\C v)}_{L^2}
  \end{equation*}
  satisfies
  \begin{align*}
    \fo
    &=
      -\inner{\fF_{\tilde u_{\sigma,\tau;J;\kappa}}(\hat\xi(\sigma,\tau;J;\kappa))}{\pull_{\sigma,\tau;J;\kappa}(\bar\zeta\otimes_\C v)}_{L^2} \\
    &=
      \sum_{n \in S_{C_\basefree}^\rext}
      \pi
      \Inner[\big]{
      \paren*{\zeta\otimes_\C \rd_{\nu_0(n)}u_{\sigma,\tau_1,0;J;\kappa}}(\tau_n),
      v
      }
      + \fe_1 + \fe_2 + \fe_3 
  \end{align*}
  with the error terms arising from the corresponding terms in the preceding steps.
  The term
  $\fe_1$ arises from the $L^2$ inner product of the sum, over all nodes $n \in S_{C_\basefree}^\rext$, of error terms from  \autoref{Step_ExpansionOfDelbarTildeUSigmaTauKappa} and $\pull_{\sigma,\tau;J;\kappa}(\bar\zeta\otimes_\C v)$, which multiplies the estimate \autoref{Eq_FirstErrorEstimate}  by $\abs{\xi}\abs{v}$.  
  The preceding steps thus yield the asserted estimate on $\fe = \fe_1+\fe_2+\fe_3$.
\end{proof}

\subsection{Proof of \autoref{Thm_LimitConstraints}}
\label{Sec_ProofOfLimitConstraints}

Without loss of generality, suppose that
\begin{equation*}
(\Sigma_k,j_k) = (\Sigma_{\sigma_k,\tau_k},j_{\sigma_k,\tau_k})
\end{equation*}
 with $\tau_{k;n} \neq 0$ for every $k \in \N$ and $n \in S$; that is: all nodes are smoothed and $\nu_{\sigma_k,\tau_k}$ is the trivial nodal structure.
By \autoref{Prop_GromovNeighborhoodGraph} there is a sequence $(\kappa_k)_{\kappa\in\N}$ in $\sI$ corresponding to $(u_k)_{k\in\N}$.

Again,
without loss of generality,
$u_\infty\co (\Sigma_\infty,j_\infty,\nu_\infty) \to (X,J_\infty)$ has at least one ghost component $C$ with one non-ghost component 
\begin{equation*}
\Sigma_\bubble = \Sigma_\infty\setminus C
\end{equation*}
 attached at a single node, so that \autoref{Hyp_BubbleAtOnePoint} is satisfied.
Denote by $\check C$ the nodal curve corresponding to $C$ and let $B$ be the base-locus of the dualizing sheaf of $\check C$.
Define $B_0 \subset B$ and 
\begin{equation*}
  C_\basefree \coloneq C \setminus B_0 \qandq \Sigma_\clubsuit \coloneq \Sigma_0 \setminus C_\basefree
\end{equation*}
as at the beginning of \autoref{Sec_LeadingOrderTermOfObstructionGhostComponents}.
Observe that 
\begin{equation*}
  \Sigma_\clubsuit = B_0 \amalg \Sigma_\bubble
\end{equation*}
so that $\Sigma_\infty$ decomposes into
\begin{equation*}
  \Sigma_\infty=C_\basefree \amalg \underbrace{B_0 \amalg \Sigma_\bubble}_{\Sigma_\clubsuit}.
\end{equation*}

Write $\tau_k = (\tau_{k,1}, \tau_{k,2})$ with $\tau_{k,1}$ and $\tau_{k,2}$ denoting the smoothing parameters corresponding to the sets of nodes $S_1$ and $S_2$ defined in \autoref{Eq_SetsOfNodes}.
Since $B_0$ is a tree of spheres, the partial smoothing $\Sigma_{\sigma_k,\tau_{k,1},0}$ contains a component biholomorphic to $\Sigma_\bubble$, as discussed in  \autoref{Ex_SphericalBubble}.  
Let 
\begin{equation*}
  b_k \co \Sigma_\bubble \to X
\end{equation*}
 be the restriction of $u_{\sigma,\tau_{k,1},0;\kappa_k}$ to this component. 
After reparametrizing $b_k$ by biholomorphisms of $\Sigma_\bubble$ we can guarantee that $b_k$ converges to $u_0 |_{\Sigma_\bubble}$ in the $C^\infty$ topology.  
Let $x$ be any node in $S_{C_\bullet}^\rext$; the point $\nu_k(x)$ under the above biholomorphisms it is mapped to a point $x_k \in \Sigma_\bubble$ and the sequence $(x_k)$ satisfies 
\begin{equation*}
 \lim_{k\to\infty} x_k = \nu_\infty(n).
\end{equation*}

Let $\check C_\basefree$ be the nodal curve corresponding to $C_\basefree$. 
By the construction of $C_\basefree$, for every node $n\in C_\basefree$ such that $\nu(n) \in \Sigma_\clubsuit$, there is exists a holomorphic section $\zeta \in H^0(\check C_\basefree,\omega_{\check C_\basefree})$ with $\zeta(n) \neq 0$.
Since $\ob(\sigma_k,\tau_k;J_k;\kappa_k) = 0$,
it follows from \autoref{Lem_KuranishiModelGhostComponent}  that
\begin{equation*}
  \abs{\rd_{x_k} b_k}
  \leq
  \epsilon_{k;\ghost}^{1/2}.
\end{equation*}
(Observe that reparametrizing $b_k$ by biholomorphisms of $\Sigma_\bubble$ does not affect the estimate in \autoref{Lem_KuranishiModelGhostComponent}, which is independent of the choice of a Riemannian metric in the given conformal class.) 
Passing to the limit  $k \to \infty$ yields that $\rd_{\nu_\infty(n)}u_0 = 0$.
\hfill\textsc{}
$\xrightswishingghost{\qquad\qquad}$



\section{Calabi--Yau classes in symplectic \texorpdfstring{$6$}{6}--manifolds}

\subsection{Proof of \autoref{Thm_CY3PrimitiveGenericLimitHasSmoothDomain}}

Denote by $C_1,\ldots,C_I$ the connected components of $\Sigma_\infty$ on which $u_\infty$ is non-constant and set $u_\infty^i \coloneq u_\infty|_{C_i}$ and $A_i \coloneq (u_\infty^i)_*[C_i]$.
By the index formula \autoref{Eq_IndexFormula},
\begin{equation*}
  \sum_{i=1}^I \ind(u_\infty^i)
  =
  \sum_{i=1}^I 2\Inner{c_1(X,\omega),A_i}
  =
  2\Inner{c_1(X,\omega),A}
  =
  0
\end{equation*}
Since $J_\infty \in \sJ_\emb(X,\omega)$,
for every $i=1,\ldots,I$,
$\ind(u_\infty^i) \geq 0$ and thus $\ind(u_\infty^i) = 0$.
Consequently, the images of the simple maps underlying $u_\infty^i$ and $u_\infty^j$ either agree or are disjoint.
However,
\begin{equation*}
  \im u_\infty = \bigcup_{i=1}^I \im u_\infty^i
\end{equation*}
is connected.
Therefore and since $A$ is primitive,
$I = 1$ and $u_\infty^1$ is simple and, hence, an embedding because $J \in \sJ_\emb(X,\omega)$.
Given the above, it follows from \autoref{Thm_LimitConstraints} that $(\Sigma_\infty,j_\infty,\nu_\infty)$ is smooth.

\subsection{Proof of \autoref{Thm_CY3PrimitiveGV}}

The proof of \autoref{Thm_CY3PrimitiveGV}\autoref{Thm_CY3PrimitiveGV_Defined} is completely standard and straightforward.
Nevertheless, let us spell it out.
Let $J \in \sJ_\emb^\star(X,\omega)$.
By \autoref{Prop_UnobstructedModuliSpaces} and \autoref{Thm_CY3PrimitiveGenericLimitHasSmoothDomain},
$\sM_{A,g}^\star(X,J)$ is a compact oriented zero-dimensional manifold;
that is: a finite set of points with signs.
The signed count
\begin{equation*}
  \#\sM_{A,g}^\star(X,J)
\end{equation*}
is independent of the choice of $J$.
To see this,
let $J_0,J_1 \in \sJ_\emb^\star(X,\omega)$ and $(J_t)_{t\in[0,1]} \in \sJ_\emb^\star(X,\omega;J_0,J_1)$.
By \autoref{Prop_UnobstructedModuliSpaces} and \autoref{Thm_CY3PrimitiveGenericLimitHasSmoothDomain},
$\sM_{A,g}^\star\paren*{X,(J_t)_{t\in[0,1]}}$ is a compact oriented manifold with boundary 
\begin{equation*}
  \sM_{A,g}^\star(X,J_1) \amalg - \sM_{A,g}^\star(X,J_0).
\end{equation*}
Therefore,
\begin{equation*}
  \#\sM_{A,g}^\star(X,J_1) = \#\sM_{A,g}^\star(X,J_0).
\end{equation*}

\autoref{Thm_CY3PrimitiveGV}\autoref{Thm_CY3PrimitiveGV_Finiteness} follows from \cite[Theorem 1.6]{Doan2018a}.
Indeed,
the latter asserts that for every $J \in \sJ_\star(X,\omega)$ the set
\begin{equation*}
  \coprod_{g=0}^\infty \sM_{A,g}^\star(X,J)
\end{equation*}
is a finite set.
Therefore,
there exists a $g_0\in \N_0$ such that for every $g\geq g_0$ the moduli space $\sM_{A,g}^\star(X,A;J)$ is empty;
in particular,
$n_{A,g}(X,\omega) = 0$ for $g \geq g_0$.
\qed



\section{Fano classes in symplectic \texorpdfstring{$6$}{6}--manifolds}

The proofs in this section make use of definitions and results from \autoref{Sub_JHolomorphicMapsWithConstraints}.

\subsection{Proof of \autoref{Thm_Fano3GenericLimitHasSmoothDomain}}

Denote by $\tilde C_1,\ldots,\tilde C_I$ the connected components of $\Sigma_\infty$ on which $u_\infty$ is non-constant,
set $\tilde u_\infty^i \coloneq u_\infty|_{\tilde C_i}$,
denote by $u_\infty^i\co C_i \to X$ the simple map underlying $\tilde u_\infty^i$,
let $d_i \in \N$ be the degree of the covering map relating $\tilde u_\infty^i$ and $u_\infty^i$,
set $A_i \coloneq (u_\infty^i)_*[C_i]$, and
set $g_i \coloneq g(C_i)$.
Since $(u_k)_{k\in\N}$ Gromov converges to $u_\infty$,
\begin{enumerate}
\item $u_k(\Sigma_k)$ converges to $u_\infty(\Sigma_\infty)$ in the Hausdorff topology, and
\item $\sum_{i=1}^I d_i A_i = A$.
\end{enumerate}
These are the only two consequences of Gromov convergence that will be used in the following argument.

Denote by $I_0$ the subset of those $i \in \set{1,\ldots,I}$ with $\Inner{c_1(X,\omega),A_i} = 0$ and set $I_+ \coloneq \set{1,\ldots,I}\setminus I_0$.
Without loss of generality all of the pseudo-cycles $f_\lambda$ have $\codim(f_\lambda) \geq 4$.
For every $i \in I_+$ denote by $\Lambda_i$ the subset of those $\lambda \in \set{1,\ldots,\Lambda}$ such that 
\begin{equation}
  \label{Eq_ComponentIntersectingPseudocycle}
  \im u_\infty^i \cap \overline{\im f_\lambda} \neq \emptyset.
\end{equation}
Since $J_\infty \in \sJ_\emb(X,\omega;f_1,\ldots,f_\Lambda)$,
for every $i \in I_0$ and $\lambda \in \set{1,\ldots,\Lambda}$ we have 
$\im u_\infty^i \cap \overline{\im f_\lambda} = \emptyset$.
Therefore and since $u_k(\Sigma_k)$ converges to $u_\infty(\Sigma_\infty)$ in the Hausdorff topology,
for every $\lambda \in \set{1,\ldots,\Lambda}$ there exists at least one $i \in I_+$ such that \autoref{Eq_ComponentIntersectingPseudocycle} holds.
For every $i \in \set{1,\ldots,I}$ and $\lambda \in \Lambda_i$ set
\begin{equation*}
  f_\lambda^i
  \coloneq
  \begin{cases}
    f_\lambda & \text{if } \im u_\infty^i \cap \im f_\lambda \neq \emptyset \\
    f_\lambda^\del & \text{otherwise}
  \end{cases}
\end{equation*}
with $f_\lambda^\del$ as in in \autoref{Sub_JHolomorphicMapsWithConstraints};
in particular:
$\codim f_\lambda \leq \codim f_\lambda^i$ with equality if and only if  $\im u_\infty^i \cap \im f_\lambda \neq \emptyset$.
By definition,
$u_\infty^i$ represents an element of $\sM_{A,g}^\star(X,J;(f_\lambda^i)_{\lambda\in \Lambda_i})$.
Therefore and since $J \in \sJ_\emb(X,\omega,f_1,\ldots,f_\Lambda)$,
\begin{equation*}
  2\inner{c_1(X,\omega)}{A_i} - \sum_{\lambda\in \Lambda_i}\paren*{\codim(f_\lambda^i)-2}
  \geq 0.
\end{equation*}
On the one hand,
multiplying by $d_i$ and summing yields
\begin{align*}
  \sum_{i\in I_+}
  \sum_{\lambda\in \Lambda_i}
  \paren*{\codim(f_\lambda^i)-2}
  &\leq
    \sum_{i\in I_+}
    \sum_{\lambda\in \Lambda_i}
    d_i\paren*{\codim(f_\lambda^i)-2} \\
  &\leq
    \sum_{i=1}^I
    2\Inner{c_1(X,\omega),d_iA_i}  \\
  &=
    2\Inner{c_1(X,\omega),A} \\
  &=
    \sum_{\lambda=1}^\Lambda \paren*{\codim(f_\lambda)-2}.
\end{align*}
On the other hand,
by the preceding discussion,
the reverse inequality also holds.
Therefore, equality holds and this implies that
\begin{enumerate}
\item
  $d_i = 1$ for every $i \in I_+$,
\item
  $2\Inner{c_1(X,\omega),A_i} = \sum_{\lambda\in \Lambda_i}\paren[\big]{\codim(f_\lambda^i)-2}$,
\item
  $f_\lambda^i = f_\lambda$,
  and
\item
  the subsets $\Lambda_i$ are non-empty and pairwise disjoint.
\end{enumerate}
This implies that for every $i \in I_+$ the map $\tilde u_\infty^i$ agrees with $u_\infty^i$ and thus is simple;
moreover, this map has index zero (in the sense of \autoref{Eq_IndexFormula}) and its image intersects $f_\lambda$ for every $\lambda \in \Lambda_i$.
Furthermore, every $f_\lambda$ intersects the image of precisely one map $u_\infty^i$ with $i \in I_+$.
Therefore, the images of the maps $u_\infty^i$ with $i \in I_+$ are pairwise disjoint.

Since $2\Inner{c_1(X,\omega),A} > 0$,
$I_+$ is non-empty.
For $i \in I_+$ and $j \in I_0$ the images of $u_\infty^i$ and $u_\infty^j$ must also be disjoint, because otherwise they would have to agree---%
contradicting $A_i \neq A_j$.
However,
\begin{equation*}
  \im u_\infty = \bigcup_{i=1}^I \im u_\infty^i
\end{equation*}
is connected.
Therefore, if $I_0 \neq \emptyset$,
then there is are $i \in I_0$ and $j \in I_+$ such that the images of $u_\infty^i$ and $u_\infty^j$ intersect.
The preceding discussion shows this to be impossible; hence: $I_0 = \emptyset$.
Similarly, if $I_+$ were to contain more than one element, then there are $i,j \in I_+$ with such that the images of $u_\infty^i$ and $u_\infty^j$ intersect---which is impossible.
Therefore,
$I = 1$ and
$\tilde u_\infty^1 = u_\infty^1$ is an embedding.

Given the above, it follows from \autoref{Thm_LimitConstraints} that $(\Sigma_\infty,j_\infty,\nu_\infty)$ is smooth and $\im u_\infty \cap \im f_\lambda \neq \emptyset$ every $\lambda=1,\ldots,\Lambda$.
\qed

\subsection{Proof of \autoref{Thm_Fano3GV}}

Given Gromov compactness, \autoref{Thm_Fano3GenericLimitHasSmoothDomain}, and \autoref{Prop_UnobstructedModuliSpacesPseudoCycles}, the proof that $n_{A,g}(X,\omega;\gamma_1,\ldots,\gamma_\Lambda)$ is well-defined and independent of the choice of $J$ is identical to that of \autoref{Thm_CY3PrimitiveGV} up to changes in notation.

To prove that $n_{A,g}(X,\omega;\gamma_1,\ldots,\gamma_\Lambda)$ is independent of the choice of pseudo-cycle representatives, suppose that  $f_1^0$ and $f_1^1$ are two representatives of $\PD[ \gamma_1]$ such that $f_1^i, f_2, \ldots, f_\Lambda$ are in general position for $i=0,1$.
Let $F \co W \to X$ be a pseudo-cycle cobordism between $f_1^0$ and $f_1^1$ such that $F, f_2, \ldots, f_\Lambda$ are in general position. 
Let $J$ be an element of the set $\sJ_\emb^\star(X,\omega;F,f_2,\ldots, f_\Lambda)$ defined in \autoref{Def_JEmbPseudoCycleCobordism}, which is  residual by \autoref{Prop_TransversalityPseudoCycleCobordism}. 
It follows that  $\sM_{A,g}^\star(X,J;f_1^0,\ldots,f_\Lambda)$ and $\sM_{A,g}^\star(X,J;f_1^1,\ldots,f_\Lambda)$ are finite sets of points with orientations and $\sM_{A,g}^\star(X,J; F,f_2,\ldots,f_\Lambda)$  is an oriented $1$--dimensional cobordism between them.
This cobordism is compact by Gromov compactness and the argument used in the proof in  \autoref{Thm_Fano3GenericLimitHasSmoothDomain}.
Thus, 
\begin{equation*}
  \# \sM_{A,g}^\star(X,J;f_1^0,\ldots,f_\Lambda) = \#\sM_{A,g}^\star(X,J;f_1^1,\ldots,f_\Lambda).
\end{equation*} 

The fact that $n_{A,g}(X,\omega;\gamma_1,\ldots,\gamma_\Lambda) = 0$ for $g \gg 1$ is a consequence of the following analog of \cite[Theorem 1.6]{Doan2018a} for Fano classes.

\begin{theorem}
  \label{Thm_FinitelyManyCurvesFanoCase}
  Let $(X,\omega)$ be a compact symplectic $6$--manifold, let $f_1,\ldots, f_\Lambda$ be a collection of even-dimensional pseudo-cycles in general position, and let $A \in H_2(X,\Z)$ be such that 
  \begin{equation*}
    2\inner{c_1(X,\omega)}{ A} = \sum_{\lambda=1}^\Lambda(\codim f_\lambda - 2) > 0.
  \end{equation*}
  For every $J \in \sJ(X,\omega;f_1,\ldots,f_\Lambda)$ there are only finitely many simple $J$--holomorphic maps representing $A$ and passing through $\im f_\lambda$ for every $\lambda=1,\ldots,\Lambda$. 
\end{theorem}

\begin{proof}
  The proof is a minor variation of the proof of \cite[Theorem 1.6]{Doan2018a}. 
  Suppose, by contradiction, that there are infinitely many distinct $J$--holomorphic curves $C_k$ representing $A$ and passing through $\im f_\lambda$ for all $\lambda=1,\ldots,\Lambda$. 
  Here, by a $J$--holomorphic curve we mean the image of a simple $J$--holomorphic map. 
  Considering $C_k$ as $J$--holomorphic cycles, we can pass to a subsequence which converges geometrically to a $J$--holomorphic cycle $C_\infty = \sum_{i=1}^I d_i C_\infty^i$, see \cite[Definition 4.1, Definition 4.2, Lemma 1.9]{Doan2018a}. 
  Here $d_i > 0$ are integers and each $C_\infty^i$ is a $J$--holomorphic curve. 
  Geometric convergence implies that
  \begin{equation*}
    \sum_{i=1}^I d_i [C_\infty^i] = [C_\infty] = A
  \end{equation*}
  and that $(C_k)_{k\in\N}$ converges to $C_\infty$ in the Hausdorff topology.
  Since these were the only two conditions needed for the argument in the proof   \autoref{Thm_Fano3GenericLimitHasSmoothDomain},
  the same argument shows that:
  \begin{enumerate}
  \item $d_i = 1$ for every $i \in I$, 
  \item $C_\infty$ has only one connected component, 
  \item $C_\infty$ intersects every $\im f_\lambda$, and consequently
  \item $C_\infty$ is embedded and unobstructed by the condition $J\in\sJ_\emb^\star(X,\omega;f_1,\ldots,f_\Lambda)$. 
  \end{enumerate}
  
  We will now adapt the rescaling argument from the proof of \cite[Proposition 5.1]{Doan2018a}---originally due to Taubes in the $4$--dimensional setting \cite{Taubes1996}---to the present situation.
  Let $N \to C_\infty$ be the normal bundle of $C_\infty$ in $X$. 
  Identify a neighborhood of $C_\infty$ with a neighborhood of the zero section in $N$ using the exponential map.
  For sufficiently large $k$,
  $C_k$ is contained in that neighborhood and by abuse of notation we will consider $C_k$ as an $\exp^*J$--holomorphic curve in $N$ and $f_\lambda$ as maps to $N$. 
  
  Since the $C_k$ are distinct,
  $C_k \neq C_\infty$.
  For $\epsilon > 0$ denote by $\sigma_\epsilon \co N \to N$ the map which rescales the fibers by $\epsilon$. 
  Let $(\epsilon_k)_{k\in\N}$ be the sequence of positive numbers such that the rescaled sequence
  \begin{equation*}
    \tilde C_k \coloneq (\sigma_{\epsilon_k})^{-1}(C_k)
  \end{equation*}
  satisfies
  \begin{equation*}
    d_H(\tilde C_k, C_\infty) = 1,
  \end{equation*}
  where $d_H$ is the Hausdorff distance.
  The sequence $(\epsilon_k)_{k\in\N}$ converges to zero.
  The curves $\tilde C_k$ are $J_k$--holomorphic where $J_k \coloneq \sigma_{\epsilon_k}^*\exp^*J$.
  The sequence of rescaled almost complex structures $(J_k)_{k\in\N}$ converges to an almost complex structure $J_\infty$ which is tamed by a symplectic form \cite[Proposition 3.10]{Doan2018a}. 
  In the same way as in the proof of \cite[Proposition 5.1]{Doan2018a} we conclude that the sequence $(\tilde C_k)_{k\in\N}$ converges geometrically to a $J_\infty$--holomorphic cycle whose support is a union of $J_\infty$--holomorphic curves $\tilde C_\infty \subset N$ satisfying 
  \begin{equation*}
    d_H(\tilde C_\infty, C_\infty) = 1.
  \end{equation*}
  Since $[\tilde C_k] = [C_\infty] = A$ for all $k$, and the bundle projection $\pi\co N \to C_\infty$ is $J_\infty$--holomorphic, $\pi$ induces an isomorphism $\tilde C_\infty \iso C_\infty$. 
  Let $\iota \co C_\infty \to X$ be the inclusion map and denote by $\fd_\iota$   the deformation operator corresponding to $\iota$, as in \autoref{Def_JHolomorphicLinearization}.
  By \cite[Proposition 3.12]{Doan2018a}, $\tilde C_\infty$ is the graph of a non-zero section $\xi \in \Gamma(C_\infty, N) \subset \Gamma(C_\infty, \iota^*TX)$ satisfying $\fd_\iota\xi = 0$.
  In \autoref{Prop_PropertiesOfXi} below we show that there is an algebraic constraint for the values of $\xi$ at the points of intersection of $C_\infty$ with each pseudocycle.
  
  For every $\lambda = 1, \ldots, \Lambda$,
  denote by $V_\lambda$ the domain of $f_\lambda$, and
  let $z_{\lambda,k} \in C_k$ and $x_{\lambda,k} \in V_\lambda$ be such that $z_{\lambda,k} = f_\lambda(x_{\lambda,k})$. 
  After passing to a subsequence, we may assume that 
  \begin{equation*}
    \lim_{k\to \infty} z_{\lambda,k} = z_{\lambda} \in C_\infty \qandq   \lim_{k\to \infty} x_{\lambda,k} = x_{\lambda} \in V_\lambda,
  \end{equation*}
  and $z_{\lambda} = f_\lambda(x_{\lambda})$.
  
  \begin{prop}
    \label{Prop_PropertiesOfXi}
    For $\lambda = 1,\ldots,\Lambda$ there exist $v_\lambda \in T_{z_\lambda} C_\infty$ and $w_\lambda \in  T_{x_\lambda} V_\lambda$ such that
    \begin{equation}
      \label{Eq_InfinitesimalIntersectionCondition}
      \xi(z_\lambda) +\rd_{z_\lambda}\iota \cdot v_\lambda = \rd_{x_\lambda}f_\lambda \cdot w_\lambda.
    \end{equation}
  \end{prop}
  
  Equation \autoref{Eq_InfinitesimalIntersectionCondition} can be understood as the limit as $k\to \infty$ of the condition that $C_k$ intersects each of the $\im f_\lambda$.
  The proof is deferred to the end of this section.
  We will now show that \autoref{Prop_PropertiesOfXi} implies \autoref{Thm_FinitelyManyCurvesFanoCase}.
  Let $g$ be the genus of $C_\infty$, so that the embedding $\iota \co C_\infty \to X$ corresponds to an element in $\sM_{A,g,\Lambda}^\star(X,J)$.
  Since $J \in \sJ_\emb^\star(X,\omega, f_1,\ldots,f_\Lambda)$,
  \begin{enumerate}
  \item \label{It_DerivativeOfEvaluation} the derivative of $\ev_\Lambda \co \sM_{A,g,\Lambda}^\star(X,J) \to X^\Lambda$ at $[\iota,z_1,\ldots, z_\Lambda]$, and
  \item \label{It_DerivativeOfPseudocycle} the derivative of $\prod_{\lambda=1}^\Lambda f_\lambda \co \prod_{\lambda=1}^\Lambda V_\lambda \to X^\Lambda$ at $\prod_{\lambda=1}^\Lambda x_\lambda$
  \end{enumerate}
  are transverse to each other. 
  Since
  \begin{equation*}
    \dim \sM_{A,g,\Lambda}^\star(X,J)  + \sum_{\lambda=1}^\Lambda \dim V_\lambda = \Lambda \dim X,
  \end{equation*}
  the images of these two maps intersect trivially. 
  Since $\xi \neq 0$, this contradicts the existence of $v_\lambda$ and $w_\lambda$ satisfying \autoref{Eq_InfinitesimalIntersectionCondition}. 
  The contradiction shows that the sequence $(C_k)$ cannot exist.
\end{proof}

\begin{proof}[Proof of \autoref{Prop_PropertiesOfXi}]
  Set $\tilde z_{\lambda,k} \coloneq \sigma_{\epsilon_k}^{-1}(z_{\lambda,k})$.
  After possibly passing to a further subsequence, 
  \begin{equation}
    \label{Eq_LimitOfTildeZ}
    \lim_{k\to\infty}\tilde z_{\lambda,k} = \xi(z_{\lambda}). 
  \end{equation}
  Let $\pr_N \rd_{x_\lambda} f_\lambda \co T_{x_\lambda}V_\lambda \to N_{z_\lambda}$ be the projection of the derivative of $f_\lambda$ at $x_\lambda$ on $N_{z_\lambda} \subset T_{z_\lambda}X$. 
  We will show that for every $\lambda$ there exists $w_\lambda \in T_{x_\lambda}V_\lambda$ such that $\lim_{k\to\infty} \tilde z_{\lambda,k} = \pr_N \rd_{x_\lambda}f_\lambda \cdot w_\lambda$. 
  
  The fact that the images of the maps \autoref{It_DerivativeOfEvaluation} and  \autoref{It_DerivativeOfPseudocycle} introduced above intersect trivially implies that $\pr_N\rd_{x_\lambda} f_\lambda$ is injective for every $\lambda$.
  Indeed, otherwise there would exist $v \in T_{z_\lambda} C_\infty$ and $w\in T_{x_\lambda}V_\lambda$ for some $\lambda$ such that
  \begin{equation*}
    \rd_{z_\lambda}\iota\cdot v = \rd_{x_\lambda} f_\lambda \cdot w,
  \end{equation*}
  violating the above transversality condition.
  Fix a trivialization of $N$ in a neighborhood of $z_\lambda$ and a chart centered at $x_\lambda$ in $V_\lambda$.
  Denoting by $\pr_N$ the projection on the fiber $N_{z_\lambda}$ in the given trivialization, the Taylor expansion gives us
  \begin{equation*}
    \pr_N z_{\lambda,k}
    =
    \pr_N f_\lambda(x_{\lambda,k}) = \pr_N \rd_{x_\lambda}f_\lambda (x_{\lambda,k}-x_\lambda) + O\paren*{\abs{x_{\lambda,k}-x_\lambda}^2}.
  \end{equation*}
  Since $\pr_N\rd_{x_\lambda}f_\lambda$ is injective, there is a constant $c>0$ such that
  \begin{equation*}
    \abs{x_{\lambda,k}-x_\lambda} \leq c \abs{\pr_N z_{\lambda,k}} \leq c \epsilon_k.
  \end{equation*}
  Thus, after passing to a subsequence, we may assume that the sequence $\epsilon_k^{-1}(x_{\lambda,k}-x_\lambda)$ converges to a limit $w_\lambda \in T_{x_\lambda}V_\lambda$.
  By construction,
  \begin{equation*}
    \lim_{k\to \infty} \tilde z_{\lambda,k} =   \lim_{k\to \infty} \pr_N \tilde z_{\lambda,k} = \pr_N \rd_{x_\lambda}f_\lambda \cdot w_\lambda.
  \end{equation*}
  Comparing this with \autoref{Eq_LimitOfTildeZ}, we see that for every $\lambda$ there exists $v_\lambda \in T_{z_\lambda}C_\infty$ such that \autoref{Eq_InfinitesimalIntersectionCondition} holds. 
\end{proof}

\appendix

\section{Transversality for evaluation maps}
\label{Sec_TransversalityEvaluationMaps}

Throughout this section, $(X,\omega)$ is a symplectic manifold of dimension $\dim X\geq 6$ and $\sJ(X,\omega)$ denotes the space of almost complex structures on $X$ compatible with $\omega$. 

\begin{definition}
  \label{Def_GeneralizedDiagonal}
  Let $\Lambda \in \N$.
  Given a partition into nonempty pairwise disjoint subsets
  \begin{equation*}
    \set{1,2,\ldots,\Lambda} = I_1 \sqcup \ldots \sqcup I_k  \quad\text{with } k < \Lambda,
  \end{equation*} 
  the \defined{generalized diagonal} associated with the partition is the submanifold $\Delta \subset X^\Lambda$ consisting of the points $(x_1,\ldots, x_\Lambda)$ such that for every pair of indices $\alpha,\beta \in I_i$ we have $x_\alpha = x_\beta$. 
\end{definition}

Generalized diagonals are partially ordered by inclusion and each point of $X^\Lambda$ which belongs to a generalized diagonal belongs to a unique one which is minimal with respect to the partial order. 

\begin{prop}
  \label{Prop_TransversalityEvaluationMaps}
  Let $V$ be a manifold and let $f \co V \to X^\Lambda$ be a map which is transverse to every generalized diagonal.
  Denote by $\sJ^\star(X,\omega;f)$  the set of all $J\in\sJ^\star(X,\omega)$ such that
  \begin{enumerate}
    \item \label{It_Unobstructed}
     every simple $J$--holomorphic map is unobstructed, and
    \item \label{It_EvaluationMap} for every $A\in H_2(X,\Z)$ and $g\in\N_0$, the evaluation map from the $\Lambda$--pointed moduli space (cf. \autoref{Def_ModuliOfPointedMaps})
    \begin{equation*} 
    \ev \co \sM_{A,g,\Lambda}^\star(X,J)\to X^\Lambda
    \end{equation*}
    is transverse to $f$.
  \end{enumerate}
  The set $\sJ^\star(X,\omega;f)$  is residual in $J\in\sJ^\star(X,\omega)$.
\end{prop}

\begin{proof}
    The proof that condition  \autoref{It_Unobstructed} is generic is a standard application of the Sard--Smale theorem \cite[Theorem 1.2]{Oh2009}, \cite[Proposition A.4]{Ionel2018}, \cite[Sections 3.2 and 6.3]{McDuff2012}.
  Below we outline this proof and adapt it to show that condition \autoref{It_EvaluationMap} is generic.
  
  Let  $(\Sigma,j_0)$ be a closed Riemann surface of genus $g$, and let $A \in H_2(X,\Z)$. 
  Denote by $W^{1,p}_{\inj}(\Sigma,X;A)$ the subset of $W^{1,p}(\Sigma,X)$  consisting of functions $u \co \Sigma \to X$ which represent $A$ and are \defined{somewhere injective} in the sense that there exist $z_0 \in\Sigma$ and $\delta >0$ such that for all $z\in\Sigma$
  \begin{equation*}
    \mathrm{dist}_X( u(z_0), u(z)) \geq \delta \mathrm{dist}_\Sigma(z_0,z).
  \end{equation*}
  A $J$--holomorphic map is somewhere injective if and only if it is simple \cite[Proposition 2.5.1]{McDuff2012}.
  Given a slice $\sS \subset \sJ(\Sigma)$ for the action of $\Diff_0(\Sigma)$ on $\sJ(\Sigma)$ passing through $j_0$, set 
  \begin{equation*}
    \sX =W^{1,p}_{\inj}(\Sigma,X;A) \times \sS
  \end{equation*}
  and let $\sE \to \sX$ be a Banach vector bundle whose fiber over $(u,j)$ is the space $L^p\Omega^{0,1}(\Sigma,u^*TX)$ defined using the complex structure $j$.
  
  Let $s \co \sJ(X,\omega) \times \sX \to \sE$ be a section given by $s(J, u,j) = \delbar_J(u,j)$. 
  The following hold; see, for example, \cite[Section 3.2]{McDuff2012}:
    \begin{itemize}
    \item  this section is Fredholm,
    \item it is transverse to the zero section, therefore
    \item $s^{-1}(0)$ is a submanifold of $\sX$; in particular, it is a Banach manifold, and
    \item the universal moduli space $\sM_{A,g}^\star(X,\omega)$ can be covered by a countable number of submanifolds of the form $s^{-1}(0)$, for different choices of $(\Sigma,j_0)$.
  \end{itemize}
  The projection $\pi \co s^{-1}(0) \to \sJ(X,\omega)$ is a Fredholm map of index $\vdim \sM_{A,g}^\star(X,J)$; in fact, the kernel and cokernel of $\rd\pi_{u,j}$ are isomorphic to the kernel and cokernel of $\rd_{u,j}\delbar_J$, and therefore finite-dimensional.
  It follows from the implicit function theorem that if $J$ is a regular value of this map, the preimage
  \begin{equation*}
     \pi^{-1}(J) = \sM_{A,g}^\star(X,J)
  \end{equation*}
  is a manifold of dimension $\vdim \sM_{A,g}^\star(X,J)$ and every map in $\sM_{A,g}^\star(X,J)$ is unobstructed. 
  Since $\pi \co s^{-1}(0) \to \sJ(X,\omega)$ is a Fredholm map between separable Banach manifolds, the Sard--Smale theorem implies that the set of regular values of $\pi$ is residual in $\sJ(X,\omega)$. 
  This shows that condition \autoref{It_Unobstructed} holds for a generic $J$. 

  Using a similar argument, we will show for a generic $J$, the evaluation map
  \begin{equation*} 
    \ev \co \sM_{A,g,\Lambda}^\star(X,J)\to X^\Lambda
  \end{equation*}
  is transverse to $f$.
  With the notation introduced above, consider the Fredholm map 
  \begin{gather*}
    S \co \sJ(X,\omega) \times \sX \times \Sigma^\Lambda \to \sE \times X^\Lambda \\
    S(J, (u,j), z_1, \ldots, z_\Lambda) = (s(J,u,j), u(z_1), \ldots, u(z_\Lambda)).
  \end{gather*} 
  We will show that $S$ is transverse to the map
  \begin{equation}
    \label{Eq_ZeroSectionTimesF}
    \text{zero section}\times f \co \sX\times V \to \sE\times X^\Lambda.
  \end{equation} 
  Since $s$ is transverse to the zero section $\sX \to \sE$, it suffices to show that whenever $(J,(u,j),z_1,\ldots, z_\Lambda)$ and $x\in V$ satisfy
  \begin{equation*}
    s(J,u,j) = 0 \qandq (u(z_1),\ldots, u(z_\Lambda)) = f(x),
  \end{equation*}
  then
  \begin{equation*}
   \im \rd S + \im \rd_x f = T_{f(x)} X^\Lambda = \bigoplus_{i=1}^\Lambda T_{u(z_i)} X.
  \end{equation*}
  Here $\rd S$ denotes the projection on $T_{f(x)} X^\Lambda$ of the derivative of $S$ at  $(J,(u,j),z_1,\ldots, z_\Lambda)$ and $\rd_x f$ is the derivative of $f$ at $x$. 
  The variation of $S$ in the direction of a vector field
  \begin{equation*}
    \xi \in W^{1,p}\Gamma(\Sigma,u^*TX) 
  \end{equation*}
  is
  \begin{equation}
    \label{Eq_VariationOfS}
    \rd S(\xi) = (\xi(z_1),\ldots, \xi(z_\Lambda)).
  \end{equation} 
  If $(u(z_1),\ldots, u(z_\Lambda))$ does not lie on any generalized diagonal in $X^\Lambda$, we can find $\xi$ with any prescribed values at $z_1, \ldots, z_\Lambda$, and 
  \begin{equation*}
    \im \rd S = T_{f(x)} X^\Lambda.
  \end{equation*}
  Suppose, on the other hand, that $(u(z_1),\ldots, u(z_\Lambda))$ belongs to a generalized diagonal $\Delta\subset X^\Lambda$, and let $\Delta$ be the minimal such diagonal.
  In that case, \autoref{Eq_VariationOfS} implies that 
  \begin{equation*}
    \im \rd S = T_{f(x)} \Delta.
  \end{equation*}
  Since $f$ is transverse to $\Delta$, we have
  \begin{equation*}
   \im \rd S + \im \rd_x f = T_{f(x)} \Delta + \im \rd_x f  = T_{f(x)} X^\Lambda  
  \end{equation*}
  as desired. 
  This shows that $S$ is transverse to the map \autoref{Eq_ZeroSectionTimesF}.
  It follows from the Sard--Smale theorem that the set of $J$ such that $S(J,\cdot)$ is transverse to \autoref{Eq_ZeroSectionTimesF} is residual in $\sJ(X,\omega)$.
  This completes the proof that condition \autoref{It_EvaluationMap} is generic.
\end{proof}



\section{Pseudo-Cycles}
\label{Sec_Pseudocycles}

Given a collection of homology classes, we are interested in counting $J$--holomorphic maps passing through cycles representing these classes.
Since not every homology class is represented by a map from a manifold, it is convenient to use the language of pseudo-cycles.
We briefly review the theory of pseudo-cycles below; for details, see \cites[Section 6.5]{McDuff2012}{Schwarz1999,Kahn2001,Zinger2008}. 

\begin{definition}
  \label{Def_Pseudocycle}
  ~
  \begin{enumerate}
  \item A subset of a smooth manifold $X$ is said to have \defined{dimension at most $k$} if it is contained in the image of a smooth map from a smooth $k$--dimensional manifold.
  \item A \defined{$k$--pseudo-cycle} is a smooth map $f\co V \to X$ from an oriented $k$--dimensional manifold $V$ such that the closure $\overline{f(V)}$ is compact and the \defined{boundary of $f$}, defined by
    \begin{equation*}
      \bdry(f) \coloneq \bigcap_{K \subset V \text{ compact}} \overline{f(V-K)},
    \end{equation*}
    has dimension at most $k-2$.  
    We will use notation
    \begin{equation*}
      \codim(f) \coloneq \dim(X) - \dim(V).
    \end{equation*} 
  \item Two $k$--pseudo-cycles $f_i \co V_i \to X$, for $i=0,1$, are \defined{cobordant} if there exists a smooth, oriented $(k+1)$--dimensional manifold with boundary $W$ and a smooth map $F \co W \to X$ such that $\overline{F(W)}$ is compact, $\bdry(F)$ has dimension at most $k-1$, and 
    \begin{equation*}
      \partial W = V_1 \amalg -V_0 \qandq F|_{V_1} = f_1, \quad F|_{V_0} = f_0.
    \end{equation*}
  \item Denote by $H_k^{\rm{pseudo}}(X)$ the set of equivalence classes of $k$--pseudo-cycles up to cobordism. 
    The disjoint union operation endows $H_k^{\rm{pseudo}}(X)$ with the structure of an abelian group.
    \qedhere
  \end{enumerate}
\end{definition}  

A smooth map $g \co X \to Y$ between two smooth manifolds induces a group homomorphism $g_* \co H_*^{\rm{pseudo}}(X) \to H_*^{\rm{pseudo}}(Y)$ by composing pseudo-cycles with $g$. 
Thus, $H_*^{\rm{pseudo}}(\cdot)$ is a functor from the category of smooth manifolds to the category of $\Z$--graded abelian groups.

\begin{theorem}[{\cite{Schwarz1999,Kahn2001,Zinger2008}}]
  There exists a natural isomorphism $H_*^{\rm{pseudo}}(\cdot) \iso H_*(\cdot,\Z)$ as functors from the category of smooth manifolds to the category of $\Z$--graded abelian groups.
\end{theorem}

In what follows we will use this isomorphism to identify these two homology theories and represent any class in $H_*(X,\Z)$ by a pseudo-cycle. 

\begin{definition}
  \label{Def_PseudocycleTransversality}
  Let $M$ be a smooth manifold and let $g \co M \to X$ be a smooth map. 
  We say that a $k$--pseudo-cycle $f \co V \to X$ is \defined{transverse as a pseudo-cycle to $g$}%
  \footnote{\citet[Definition 6.5.10]{McDuff2012} use the term \defined{weakly transverse}, which we prefer to avoid, regarding that this notion of transversality is stronger than the transversality of $f$ and $g$ as smooth maps in the usual sense.}
  if 
  \begin{enumerate}
  \item
    there exists a smooth manifold $V^\del$ of dimension $\dim V^\del \leq \dim V -2$ and a smooth map $f^\del \co V^\del \to X$ such that $\bdry(f) \subset \im f^\del$,
    and
  \item
    $f$ and $f^\del$ are transverse to $g$ as smooth maps from manifolds. 
  \end{enumerate}
  If $W$ is a manifold with boundary $\partial W$, we require additionally that $f$ is transverse as a pseudo-cycle to $g|_{\partial W} \co \partial W \to X$. 
  
  Similarly, if $M$ is a manifold without boundary and $F \co W \to X$ is a cobordism between two pseudo-cycles $f_0$ and $f_1$, we say that $F$ is \defined{transverse as a pseudo-cycle cobordism to $g$} if 
  \begin{enumerate}
  \item there exists a smooth manifold with boundary $W^\del$ of dimension $\dim W^\del \leq \dim W - 2$ and a smooth map $F^\del \co W^\del \to X$ such that $\bdry(F) \subset \im F^\del$ and $\bdry(f_i) \subset \im F^\del|_{\del W^\del}$  for $i=0,1$, 
  \item $F$ and $F^\del$ are transverse to $g$ as smooth maps from manifolds with boundary.
    \qedhere
  \end{enumerate}
\end{definition}

Note that  if $f \co V \to X$ be a $k$--pseudo-cycle and $g \co W \to X$ is an $\ell$--pseudo-cycle,
then $f \times g \co V \times W \to X^2$ is a $(k+\ell)$--pseudo-cycle.

\begin{definition}
  \label{Def_PseudocyclesInGeneralPosition}
  Let $(f_\lambda \co V_\lambda \to X)$ be a collection of pseudo-cycles indexed by a finite set $I$.
  We say that $(f_\lambda)_{\lambda \in I}$ are \defined{in general position} if the pseudo-cycle
  \begin{equation*}
    \prod_{\lambda\in I} f_\lambda \co \prod_{\lambda\in I} V_\lambda \to X^{\abs{I}}
  \end{equation*}
  is transverse as a pseudo-cycle to all generalized diagonals in $X^{\abs{I}}$; see \autoref{Def_GeneralizedDiagonal} for the definition of a generalized diagonal. 
  This is equivalent to the following condition: for every subset $S\subset I$, the pseudo-cycle $\prod_{\lambda\in S} f_{\lambda}$ is transverse as a pseudo-cycle to the diagonal $X \hookrightarrow X^{\abs{S}}$.   
  
  Similarly, if one of $f_\lambda$ is a cobordism between two pseudo-cycles, then so is $\prod_{\lambda \in I} f_\lambda$ and we require that it is transverse to all generalized diagonals as a pseudo-cycle cobordism. 
\end{definition}

\begin{prop}
  Given a finite collection of pseudo-cycles $(f_\lambda \co V_\lambda \to X)_{\lambda \in I}$, the set
  \begin{equation*}
    \set*{
      (\phi_\lambda)_{\lambda\in I} \in\Diff(X)^{\abs{I}}
      :
      (\phi_\lambda \circ f_\lambda)_{\lambda\in I} \text{ are in general position}
    }
  \end{equation*}
  is residual in $\Diff(X)^{\abs{I}}$. 
\end{prop}

\begin{proof}
  The proof is similar to that of \cite[Lemma 6.5.5]{McDuff2012}.
  Let us work with the group $\Diff_k(X)$ of $C^k$ diffeomorphism for any integer $k\geq 1$; the corresponding statement for $\Diff(X)$ follows then using standard arguments \cite[pp. 52--54, Remark 3.2.7]{McDuff2012}. 
  A countable intersection of residual sets is residual;
  therefore, without loss of generality, consider the case $S=I$ in \autoref{Def_PseudocyclesInGeneralPosition}.
  Define the map $\sF \co \Diff_k(X)^{\abs{I}} \times \prod_{\lambda\in I} V_\lambda \to X^{\abs{I}}$ by
  \begin{equation*}
    \sF\paren*{(\phi_\lambda)_{\lambda\in I}, (x_\lambda)_{\lambda\in I}}
    \coloneq
    \paren*{\phi_\lambda\circ f_\lambda(x_\lambda)}_{\lambda\in I}.
  \end{equation*}  
  Let $\Delta \subset X^{\abs{I}}$ be the diagonal. 
  If we show that $\sF$ is transverse to $\Delta$,
  then it follows from the Sard--Smale theorem that for all $(\phi_\lambda)_{\lambda\in I}$ from a residual subset of $\Diff_k(X)$ the maps $\prod\phi_\lambda\circ f_\lambda$ is transverse to $\Delta$. 
  (The same argument can be applied to $f_\lambda^\del$ to conclude transversality as pseudo-cycles.)
  In fact, the derivative of $\sF$ is surjective at every point $\bx = ((\phi_\lambda)_{\lambda\in I}, (x_\lambda)_{\lambda\in I})$.
  Without loss of generality suppose that $\phi_\lambda = \id$ for all $\lambda\in I$.
  Let $\Vect_k(X)$ denote the space of $C^k$ vector fields on $X$.
  Given
  \begin{equation*}
    \bxi = (\xi_\lambda)_{\lambda \in I} \in \prod_{\lambda \in I} T_{\id}\Diff_k(X) = \prod_{\lambda\in I} \Vect_k(X),
  \end{equation*} 
  we have
  \begin{equation*}
    \rd_\bx\sF (\bxi) = ( \xi_\lambda( f_\lambda(x_\lambda)) )_{\lambda\in I} \in  \prod_{\lambda \in I} T_{f_\lambda(x_\lambda)} X.
  \end{equation*}
  Since for every $p \in X$ the evaluation map $\Vect(X) \to T_p X$ is surjective, the map $\rd_\bx \sF$ is surjective, which finishes the proof.
\end{proof}



\section{Proof of \texorpdfstring{$n_{A,g} = \BPS_{A,g}$}{n=BPS}}
\label{Sec_N=BPS}

In this section, we outline Zinger's proof that for a primitive Calabi--Yau class
\begin{equation*}
 n_{A,g}(X,\omega) = \BPS_{A,g}(X,\omega),
\end{equation*}
 where $\BPS_{A,g}(X,\omega)$ is the Gopakumar--Vafa invariant defined in terms of the Gromov--Witten invariants via \autoref{Eq_CY3GopakumarVafa}.
We use the same notation as in the proof of \autoref{Thm_CY3PrimitiveGV}.

Given $J\in\sJ^\star_\emb(X,\omega)$, every stable $J$--holomorphic map of arithmetic genus $h$ factors through a $J$--holomorphic embedding from a smooth domain of genus $g\leq h$.
In other words, every element of $\overline\sM_{A,h}(X,J)$ is of the form $[u\circ\varphi]$ for some $[u] \in \sM_{A,g}^\star(X,J)$ with $g\leq h$, and $[\varphi]\in\overline\sM_{[\Sigma],h}(\Sigma,j)$. 
Here $(\Sigma,j)$ is the domain of $u$. 
Denote by $(\tilde\Sigma,\tilde\nu,\tilde j)$ the domain of $\varphi$. 
Given such $J$--holomorphic maps, let $N$ be the normal bundle of $u(\Sigma)$, and let
\begin{equation*}
  \fd_{u}^N \co W^{1,p}\Gamma(\Sigma, u^*N) \to L^p\Omega^{0,1}(\Sigma,u^*N)
\end{equation*}
be the restriction of the operator $\fd_{u} = \fd_{u, j;J}$ to the subbundle $u^*N \subset u^*TX$ followed by the projection on $\tilde u^*N$. 
Similarly,
we define 
\begin{equation*}
  \fd_{\tilde u}^N \co W^{1,p}\Gamma(\tilde\Sigma, \tilde\nu;\tilde u^*N) \to L^p\Omega^{0,1}(\tilde\Sigma,\tilde u^*N).
\end{equation*}
The spaces $\coker\fd_{\tilde u}^N$, as $\varphi$ varies, play an important role in computing the contribution of maps factoring through $u$ to the Gromov--Witten invariant of $(X,\omega)$. 
In this case, there is a simple description of these spaces.

First, we will see that $\ker\fd_u^N =\set{0}$ and $\coker\fd_u^N = \set{0}$. 
Indeed, the Hermitian metric on $u^*TX$ induced from $X$ gives us a splitting $u^*TX = T\Sigma\oplus N_u$, with respect to which
\begin{equation*}
  \fd_u =
  \begin{pmatrix}
    \delbar_{T\Sigma} & \ast \\
    0 & \fd_u^N
  \end{pmatrix}; 
\end{equation*}
see, for example, \cite[Appendix A]{Doan2018}. 
Since $u$ is unobstructed, i.e. $\coker\fd_u = \set{0}$, and $\ind(u) = 0$, we have $\ker \fd_u^N = \set{0}$ and $\coker\fd_u^N = \set{0}$.

Second, since $\varphi \co (\tilde\Sigma,\tilde\nu,\tilde j) \to (\Sigma,j)$ has degree one, $(\tilde\Sigma,\tilde\nu,\tilde j)$ has a unique irreducible component which is mapped by $\varphi$ biholomorphically to $(\Sigma,j)$, and $\varphi$ is constant on the other components.
In particular, $\tilde u^*N$ is trivial over these components.
It follows that $\ker\fd_{\tilde u}^N \iso \set{0}$ and $\coker\fd_{\tilde u}^N$ is the direct sum of the corresponding spaces for the standard $\delbar$--operator with values in the trivial bundle $\tilde u^*N$ over the components which are mapped to a point by $\varphi$. 

In this situation, the following is a special instance of \cite[Theorem 1.2]{Zinger2011}.

\begin{prop}
  ~
  \begin{enumerate}
     \item
     The family of vector spaces $\coker\fd_{u\circ\varphi}^N$, as $[\tilde\Sigma,\tilde\nu,\tilde j,\varphi]\in\overline\sM_{[\Sigma],h}(\Sigma,j)$ varies, forms an oriented orbibundle $\fO_h(\Sigma,j,u) \to \overline\sM_{[\Sigma],h}(\Sigma,j)$, called the \defined{obstruction bundle}.
     \item 
     Denoting by $[\overline\sM_{[\Sigma],h}(\Sigma,j)]^\text{vir}$ the virtual fundamental class and by  
     $e(\fO_h(\Sigma,j,u))$ the Euler class of the obstruction bundle, we have 
     \begin{equation*}
       \GW_{A,h}(X,\omega) = \sum_{g=0}^h \sum_{[u]\in\sM_{A,g}^\star(X,J)}\sign(\Sigma,j,u)\inner{e(\fO_h(\Sigma,j,u))}{[\overline\sM_{[\Sigma],h}(\Sigma,j)]^\text{vir} }. 
     \end{equation*}
   \end{enumerate}
\end{prop}
\citet[Section 2.3]{Pandharipande1999} proved that
for $g \coloneq g(\Sigma)$,
\begin{equation*}
  \sum_{h=g}^\infty \inner{e(\fO_h(\Sigma,j,u))}{[\overline\sM_{[\Sigma],h}(\Sigma,j)]^\text{vir} }t^{2h-2}
  =
  t^{2g-2} \paren*{\frac{\sin(t/2)}{t/2}}^{2g-2}
\end{equation*}
Therefore, after changing the order of summation $\sum_{h=0}^\infty\sum_{g=0}^h = \sum_{g=0}^\infty\sum_{h=g}^\infty$, we obtain
\begin{equation*}
  \sum_{h = 0}^\infty \GW_{A,h}(X,\omega) t^{2h-2}
  =
  \sum_{g=0}^\infty n_{A,g}(X,\omega)   t^{2g-2} \paren*{\frac{\sin(t/2)}{t/2}}^{2g-2}.
\end{equation*}
Since the numbers $\BPS_{A,g}(X,\omega)$ are uniquely determined by the Gopakumar--Vafa formula \autoref{Eq_CY3GopakumarVafa} \cite[Section 2]{Bryan2001},
$n_{A,g}(X,\omega) = \BPS_{A,g}(X,\omega)$. 


 
\printreferences


@misc{Huang2015,
	arxiv = {1501.04891},
	author = {Huang, Min-xin and Katz, Sheldon and Klemm, Albrecht},
	doi = {10.1007/JHEP10(2015)125},
	journal = {Journal of High Energy Physics},
	number = {125},
	publisher = {Springer},
	title = {{Topological String on elliptic CY 3-folds and the ring of Jacobi forms}},
	volume = {2015},
	year = {2015}}

@misc{Knapp2021,
	arxiv = {2107.05647},
	author = {Knapp, Johanna and Scheidegger, Emanuel and Schimannek, Thorsten},
	title = {On genus one fibered Calabi--Yau threefolds with $5$--sections},
	year = {2021}}

@article{Hosono2001,
	author = {Hosono, Shinobu and Saito, Masa-Hiko and Takahashi, Atsushi},
	doi = {10.1155/S107379280100040X},
	issn = {1073-7928},
	journal = {International Mathematics Research Notices},
	mr = {1849482},
	number = {15},
	pages = {783--816},
	title = {{Relative Lefschetz action and BPS state counting}},
	year = {2001},
	zbl = {1060.14017}}

@article{Maulik2018,
	author = {Maulik, Davesh and Toda, Yukinobu},
	doi = {10.1007/s00222-018-0800-6},
	issn = {0020-9910},
	journal = {Inventiones Mathematicae},
	mr = {3842061},
	number = {3},
	pages = {1017--1097},
	title = {{Gopakumar--Vafa invariants via vanishing cycles}},
	volume = {213},
	year = {2018},
	zbl = {1400.14141}}

@misc{Ekholm2019,
	arxiv = {1901.08027},
	author = {Ekholm, Tobias and Shende, Vivek},
	title = {{Skeins on Branes}},
	year = 2019}

@article{Zinger2008,
	author = {Zinger, Aleksey},
	journal = {{Transactions of the American Mathematical Society}},
	number = {5},
	pages = {2741--2765},
	title = {{Pseudocycles and integral homology}},
	volume = {360},
	year = {2008},
	zbl = {1213.57031}}

@misc{Kahn2001,
	arxiv = {math/0111223},
	author = {Kahn, Peter J.},
	title = {{Pseudohomology and homology}},
	year = 2001}

@misc{Kiem2012,
	arxiv = {1212.6444},
	author = {Kiem, Young-Hoon and Li, Jun},
	title = {{Categorification of Donaldson--Thomas invariants via perverse sheaves}},
	year = {2012}}

@book{Arbarello1985,
	author = {Arbarello, Enrico and Cornalba, Maurizio and Griffiths, Phillip A. and Harris, Joseph},
	doi = {10.1007/978-1-4757-5323-3},
	isbn = {0-387-90997-4},
	mr = 770932,
	mrclass = {14Hxx (14-02)},
	mrreviewer = {Werner Kleinert},
	number = 267,
	publisher = {Springer},
	series = {Grundlehren der Mathematischen Wissenschaften},
	title = {Geometry of algebraic curves},
	vol = 1,
	year = 1985,
	zbl = {0559.14017}}

@book{Arbarello2011,
	author = {Arbarello, Enrico and Cornalba, Maurizio and Griffiths, Phillip A. and Harris, Joseph},
	doi = {10.1007/978-3-540-69392-5},
	isbn = {978-3-540-42688-2},
	mr = {2807457},
	mrclass = {14H10 (32G15)},
	mrreviewer = {E. Looijenga},
	note = {With a contribution by Joseph Daniel Harris},
	number = {268},
	pages = {xxx+963},
	publisher = {Springer},
	series = {Grundlehren der Mathematischen Wissenschaften},
	title = {Geometry of algebraic curves. {V}olume {II}},
	year = {2011},
	zbl = {1235.14002}}

@article{Bryan2001,
	arxiv = {math/0009025},
	author = {Bryan, Jim A. and Pandharipande, Rahul},
	doi = {gt.2001.5.287},
	journal = {Geometry and Topology},
	mr = 1825668,
	mrreviewer = {Vicente Mu\~noz},
	pages = {287--318},
	title = {{BPS states of curves in Calabi--Yau $3$--folds}},
	volume = 5,
	year = 2001,
	zbl = {1063.14068}}

@article{Castelnuovo1889,
	author = {Castelnuovo, Guido},
	journal = {{Atti della Accademia delle Scienze di Torino}},
	pages = {346--373},
	publisher = {Accademia delle Scienze di Torino, Torino},
	title = {{Ricerche di geometria sulle curve algebriche}},
	volume = {24},
	year = {1889},
	zbl = {21.0669.01}}

@article{Deligne1969,
	author = {Deligne, Pierre and Mumford, David},
	journal = {Publications math{\'e}matiques de l'IH{\'E}S},
	pages = {75--109},
	title = {{The irreducibility of the space of curves of a given genus}},
	volume = {36},
	year = {1969},
	zbl = {0181.48803}}

@article{Doan2018,
	arxiv = {2006.01352},
	author = {Doan, Aleksander and Walpuski, Thomas},
	doi = {https://doi.org/10.1017/fms.2022.104},
	journal = {Forum of Mathematics Sigma},
	publisher = {Cambridge University Press},
	title = {{Equivariant Brill--Noether theory for elliptic operators and super-rigidity of $J$--holomorphic maps}},
	url = {https://walpu.ski/Research/EquivariantBrillNoetherSuperRigidity.pdf},
	volume = {11},
	year = 2023}

@article{Doan2021,
	arxiv = {2103.08221},
	author = {Doan, Aleksander and Ionel, Eleny-Nicoleta and Walpuski, Thomas},
	fjournal = {Annals of Mathematics. Second Series},
	journal = {Ann. of Math. (2)},
	pubstate = {to appear},
	title = {{The Gopakumar--Vafa finiteness conjecture}},
	url = {https://walpu.ski/Research/GopakumarVafaFiniteness.pdf},
	year = 2025}

@article{Doan2018a,
	arxiv = {1809.04731},
	author = {Doan, Aleksander and Walpuski, Thomas},
	doi = {10.1016/j.aim.2020.107550},
	journal = {Advances in Mathematics},
	mr = {4199270},
	title = {{Castelnuovo's bound and rigidity in almost complex geometry}},
	url = {https://walpu.ski/Research/CastelnuovoRigidity.pdf},
	volume = {379},
	year = {2021},
	zbl = {07300460}}

@misc{Gopakumar1998,
	arxiv = {hep-th/9809187},
	author = {Gopakumar, Rajesh and Vafa, Cumrun},
	title = {{M--theory and topological strings--I.}},
	year = {1998}}

@misc{Gopakumar1998a,
	arxiv = {hep-th/9812127},
	author = {Gopakumar, Rajesh and Vafa, Cumrun},
	title = {{M--theory and topological strings--II}},
	year = {1998}}

@misc{Gournay2009,
	author = {Gournay, Antoine},
	title = {Complex surfaces and interpolation on pseudo-holomorphic cylinders},
	url = {http://www.math.tu-dresden.de/~gournay/interpol.pdf},
	year = {2009}}

@article{Gromov1985,
	author = {Gromov, Mikhael},
	doi = {10.1007/BF01388806},
	issn = {0020-9910},
	journal = {Inventiones Mathematicae},
	mr = {809718},
	mrclass = {53C15 (32F25 53C57 57R15)},
	mrreviewer = {Yakov Eliashberg},
	number = {2},
	pages = {307--347},
	title = {Pseudo holomorphic curves in symplectic manifolds},
	volume = {82},
	year = {1985},
	zbl = {0592.53025}}

@book{Hartshorne2010,
	author = {Hartshorne, R.},
	doi = {10.1007/978-1-4419-1596-2},
	isbn = {978-1-4419-1595-5},
	mr = {2583634},
	mrclass = {14D15 (13D10 14B07 14B12)},
	mrreviewer = {Arvid Siqveland},
	pages = {viii+234},
	publisher = {Springer},
	series = {Graduate Texts in Mathematics},
	title = {Deformation theory},
	volume = {257},
	year = {2010},
	zbl = {1186.14004}}

@book{Hummel1997,
	author = {Hummel, Christoph},
	doi = {10.1007/978-3-0348-8952-0},
	isbn = {3-7643-5735-5},
	mr = {1451624},
	mrclass = {58D15 (58D27 58F05)},
	mrreviewer = {Bernd Siebert},
	publisher = {Birkh\"auser Verlag, Basel},
	series = {Progress in Mathematics},
	title = {Gromov's compactness theorem for pseudo-holomorphic curves},
	volume = {151},
	year = {1997},
	zbl = {0870.53002}}

@article{Ionel1998,
	author = {Ionel, Eleny-Nicoleta},
	doi = {10.1215/S0012-7094-98-09414-5},
	journal = {Duke Mathematical Journal},
	mr = {1638587},
	number = {2},
	pages = {279--324},
	title = {{Genus $1$ enumerative invariants in $\mathbf{P}^n$ with fixed $j$ invariant}},
	volume = {94},
	year = {1998},
	zbl = {0974.14038}}

@article{Ionel2018,
	arxiv = {1306.1516},
	author = {Ionel, Eleny-Nicoleta and Parker, Thomas H.},
	doi = {10.4007/annals.2018.187.1.1},
	issn = {0003-486X},
	journal = {Annals of Mathematics},
	mr = {3739228},
	mrclass = {53D45},
	number = {1},
	pages = {1--64},
	title = {{The Gopakumar--Vafa formula for symplectic manifolds}},
	volume = {187},
	year = {2018},
	zbl = {06841536}}

@book{McDuff2012,
	author = {McDuff, Dusa and Salamon, Dietmar A.},
	edition = {Second},
	mr = {2954391},
	publisher = {American Mathematical Society},
	series = {American Mathematical Society Colloquium Publications},
	title = {{$J$--holomorphic curves and symplectic topology}},
	volume = {52},
	year = {2012},
	zbl = {1272.53002}}

@article{Oh2009,
	arxiv = {0805.3581},
	author = {Oh, Yong-Guen and Zhu, Ke},
	doi = {10.4310/AJM.2009.v13.n3.a4},
	issn = {1093-6106},
	journal = {Asian Journal of Mathematics},
	mr = {2570442},
	mrclass = {53D45 (32Q65)},
	mrreviewer = {Timothy Perutz},
	number = {3},
	pages = {323--340},
	title = {Embedding property of {$J$}--holomorphic curves in {C}alabi-{Y}au manifolds for generic {$J$}},
	volume = {13},
	year = {2009},
	zbl = {1193.53180}}

@misc{Pandharipande1995,
	arxiv = {alg-geom/9505023},
	author = {Pandharipande, Rahul},
	title = {{A note on elliptic plane curves with fixed $j$--invariant}},
	year = {1995}}

@article{Pandharipande1999,
	arxiv = {math/9811140},
	author = {Pandharipande, Rahul},
	doi = {10.1007/s002200050766},
	journal = {Communications in Mathematical Physics},
	mr = {1729095},
	number = {2},
	pages = {489--506},
	title = {{Hodge integrals and degenerate contributions}},
	volume = {208},
	year = {1999},
	zbl = {0953.14036}}

@article{Pandharipande2009,
	arxiv = {0707.2348},
	author = {Pandharipande, Rahul and Thomas, Richard P.},
	doi = {10.1007/s00222-009-0203-9},
	issn = {0020-9910},
	journal = {Inventiones Mathematicae},
	mr = {2545686},
	mrclass = {14N35 (14F05 18E30)},
	mrreviewer = {Yunfeng Jiang},
	number = {2},
	pages = {407--447},
	title = {Curve counting via stable pairs in the derived category},
	volume = {178},
	year = {2009},
	zbl = {1204.14026}}

@article{Pandharipande2010,
	arxiv = {0711.3899},
	author = {Pandharipande, Rahul and Thomas, Richard P.},
	doi = {10.1090/S0894-0347-09-00646-8},
	issn = {0894-0347},
	journal = {Journal of the American Mathematical Society},
	mr = {2552254},
	mrclass = {14N35 (53D45)},
	mrreviewer = {Hsian-Hua Tseng},
	number = {1},
	pages = {267--297},
	title = {Stable pairs and {BPS} invariants},
	volume = {23},
	year = {2010},
	zbl = {1250.14035}}

@article{Pardon2016,
	arxiv = {1309.2370},
	author = {Pardon, John},
	doi = {10.2140/gt.2016.20.779},
	journal = {Geometry and Topology},
	mr = {3493097},
	number = {2},
	pages = {779--1034},
	title = {{An algebraic approach to virtual fundamental cycles on moduli spaces of pseudo-holomorphic curves}},
	volume = {20},
	year = {2016},
	zbl = {1342.53109}}

@article{Parker1993,
	author = {Parker, Thomas~H. and Wolfson, J.~G.},
	doi = {10.1007/BF02921330},
	issn = {1050-6926},
	journal = {Journal of Geometric Analysis},
	mr = 1197017,
	mrclass = {58D15 (58D27 58E12 58E20)},
	number = 1,
	pages = {63--98},
	title = {Pseudo-holomorphic maps and bubble trees},
	volume = 3,
	year = 1993,
	zbl = {0759.53023}}

@article{Robbin2006,
	author = {Robbin, Joel W. and Salamon, Dietmar A.},
	journal = {Journal of the European Mathematical Society},
	number = {4},
	pages = {611--699},
	title = {A construction of the Deligne-Mumford orbifold},
	volume = {8},
	year = {2006},
	zbl = {1105.32011}}

@incollection{Schwarz1999,
	author = {Schwarz,  Matthias},
	booktitle = {{Geometry and topology in dynamics. AMS special session on topology in dynamics, Winston-Salem, NC, USA, October 9--10, 1998 and the AMS-AWM special session on geometry in dynamics, San Antonio, TX, USA, January 13--16, 1999}},
	publisher = {Providence, RI: American Mathematical Society},
	title = {{Equivalences for Morse homology}},
	year = {1999},
	zbl = {0951.55009}}

@article{Taubes1996,
	author = {Taubes, Clifford H.},
	doi = {10.1090/S0894-0347-96-00211-1},
	issn = {0894-0347},
	journal = {Journal of the American Mathematical Society},
	mr = {MR1362874},
	mrclass = {57R57 (53C15 58D10 58D27 58G30)},
	mrreviewer = {Dietmar A. Salamon},
	number = {3},
	pages = {845--918},
	title = {{${\rm SW}\Rightarrow{\rm Gr}$}: from the {S}eiberg--{W}itten equations to pseudo-holomorphic curves},
	volume = {9},
	year = {1996},
	zbl = {0867.53025}}

@thesis{Niu2016,
	author = {Niu, Jingchen},
	title = {Refined Convergence for Genus-Two Pseudo-Holomorphic Maps},
	year = 2016,
	url = {https://hdl.handle.net/11401/76380},
        school = {State University of New York at Stony Brook}
}

@article{Ye1994,
	author = {Ye, Rugang},
	doi = {10.2307/2154647},
	issn = {0002-9947},
	journal = {Transactions of the American Mathematical Society},
	mr = {1176088},
	mrclass = {58E12 (53C23)},
	mrreviewer = {Nikolai K. Smolentsev},
	number = {2},
	pages = {671--694},
	title = {Gromov's compactness theorem for pseudo holomorphic curves},
	volume = {342},
	year = {1994},
	zbl = {0810.53024}}

@article{Zinger2009,
	arxiv = {math/0406103},
	author = {Zinger, Aleksey},
	doi = {10.2140/gt.2009.13.2427},
	issn = {1465-3060},
	journal = {Geometry and Topology},
	mr = {2529940},
	mrclass = {53D45 (32Q65)},
	mrreviewer = {Hsian-Hua Tseng},
	number = {5},
	pages = {2427--2522},
	title = {A sharp compactness theorem for genus-one pseudo-holomorphic maps},
	volume = {13},
	year = {2009},
	zbl = {1174.14012}}

@article{Zinger2011,
	arxiv = {0807.0805},
	author = {Zinger, Aleksey},
	doi = {10.1016/j.aim.2011.05.021},
	journal = {Advances in Mathematics},
	mr = {2822239},
	number = {1},
	pages = {535--574},
	title = {{A comparison theorem for Gromov--Witten invariants in the symplectic category}},
	volume = {228},
	year = {2011},
	zbl = {1225.14046}}
\end{document}
